\documentclass[12pt]{amsart}
\usepackage[colorlinks,citecolor=blue,linkcolor=blue, urlcolor=blue]{hyperref}
\usepackage{graphicx, geometry}

\newtheorem{theorem}{Theorem}[section]
\newtheorem{lemma}[theorem]{Lemma}
\newtheorem{corollary}[theorem]{Corollary}
\newtheorem{proposition}[theorem]{Proposition}
\theoremstyle{definition}
\newtheorem*{definition*}{Definition}

\newtheorem*{remark*}{Remark}
\newtheorem*{conjecture*}{Conjecture}
\newtheorem{example}[theorem]{Example}
\numberwithin{equation}{section}

\allowdisplaybreaks[4]

\begin{document}

\title[Hypercontractivity of $\mathbb{Z}_3$ and biased Bernoulli]{The optimal hypercontractive constants for $\mathbb{Z}_3$ and biased Bernoulli random variables}

\date{}

\author%[authorlabel1]
{Jie Cao}
\address%[authorlabel1]
{Jie CAO: School of Mathematical Sciences, Shenzhen University, Shenzhen 518060, Guangdong, China}
\email{mathcj@foxmail.com}

\author%[authorlabel1]
{Shilei Fan}
\address%[authorlabel1]
{Shilei FAN: School of Mathematics and Statistics, and Key Lab NAA--MOE, Central China Normal University, Wuhan 430079, China}
\email{slfan@ccnu.edu.cn}

\author%[authorlabel1]
{Yong Han}
\address%[authorlabel1]
{Yong HAN: School of Mathematical Sciences, Shenzhen University, Shenzhen 518060, Guangdong, China}
\email{hanyong@szu.edu.cn}

\author%[authorlabel1]
{Yanqi Qiu}
\address%[authorlabel1]
{Yanqi QIU: School of Fundamental Physics and Mathematical Sciences, HIAS, University of Chinese Academy of Sciences, Hangzhou 310024, China}
\email{yanqi.qiu@hotmail.com; yanqiqiu@ucas.ac.cn}

\author{Zipeng Wang}
\address{Zipeng WANG: College of Mathematics and Statistics, Chongqing University, Chongqing
401331, China}
\email{zipengwang2012@gmail.com; zipengwang@cqu.edu.cn}

%\thanks{JC is supported by Postdoctoral Daily Funding (No.000141020162),SF is supported by NSFC (No.12331004 and No.12231013), YH is supported by NSFC (No.12131016), ZW is supported by NSFC (No.12471116) and FRFCU (No.2025CDJ-IAIS YB004).}

\begin{abstract}
We resolve a folklore problem of determining the optimal hypercontractive constants $r_{p,q}(\mathbb{Z}_3)$ for the cyclic group \( \mathbb{Z}_3\) for all \( 1 < p < q < \infty \). More precisely, we have 
\[
r_{p,q}(\mathbb{Z}_3) = \frac{(1 + 2x)(1 - y)}{(1 + 2y)(1 - x)},
\]
where \((x,y)\) is the {\it unique} solution in the open unit square $(0,1)\times (0,1)$ to the system of equations
\begin{align*}
\left\{
\begin{aligned}
&\frac{1}{1+2x}\Big(\frac{1+2x^p}{3}\Big)^{\frac{1}{p}}=\frac{1}{1+2y}\Big(\frac{1+2y^q}{3}\Big)^{\frac{1}{q}},\\
&\frac{(1-x)(1-x^{p-1})}{1+2x^p}=\frac{(1-y)(1-y^{q-1})}{1+2y^q}.
\end{aligned}
\right.
\end{align*}
Consequently, for rational 
$p, q\in \mathbb{Q}$, the constants 
$r_{p,q}(\mathbb{Z}_3)$ are algebraic numbers which generally admit {\it no  radical expressions}, since their often rather complicated minimal polynomials may have non‑solvable Galois groups.

Our formalism relies on a key observation: the existence of {\it nontrivial critical extremizers}. This approach can also be adapted to resolve a long-standing open problem—determining all optimal $(p,q)$-hypercontractive constants for biased Bernoulli random variables, which are closely related to noise operators.  Several noteworthy phenomena emerge from numerical simulations: the monotonicity of the hypercontractive constants in the parameters, and the appearance of intriguing limit shapes. These phenomena merit further investigation.
\end{abstract}


\maketitle

\setcounter{tocdepth}{1}
\tableofcontents

\section{Introduction}

\subsection{Main results}
Given an integer $n \ge 2$, let $\mathbb{Z}_n = \mathbb{Z}/n\mathbb{Z}$ denote the finite cyclic group equipped with the Haar probability measure $\mathfrak{m}$ and the graph distance $d(k)= \min \{k, n-k\}$ for $0\le k\le n-1$.  Throughout the whole paper, let $\chi\in \widehat{\mathbb{Z}_n}$ be the unique character determined by $\chi(1)=e^{i 2\pi/n}$. 
 For any $r \in [0,1]$ and any $p>1$,  define the dilation operator on $L^p(\mathbb{Z}_n)=L^p(\mathbb{Z}_n,\mathfrak{m})$ as
$$
T_{r}: \sum_{k=0}^{n-1} a_k\chi^k \mapsto \sum_{k=0}^{n-1} a_k r^{d(k)} \chi^k, \quad \text{with $a_k\in\mathbb{C}$ for all $0\le k\le n-1$.}
$$

Let $1 < p < q < \infty$. The hypercontractivity problem for $\mathbb{Z}_n$ aims to determine the optimal (largest) constant $r_{p, q}(\mathbb{Z}_n) \in [0,1]$ such that the  following hypercontractive inequality holds: 
\begin{align}\label{intro-hyper-pq}
\left\|T_{r} f\right\|_{q} \leq \|f\|_{p}  \quad  \text{for all $0\le r \leq r_{p, q}(\mathbb{Z}_n)$ and all $f \in L^p(\mathbb{Z}_n)$.}
\end{align}

In the case of $\mathbb{Z}_2$, the hypercontractive inequality is referred to as the {\it two-point inequality} in literature and it was proved by  Bonami \cite{Bonami}-Nelson \cite{Nelson-73}-Beckner \cite{Beckner}-Gross \cite{Groos} that 
$$
r_{p,q}(\mathbb{Z}_2) = \sqrt{(p-1)/(q-1)}. 
$$
 And, it is proved that this hypercontractive constant is shared by many other finite cyclic groups:
\begin{itemize}
  \item $n=4$: $r_{p,q}(\mathbb{Z}_4)=r_{p,q}(\mathbb{Z}_2)$ by Beckner, Janson, and Jerison \cite{BJJ83} in 1983.
  \item $n=5$: $r_{2,2q}(\mathbb{Z}_5)=r_{2,2q}(\mathbb{Z}_2)$ with $q \in \mathbb{Z}_{+}$ by Andersson \cite{Andersson} in 2002.
  \item $n \geq 6$: (either $n$ is odd with $n\geq q$ or $n$ is even), $r_{2,2q}(\mathbb{Z}_n)= r_{2, 2q}(\mathbb{Z}_2)$  with $q \in \mathbb{Z}_{+}$ by Junge, Palazuelos, Parcet and Perrin \cite{JPPP} in 2017.
  \item \(n \in \{3 \cdot 2^k, 2^k\}\) with \(k \geq 1\), \(r_{p,q}(\mathbb{Z}_n) = r_{p,q}(\mathbb{Z}_2)\) by Yao \cite{yao2024optimal} in 2025.
\end{itemize}

% \cite{BJJ83} obtained the hypercontractive constant $r_{p,q}(\mathbb{Z}_4)=r_{p,q}(\mathbb{Z}_2)$.  For \( n = 5 \), $p=2$ and \( q \in 2\mathbb{Z}_+ \), Andersson \cite{And02} obtained the the hypercontractive constant $r_{p,2q}(\mathbb{Z}_5)$ coincides $r_{p,2q}(\mathbb{Z}_2)$ again.
%By developing novel combinatorial techniques, Junge, Palazuelos, Parcet and Perrin investigated the hypercontractive inequalities for $\mathbb{Z}_n$ with $n \geq 6$. They showed that for even and odd integers $n \geq q$, the optimal constant satisfies $r_{2,2q}(\mathbb{Z}_n)= r_{2, 2q}(\mathbb{Z}_2)$ (see \cite[Theorem A.3]{JPPP}). In the most recent work, Yao established the hypercontractive constant for the hypercontractive inequality on a large class of cyclic groups: specifically, it was shown that \(r_{p,q}(\mathbb{Z}_n) = r_{p,q}(\mathbb{Z}_2)\) for all \(n \in \{3 \cdot 2^k, 2^k\}\) with \(k \geq 1\).

%In \cite{Weissler}, Weissler obtained the same hypercontractive constant for the hypercontractive inequality for the  group $\mathbb{Z}$.

 However, the
hypercontractivity  of $\mathbb{Z}_3$ is drastically different and partial results  were obtained by Andersson \cite{Andersson}, Diaconis and Saloff-Coste \cite{DS96}.    
The argument of Wolff \cite{Wolff} implies
\begin{align}\label{2q-Z3}
r_{2,q}(\mathbb{Z}_3) = \sqrt{2(4^{1/q} - 1)/(4 - 4^{1/q})}.
\end{align}
Wolff's result is based on Lata{\l}a and Oleszkiewicz \cite{LaO00,Ole03}, where the condition $p=2$ is crucially used.   Beyond their own interest in analysis, hypercontractive inequalities on $\mathbb{Z}_3$ play roles in the study of the traffic light problem and subsets of independent sets in product graphs, see \cite{ADFS-04}, \cite{DFR-08} and \cite{Fri-23}. More details and applications of hypercontractive inequalities on discrete space can be found in \cite{And99, CLS08, Donnell, JPPE, IN-22, KLLM-24, SVZ24}. One can also consult \cite{Jan-83,Wei-90,Gro-99, Gro-02, EGS-18, FI-21, Kui-22, Mel-23} for the studies and applications of the hypercontractive and related log-Sobolev inequalities on complex manifolds.

By developing powerful combinatorial techniques, Junge, Palazuelos, Parcet, and Perrin \cite{JPPP} obtained optimal hypercontractive constant for free groups, some Coxeter groups, and many finite cyclic groups. However,  in  \cite[after Theorem A2 in page \bf{x}]{JPPP}, the authors note:

\begin{quote}
{\it \(\mathbb{Z}_3\) is the simplest group with a 3-loop in its Cayley graph, and the hypercontractive constant is not even conjectured.}
\end{quote}

\noindent This  inspires us to move away from seeking an explicit closed-form expression for \( r_{p,q}(\mathbb{Z}_3) \), which appears unlikely to exist, and instead to analyze its implicit functional structure. Through this shift in perspective, we  resolving the problem by determining \( r_{p,q}(\mathbb{Z}_3) \) for all $1<p<q<\infty$. 

Our work relies crucially on the following concept of {\it nontrivial critical extremizers}.  

\begin{definition*}
A {\it non-constant} function $f$  on $\mathbb{Z}_n$ (if it exists) for which the hypercontractive inequality \eqref{intro-hyper-pq} becomes an equality at the critical  $r=r_{p, q}(\mathbb{Z}_n)$  is called a {\it nontrivial $(p,q)$-critical extremizer}  for $\mathbb{Z}_n$ or simply a {\it nontrivial critical extremizer}.    For any given $r>r_{p,q}(\mathbb{Z}_n)$,  an {\it $r$-supercritical $(p,q)$-extremizer}  or simply {\it supercritical extremizer} (which is automatically non-constant) is defined as a function $f$  on $\mathbb{Z}_n$ satisfying
\[
1< \| T_r\|_{L^p(\mathbb{Z}_n) \to L^q(\mathbb{Z}_n)} = \| T_r f\|_q/\|f\|_p. 
\]
\end{definition*}

%This observation leads us to believe that the hypercontractive constant \(r_{p,q}(\mathbb{Z}_3)\) may not admit an exact closed-form expression in terms of \(p\) and \(q\). We therefore focus on studying its implicit functional behavior.  
%Following this approach, we resolve the problem by determining \(r_{p,q}(\mathbb{Z}_3)\) for all \(1 < p \leq q < \infty\).

%To the best of our knowledge, the general optimal constants \(r_{p,q}(\mathbb{Z}_3)\) remain unknown. 

The supercritical extremizers always exist. Indeed, since $L^p(\mathbb{Z}_n)$ is of finite dimension, the standard compactness argument implies that, for any $r>r_{p,q}(\mathbb{Z}_n)$, the $r$-supercritical extremizers always exist.  In the particular case of $\mathbb{Z}_3$, Wolff \cite[Theorem 3.1]{Wolff} provides a remarkable qualitative description of supercritical extremizers: they  must take exactly two values!  This description of supercritical extremizers combined with Lata{\l}a and Oleszkiewicz \cite{LaO00,Ole03} leads to the explicit expression \eqref{2q-Z3} of $r_{2,q}(\mathbb{Z}_3)$. Indeed, Wolff \cite[Theorem 3.1]{Wolff} proved that, for more general operators on a finite discrete probability space, any supercritical extremizer must exhibit a two-valued feature (see \S \ref{sec-hyper-rv} for a brief introduction). Wolff's qualitative description for supercritical extremizers has proven to be highly influential; see, e.g., \cite{DFR-08, Mos-10, MDO-10, FF-14, HHM-18}.

On the other hand, the existence of nontrivial critical extremizers {\it is not guaranteed} (although Wolff's argument implies that all nontrivial extremizers, once exist, must be exactly two-valued).  For instance,  it is not hard to show that $\mathbb{Z}_2$ does not admit nontrivial critical extremizers for any $1<p<q<\infty$.  In sharp contrast to the classical situation for $\mathbb{Z}_2$, partially owing to the strict inequality (see Lemma~\ref{lem-rough}) 
$$
%\label{strict-rpq}
r_{p,q}(\mathbb{Z}_3) < \sqrt{(p-1)/(q-1)}, 
$$
 we unexpectedly discover the existence of nontrivial $(p,q)$-critical extremizers.  

\begin{lemma}\label{lem-nce}
For any $1<p<q<\infty$, there exists a nontrivial $(p,q)$-critical extremizer for $\mathbb{Z}_3$. 
\end{lemma}

Lemma~\ref{lem-nce} will be proved in Proposition~\ref{prop-nontrivial-extremizer} in the cases \(1 < p < q < 2\) or \(1 < p \leq 2 <  q<\infty\), while the remaining case will be proved in \S \ref{sec-ass-unique}. 

\begin{conjecture*}
 One may conjecture that, for  $n\ge 2$,  the equality $r_{p,q}(\mathbb{Z}_n) = \sqrt{(p-1)/(q-1)}$  holds if and only if $\mathbb{Z}_n$ does not admit  nontrivial $(p,q)$-critical extremizer. 
\end{conjecture*}

The existence of nontrivial critical extremizers for $\mathbb{Z}_3$ plays a key role in this work. Specifically, any such extremizer $f$ carries essential information about the constant $r_{p,q}(\mathbb{Z}_3)$. Through the simplest {\it variational principle} (i.e., Fermat's principle), this enables us to derive a system of equations involving both the Fourier coefficients of $f$ and the target constant $r_{p,q}(\mathbb{Z}_3)$.   However, for obtaining a  tractable equation-system, we need to further reduce the space for candidate nontrivial critical extremizers,   see \S \ref{sec-intro-reducing} and \S \ref{S:parameter-reduction} for the details.  Note that the application of ideas from variational principles to the hypercontractive inequality is not new. The analysis of Lagrange multipliers, in particular, have been skillfully applied in works such as \cite{Wolff}, \cite{DFR-08} and \cite{yao2024optimal}.

 Our first main result is
 
 \begin{theorem}\label{thm-main-thm-one-xy}
Let \(1 < p <  q < \infty\). Then the hypercontractive constant for \(\mathbb{Z}_3\) is given by
\begin{align}\label{cross-rpq}
r_{p,q}(\mathbb{Z}_3) = \frac{(1 + 2x)(1 - y)}{(1 + 2y)(1 - x)},
\end{align}
where \((x,y)\) is the unique solution in the open unit square $(0,1)\times(0,1)$ to the system of equations
\begin{align}\label{xy-eq}
\left\{
\begin{aligned}
&\frac{1}{1+2x}\Big(\frac{1+2x^p}{3}\Big)^{\frac{1}{p}}=\frac{1}{1+2y}\Big(\frac{1+2y^q}{3}\Big)^{\frac{1}{q}},\\
&\frac{(1-x)(1-x^{p-1})}{1+2x^p}=\frac{(1-y)(1-y^{q-1})}{1+2y^q}.
\end{aligned}
\right.
\end{align}
Moreover,  the function $f=1 + \frac{1-x}{1+2x} (\chi +\overline{\chi})
$
is a nontrivial critical extremizer.
\end{theorem}

 The uniqueness of the solution in $(0,1)^2$ to the system \eqref{xy-eq} is a highly nontrivial fact (which we suspect may be related to some deep result in real algebraic geometry, at least when $p, q$ are rational). For proving this uniqueness,  inspired by the theory of {\it Whitney pleats} (see \cite[\S1]{Arnold-Sing}), we are led to analyze the intersection properties of a  curve-family  associated with \eqref{xy-eq} by employing {\it symmetrization} and {\it blowup techniques}.

Since the proof of the uniqueness statement of  Theorem~\ref{thm-main-thm-one-xy} is considerably more involved than the derivation of the system \eqref{xy-eq} itself, here  we  provide an alternative description of $r_{p,q}(\mathbb{Z}_3)$  that avoids relying on this uniqueness. Specifically, by the absence of nontrivial subcritical extremizers for $\mathbb{Z}_3$ (see the remark below) and the argument in the establishing \eqref{xy-eq}, one may obtain 
\begin{align}\label{rpq-as-min}
r_{p,q}(\mathbb{Z}_3) = \min\left\{ \frac{(1 + 2x)(1 - y)}{(1 + 2y)(1 - x)}:   \text{\((x,y)\) is a solution in $(0,1)^2$ to \eqref{xy-eq}} \right\}. 
\end{align}

Even without the uniqueness result from Theorem~\ref{thm-main-thm-one-xy}, the system \eqref{xy-eq} is already quite useful, as it enables numerical computation of $r_{p,q}(\mathbb{Z}_3)$. However, our proof of uniqueness carries both theoretical and practical significance: we observed that directly solving \eqref{xy-eq} numerically is highly unstable, whereas employing the symmetrization and blowup techniques developed in the uniqueness proof yields an efficient and stable numerical simulation of $r_{p,q}(\mathbb{Z}_3)$.

\begin{remark*}
For any $1<p<q<\infty$, one can define  {\it subcritical extremizers} for  $0<r<r_{p,q}(\mathbb{Z}_n)$. However,   there is no  nontrivial subcritical extremizer for $\mathbb{Z}_n$, since  $T_r = T_{r/r_{p,q}(\mathbb{Z}_n)}\circ T_{r_{p,q}(\mathbb{Z}_n)}$ and only  the constant functions can have $T_{r/r_{p,q}(\mathbb{Z}_n)}$-invariant $L^q$-norms. 
\end{remark*}

\begin{remark*}  As a consequence of the uniqueness of the solution to \eqref{xy-eq}, for any $1<p<q<\infty$,  one may prove that,  up to multiplication by a non-zero constant,  there exists a unique nontrivial critical extremizer for $\mathbb{Z}_3$ (which however will vary for different pairs $(p,q)$).   
\end{remark*}

\begin{remark*}
The formula \eqref{cross-rpq} for $r_{p,q}(\mathbb{Z}_3)$ is exactly a cross ratio (see, e.g.,  \cite[\S 4.4]{Bea-95}): 
$$
%\label{cross-form}
r_{p,q}(\mathbb{Z}_3)= (x,y; -1/2,1): = \frac{(x+1/2)(y-1)}{(x-1) (y+1/2)}.
$$
Currently,  the geometric connection between  $r_{p,q}(\mathbb{Z}_3)$ and the cross ratio is not known to us. 
\end{remark*}

\begin{remark*}
In \cite{Borell-82}, Borell established a reverse hypercontractive inequality on $\mathbb{Z}_2$. Recently, Mossel, Oleszkiewicz, and Sen \cite{MOS-13} showed that reverse hypercontractive inequalities follow from standard hypercontractive inequalities or the modified log-Sobolev inequality. In particular, they obtained that Borell's reverse hypercontractive inequalities hold uniformly for all probability spaces. It seems that the method in this paper can be adapted to obtain the optimal constant of the reverse hypercontractivity on $\mathbb{Z}_3$.  Indeed, one can show (with a little effort) that the equation-system \eqref{xy-eq} has a unique solution in $(0,1)^2$ when $0<q<p<1$. Thus, the optimal reverse hypercontractive constants also have the form \eqref{cross-rpq}. 
\end{remark*}

It is clear that the hypercontractive constant $r_{p,q}(\mathbb{Z}_3)$ has a  symmetry of duality: 
\begin{align}\label{dual-rpq}
r_{p, q}(\mathbb{Z}_3)= r_{q^*, p^*}(\mathbb{Z}_3),
\end{align}
 where $p^* = p/(p-1), q^* = q/(q-1)$ are the  conjugate exponents.  Meanwhile, the  equation-system \eqref{xy-eq} also possesses a symmetry of duality.   For any $1<p<\infty$, define 
\[
\ell(p, x): = \frac{1}{1+2x}\Big(\frac{1+2x^p}{3}\Big)^{\frac{1}{p}}, \quad x \in (0,1). 
\]

\begin{lemma}\label{lem-self-dual}
The equation-system \eqref{xy-eq} has the following self-dual equivalent formulation: 
\begin{align}\label{self-dual-form}
\left\{
\begin{aligned}
\ell(p, x) & =\ell(q,y),\\
\ell(p^*, x^{p-1})&=\ell(q^*, y^{q-1}).
\end{aligned}
\right.
\end{align}
In particular, the solution of the equation-system \eqref{xy-eq} is invariant under the transform 
\[
(p, q, x, y) \mapsto (q^*, p^*, y^{q-1}, x^{p-1}). 
\] 
This transform also leaves invariant the cross ratio $(x,y; -1/2, 1)$ for the solution $(x,y)$ of \eqref{xy-eq}: 
\begin{align}\label{cross-sym}
(x,y; -1/2, 1) = (y^{q-1}, x^{p-1}; -1/2, 1). 
\end{align}
\end{lemma}

In general, $r_{p,q}(\mathbb{Z}_3)$ rarely admits an explicit expression; the following corollary presents the rare instances that we conjecture to be the only ones.  Note that,  $r_{2,p^*}(\mathbb{Z}_3)   = r_{p,2}(\mathbb{Z}_3)$ is due to Wolff \cite{Wolff}, while $r_{p,p^*}(\mathbb{Z}_3)$ appears to be  new in the literature. One shall see in \S \ref{sec-sym} that these expressions are related to the symmetry \eqref{self-dual-form}.

\begin{corollary}\label{cor:pp}
Let $1 < p < 2$.
Then
$$r_{p,p^*}(\mathbb{Z}_3)=
2(4^{1/p^*} - 1)/(4 - 4^{1/p^*}) \text{\,\, and \,\,}
r_{p,2}(\mathbb{Z}_3)=
r_{2,p^*}(\mathbb{Z}_3)=\sqrt{2(4^{1/p^*} - 1)/(4 - 4^{1/p^*})}.$$
Consequently, the multiplicative relation $r_{p,p^*}(\mathbb{Z}_3)=r_{p,2}(\mathbb{Z}_3)r_{2,p^*}(\mathbb{Z}_3)$ holds. 
\end{corollary}

%\begin{remark}
%The hypercontractive constants $r_{2,p^*}(\mathbb{Z}_3)$ and $r_{p,2}(\mathbb{Z}_3)$ are obtained by P. Wolff \cite{Wolff}.
%\end{remark}
%\begin{remark}
%The following proof of Corollary \ref{cor:pp} reflects geometric intuitions regarding curve intersections which are key ideas in the uniqueness of the system \eqref{eq:xy-first}. . 
%\end{remark}
\begin{remark*}
While $r_{p,q}(\mathbb{Z}_2)=r_{p,s}(\mathbb{Z}_2)r_{s,q}(\mathbb{Z}_2)$ for all $1<p\leq s\leq q<\infty$, it seems to be a common understanding that the multiplicative relation does not generally hold for \(r_{p,q}(\mathbb{Z}_3)\). See also the numerical solutions of the system \eqref{xy-eq}  in  Figure ~\ref{multiplicative relationship}. The non-multiplicative relation seems another obstacle in expecting an explicit formula for $r_{p,q}(\mathbb{Z}_3)$.
\begin{figure}[htbp]
\centering
\includegraphics[width=0.35\textwidth]{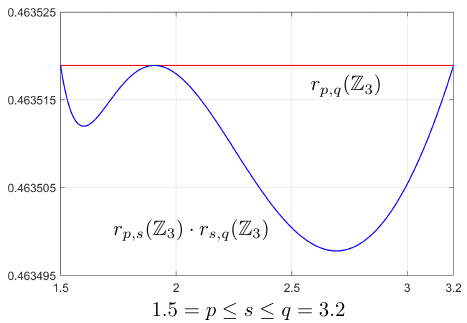}
\caption{Non-multiplicative relationship}
\label{multiplicative relationship}
\end{figure}
%\textcolor{red}{To be added by Jie Cao}
\end{remark*}

\begin{corollary}\label{coro-algebraic}
For any $p,q\in \mathbb{Q}$ such that  $1<p<q<\infty$,  $r_{p,q}(\mathbb{Z}_3)$ is an algebraic number.
\end{corollary}

The derivation of Corollary \ref{coro-algebraic} from Theorem \ref{thm-main-thm-one-xy} seems rather elementary in algebraic geometry, for the reader's convenience we provide a short proof in \S \ref{sec-daishushu}  of the appendix  that does not require any background in the subject.

\begin{remark*}[Absence of radical expression in general]
For $p, q\in \mathbb{Q}$,  the minimal polynomial of $r_{p,q}(\mathbb{Z}_3)$ may be written explicitly. Indeed, in view of Theorem~\ref{thm-main-thm-one-xy} or  the equivalent equation-system \eqref{eq:original} concerning directly $r_{p,q}(\mathbb{Z}_3)$,  one can use the standard {\it method of resultant}  (see the proof of Corollary~\ref{coro-algebraic} in \S \ref{sec-daishushu} of the appendix  for a quick explanation) to derive an integer-coefficient polynomial (and then an irreducible one) having $r_{p,q}(\mathbb{Z}_3)$ as a root.  For instance, 
using computer algebra systems (e.g., Mathematica, Matlab), we obtain the minimal polynomial of $r_{3,6}(\mathbb{Z}_3)$: 
\begin{align*}
&2600125 r^{20}+2600125 r^{19}-54275 r^{18}+3456600 r^{17}-1846590 r^{16}
-5287590 r^{15}+901467 r^{14}\\
-\ &3882063 r^{13}+1557057 r^{12}+4269364 r^{11}-3942536 r^{10}+1575484 r^9+1067232 r^8\\
-\ &1287048 r^7+407592 r^6+97920 r^5-118680 r^4+17160 r^3-5540 r^2-20 r-20.
\end{align*}
Furthermore, with Magma (available online at \url{https://magma.maths.usyd.edu.au/calc/}), we confirm that its Galois group is the symmetric group $S_{20}$, which is non‑solvable. 
Therefore, $r_{3,6}(\mathbb{Z}_3)$ has no radical expression.  This appears to be the underlying reason behind the lack of an explicit expression for  $r_{p,q}(\mathbb{Z}_3)$ for general $p,q$. 
\end{remark*}

\subsection{Optimal hypercontractive constants for biased Bernoulli random variables} \label{sec-hyper-rv}
For the study of hypercontractive phenomena on the discrete probability space \((\Omega, \pi)\), it is convenient to use the notation of hypercontractive random variables introduced by Krakowiak and Szulga (see, e.g., \cite{KS-88}, \cite{KS-91} and \cite[Definition 9.13]{Donnell}). For \(1 < p < q < \infty\) and \(0 \leq r < 1\),  a real-valued random variable \(X \in L^q\) is said to be $(p,q, r)$-hypercontractive if  
\[
\|1 + r \rho X\|_q \leq \|1 + \rho X\|_p, \qquad \forall \rho \in \mathbb{R}.
\]
The largest \( r \) for which the \((p, q, r)\)-hypercontractive inequality holds is called the optimal hypercontractivity constant of the random variable \(X\), and is denoted by \(\sigma_{p,q}(X)\). 

Let \(\lambda \in (0, 1/2]\). The biased Bernoulli random variable \(X_\lambda\) is defined by
\begin{align}\label{eqn-biased-rw}
\mathbb{P}(X_\lambda = 1- \lambda)  =1  - \mathbb{P}(X_\lambda = -\lambda) =  \lambda. 
\end{align}
The optimal hypercontractivity constant \(\sigma_{p,q}(X_\lambda) \) is, in this case, simply denoted by \(\sigma_{p,q}(\lambda)\).
The $\lambda=1/2$ case is the classical symmetric two-point hypercontractivity inequality, i.e., 
\[
\sigma_{p,q}(1/2)=r_{p,q}(\mathbb{Z}_2)=\sqrt{(p-1)/(q-1)}.
\]
For $\lambda\not=1/2$, the constant $\sigma_{p,q}\left(\lambda \right)$ is called the asymmetric (two-point) hypercontractive constant. 

Determining the asymmetric hypercontractive constant $\sigma_{p,q}(\lambda)$ is a highly nontrivial task, even in specific cases. Talagrand \cite{Tala-94} established the first nontrivial estimate. Oleszkiewicz \cite{Ole03} derived the first exact formulas when $p=2$ (or $q = 2$). Wolff not only extended the Oleszkiewicz formula to general discrete spaces but also obtained analogous results, in a slightly less precise form, for all $2\leq p<q$ and all $1<p<q\leq 2$. Specifically, Wolff \cite[Theorems 2.1 and 2.2]{Wolff} showed that $\sigma_{p,q}(\lambda)$ is equivalent (up to a universal numerical constant) to an explicit constant $\widetilde{\sigma}_{p,q}(\lambda)$. He \cite[Theorem 3.1]{Wolff} also proved that the hypercontractivity constant of any finite discrete measure is bounded by $\sigma_{p,q}(\lambda)$, where $\lambda$ denotes the mass of the measure's smallest atom. In particular,  by adapting Wolff's argument, one may show that for all $1 < p < q < \infty$, 
\[
r_{p,q}(\mathbb{Z}_3) = \sigma_{p,q}(1/3).
\]

To the best of our knowledge, the exact Oleszkiewicz formula for all $1<p<q<\infty$ remains open. We now state our second main result. Using ideas of proving Theorem \ref{thm-main-thm-one-xy}, we have 

\begin{theorem}\label{thm-biased-alpha}
For $\lambda\in (0,1/2)$ and $1<p<q<\infty$, the hypercontractive constant for the $\lambda$-biased Bernoulli random variable  satisfies
$$
%\label{eqn-optionmal-biased}
\sigma_{p,q}(\lambda)=\frac{(1-y)\left(x+\frac{\lambda}{1-\lambda}\right)}{(1-x)\left(y+\frac{\lambda}{1-\lambda}\right)},
$$
with $(x,y)$ is the  solution in $(0,1)^2$ of the system
\begin{align}\label{eqn-biased-duiou-system}
\left\{
\begin{aligned}
&\frac{\big(\lambda + (1-\lambda) x^p\big)^{\frac{1}{p}}}{\lambda + (1-\lambda)x} = \frac{\big(\lambda + (1-\lambda) y^q\big)^{\frac{1}{q}}}{\lambda + (1-\lambda)y},\\
&\frac{(1-x)(1 - x^{p-1})}{\lambda + (1-\lambda) x^p} = \frac{(1-y)(1 - y^{q-1})}{\lambda+ (1-\lambda) y^q}.
\end{aligned}
\right.
\end{align}
In the particular case $q=p^*$, one has
\begin{align}\label{eqn-biased-duiou}
\sigma_{p,p^*}(\lambda)=\frac{\sinh \left(\frac{1}{p^*} \log \frac{1-\lambda}{\lambda}\right)}{\sinh \left(\frac{1}{p} \log \frac{1-\lambda}{\lambda}\right)}.
\end{align}
\end{theorem}

\begin{remark*}
The reader may notice that, in Theorem~\ref{thm-biased-alpha}, we do not state the uniqueness of the solution in the unit square of the equation-system \eqref{eqn-biased-duiou-system}.    Indeed, with a few modifications,  the proof of the uniqueness statement in the setting of Theorem~\ref{thm-biased-alpha} follows  almost verbatim from that of Theorem~\ref{thm-main-thm-one-xy} for large $\lambda$ (say $1/5\leq\lambda<1/2$). However, for smaller  parameter $\lambda$, some additional work is needed. Therefore, in this general setting of Theorem~\ref{thm-biased-alpha}, we have decided to include all the proofs of the uniqueness in a separate document not intended for publication.  
\end{remark*}

\begin{remark*}
 For small parameter $\lambda$, a brief description of the scenario is given below. For each $\lambda\in(0,1/2)$, define
$$
\begin{array}{rccc}
h_{\lambda}:& (1,\infty)\times(0,1)&\longrightarrow &\mathbb{R}\times \mathbb{R}\\
& (p,x)&\longmapsto&
\Big(
\frac{(\lambda + (1-\lambda) x^p)^{1/p}}
{\lambda + (1-\lambda)x} ,\;
\frac{(1-x)(1 - x^{p-1})}{\lambda + (1-\lambda) x^p}
\Big).
\end{array}
$$
Since $1-x$, $1 - x^{p-1}$, $1/(\lambda + (1-\lambda) x^p)$ are all strictly decreasing in $x$, the ordinate of the curve $h_\lambda(p,\cdot)$ decreases in $x$. 
So $h_\lambda(p,\cdot)$ has no self-intersecting point. 
It follows that the system 
\eqref{eqn-biased-duiou-system} admits a unique solution in \((0,1)^2\) if and only if the curves $h_\lambda(p,\cdot)$ and $h_\lambda(q,\cdot)$ intersect exactly once. 
Through suitable changes of coordinates, we will obtain another curve family 
$$
\begin{array}{rccc}
H_{\lambda}=(H_{\lambda,1},H_{\lambda,2}):& (-1,1)\times(0,1)&\longrightarrow &\mathbb{R}\times \mathbb{R}\\
& (\alpha,t)&\longmapsto&
\left(\begin{array}{c}
\tfrac{\alpha\log(\lambda+(1-\lambda)t^2) 
-\log(\lambda+(1-\lambda) t^{1+\alpha})
+\log(\lambda+(1-\lambda) t^{1-\alpha})}
{\log\lambda}\\[2mm]
\tfrac{-\log(\lambda+(1-\lambda)t^2) 
+\log(\lambda+(1-\lambda) t^{1+\alpha})
+\log(\lambda+(1-\lambda) t^{1-\alpha})}
{\log\lambda}
\end{array}\right).
\end{array}
$$
Direct computation gives the following properties. 
\begin{itemize}
\item Symmetries $H_{\lambda,1}(-\alpha,t)=-H_{\lambda,1}(\alpha,t)$ and $H_{\lambda,2}(-\alpha,t)=H_{\lambda,2}(\alpha,t)$. 
\item The end points of $H_\lambda(\alpha,\cdot)$ are  
$H_\lambda(\alpha,0) = (\alpha,1)$ and $H_\lambda(\alpha,1) = (0,0)$. 
\item Since $H_{\lambda,1}(0,t)\equiv0$, we have $H_\lambda(0,[0,1]) = \{0\}\times[0,1]$. 
\item Since $H_{\lambda,1}(\alpha,\tfrac{\lambda}{1-\lambda}) \equiv 0$, we have $H_\lambda(\alpha,\tfrac{\lambda}{1-\lambda}) \in \{0\}\times[0,1]$. 
\end{itemize}
Furthermore, we have the following observations from graphs. 
\begin{itemize}
\item For $\alpha\neq\beta$, the curves $H_\lambda(\alpha,\cdot)$ and $H_\lambda(\beta,\cdot)$ intersect exactly once in the upper half-plane. This is the evidence that the system 
\eqref{eqn-biased-duiou-system} admits a unique solution in \((0,1)^2\). 
\item The singularity $(\alpha,t)=(0,\tfrac{\lambda}{1-\lambda})$ is a Whitney pleat and 
$$H_\lambda(0,\tfrac{\lambda}{1-\lambda}) = (0,\tfrac{\log(4\lambda(1-\lambda))}{\log\lambda})\xrightarrow[\lambda\to0]{}(0,1).$$
\item When $\lambda\to0$, the curve family $H_\lambda$ has a limit shape as shown in Figure \ref{fig-H-lambda}. 
\end{itemize}

\begin{figure}[htbp]
\centering
\includegraphics[width=0.5\textwidth]{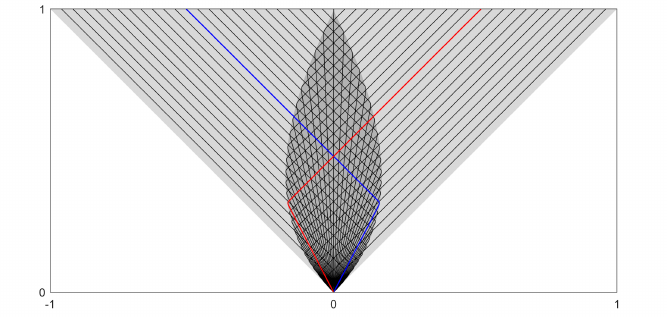}
\caption{The curve family $\{H_\lambda(\alpha,\cdot)\}_{\alpha\in(-1,1)}$ with $\lambda = 10^{-100}$
}
\label{fig-H-lambda}
\end{figure}
\end{remark*}

\begin{remark*}
Similar to   \eqref{rpq-as-min},  one also has  a description of $\sigma_{p, q}(\lambda)$ that avoids relying on the uniqueness of the solution in $(0,1)^2$ to \eqref{eqn-biased-duiou-system}: 
\[
\sigma_{p,q}(\lambda)= \min\left\{\frac{(1-y)\left(x+\frac{\lambda}{1-\lambda}\right)}{(1-x)\left(y+\frac{\lambda}{1-\lambda}\right)} :   \text{\((x,y)\) is a solution in $(0,1)^2$ to \eqref{eqn-biased-duiou-system}} \right\}. 
\]
\end{remark*}

Let $\mu$ be a probability measure on a finite set $\Omega$, for any $1<p<q<\infty$,  define 
\begin{align*}
\sigma_{p,q}(\mu): &= \max\big\{0<r<1: \| \mathbb{E}_\mu + r ( Id  - \mathbb{E}_\mu)\|_{L^p(\mu) \rightarrow L^q(\mu)}\le 1\big\}
\\ &= \inf\big\{\sigma_{p,q}(X): \text{$X$ is random variable defined on $\Omega$ with $\mathbb{E}_\mu(X)=0$}\big\}.
\end{align*} 
Since $\sigma_{p,p^*}(\lambda)$ given in \eqref{eqn-biased-duiou} is increasing in $\lambda\in (0,1/2)$, 
using \cite[Theorem 3.1]{Wolff}, we have
\begin{corollary}\label{corollary-wolf-df}
Let $\mu$ be a probability measure on a finite set $\Omega$ with the mass of the least atom equal to $\lambda \in(0,1 / 2)$. Then for all $1<p<2$,  we have $\sigma_{p, p^*}(\mu)=\sigma_{p, p^*}(\lambda)$
%$$
%\sigma_{p, p^*}(\mu)=\sigma_{p, p^*}(\alpha)=\frac{\sinh \left(\frac{1}{p^*} \log \frac{1-\alpha}{\alpha}\right)}{\sinh \left(\frac{1}{p} \log \frac{1-\alpha}{\alpha}\right)}
%$$
which is given by \eqref{eqn-biased-duiou}.
\end{corollary}

%\begin{remark*}
%The Haar  probability measure  on $\mathbb{Z}_3$ is the uniform  measure. By \cite[Theorem 3.1]{Wolff},
%\[
%\inf_{\lambda\in[1/3,1/2]}\sigma_{p,q}(\lambda)\leq r_{p,q}(\mathbb{Z}_3)\leq \sigma_{p,q}(1/3).
%\] 
%For $2< q < \infty$, the monotonicity of $\sigma_{2,q}(\lambda)$ with respect to $\lambda$ (see \cite[Lemma 2.1]{Wolff}) implies
%\begin{equation}\label{eq:sigman-r-2-q}
%r_{2,q}(\mathbb{Z}_3) = \sigma_{2,q}(1/3).
%\end{equation}
%Our result in this paper, as well as some variant of Wolff's argument,  shows that \eqref{eq:sigman-r-2-q} holds for all $1 < p < q < \infty$.
%% Consequently, Theorem~\ref{thm-main-thm-one-xy} can be viewed as a particular case of Theorem~\ref{thm-biased-alpha}. 
%\end{remark*}

\begin{remark*}
Since 
$
\sigma_{p,q}(\lambda)=\sigma_{p,q}(1-\lambda)
$, Theorem~\ref{thm-biased-alpha} gives also $\sigma_{p,q}(\lambda)$ for all $\lambda\in (0,1)\setminus\{1/2\}$. 
\end{remark*}
\begin{remark*}
Numerical simulation  indicates that $\sigma_{p,q}(\lambda)$ 
is increasing in $\lambda \in (0,1/2]$. 
\begin{figure}[htbp] 
\centering
\includegraphics[width=0.45\textwidth]{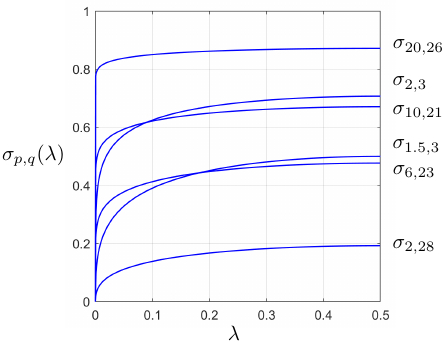}
\caption{Numerical simulation of $\sigma_{p,q}(\lambda)$}\label{figtriangle001}
\end{figure}
\end{remark*}

 Inspired by the  simulations in Figure~\ref{figtriangle001}, using \cite[Theorem 3.1]{Wolff}, we formulate the following 
\begin{conjecture*}
Let $\mu$ be a probability measure on a finite set $\Omega$ with the mass of the least atom equal to $\lambda \in(0,1 / 2)$. Then for all $1<p<q<\infty$,  we have $\sigma_{p, q}(\mu)=\sigma_{p, q}(\lambda)$.
\end{conjecture*}

\subsection{Noise operators}

The noise operator plays an important role in $\lambda$-biased Fourier analysis, and provides another key example of hypercontractive random variables. For $\lambda \in (0,1/2]$, let $\Omega = \{-1,1\}$ and let $\pi_\lambda$ be the distribution on $\Omega$ defined by
\[
\pi_\lambda(-1) = \lambda \quad \text{and} \quad \pi_\lambda(1) = 1-\lambda.
\]
An orthogonal basis (Fourier basis) for the real Hilbert space $L^2(\Omega,\pi_\lambda)$ consists of the constant function $1$ and 
\[
\phi_\lambda(x) = \frac{x - (1 - 2\lambda)}{2\sqrt{\lambda(1-\lambda)}}, \quad x \in \Omega.
\]
For $r \in (0,1]$, the \emph{noise operator with parameter $r$} is the linear operator on $L^2(\Omega,\pi_\lambda)$ defined by
\begin{align}\label{eqn-nise-operator}
T_r f = 1 + r\rho \phi_\lambda, \quad \text{with} \quad  f=1+\rho\phi_\lambda \quad \text{and} \quad \rho = \mathbb{E}_{x \sim \pi_\lambda}[f(x)\phi_\lambda(x)]. 
\end{align}
 In this setting, the hypercontractive random variable is given by $Y_\lambda= \phi_\lambda$:
\begin{align}\label{eqn-biasedss}
\mathbb{P}\Big(Y_\lambda=\sqrt{\lambda/(1-\lambda)}\Big)=1-\lambda \text{\,\, and \,\,} \mathbb{P}\Big(Y_\lambda=-\sqrt{(1-\lambda)/\lambda}\Big)=\lambda.
\end{align}
The related problem is known as $\lambda$-biased hypercontractive problem (see \cite[Chapters 8 and 9]{Donnell}). Further details can be found in \cite[Definitions 8.26, 8.27, 8.39 and  Proposition 8.27]{Donnell}.

%When $\Omega = \{-1,1\}$ is equipped with the uniform measure, $T_r$ reduces to the classical unbiased noise operator. %Specifically, letting $a = f(1)$ and $b = f(-1)$, the function $T_\rho f$ takes the values  
%\[
%T_\rho f(1) = \frac{1+\rho}{2}\,a + \frac{1-\rho}{2}\,b, \qquad
%T_\rho f(-1) = \frac{1-\rho}{2}\,a + \frac{1+\rho}{2}\,b.
%\]
%The $(p,q)$-hypercontractive inequality for the noise operator $T_r$ is exact the two-point hypercontractive inequality. 
%One can consult \cite[Chapter??]{Donnell} and \cite[Chapter 1]{Zhao2021} for more details.

By the works of Bourgain-Kahn-Kalai-Katznelson-Linial \cite{BKKKL}, Talagrand \cite{Tala-94}, Friedgut-Kalai \cite{FK-96}, and Friedgut \cite{Fri-98}, the general hypercontractive theorem \cite[Page 278]{Donnell} states that for every $f \in L^2(\Omega,\pi_\lambda)$,
\[
\|T_r f\|_q \leq \|f\|_2 \quad \text{whenever} \quad 0 \leq r \leq \frac{1}{\sqrt{q-1}} \lambda^{\,1/2 - 1/q}.
\]

It is observed that when $\lambda$ is small, the hypercontractive inequality requires $r$ to also be small, and such a bound is often too weak for applications. In a recent breakthrough, Keevash, Lifshitz, Long, and Minzer \cite{KLLM-24} introduced global hypercontractivity, a new form of biased hypercontractive inequality. This tool yields a strengthened version of Bourgain's theorem, advances on a conjecture of Kahn and Kalai, and a $\lambda$-biased analogue of the Mossel–O'Donnell–Oleszkiewicz invariance principle. Subsequently, Keller, Lifshitz, and Marcus \cite{KLM-26} established sharp global hypercontractivity, leading to numerous applications in extremal set theory, group theory, theoretical computer science, and number theory. From the concluding remarks of \cite{KLLM-24}, global hypercontractivity yields a variant of the Kahn--Kalai isoperimetric conjecture that is effective only in the $\lambda$-biased setting for small $\lambda$; corresponding results for the uniform measure are given in \cite{KL-18} and \cite{KLi-18}. The authors posed the open problem of finding a unified approach that extends both results to all $\lambda \in (0, 1/2]$. However, we note that the correct bound for general hypercontractivity is unknown \cite{Zhao2021}. 

One may notice that the hypercontractive analysis of the noise operator defined in \eqref{eqn-nise-operator} can be transferred to that  of the  random variable $Y_\lambda$ defined as \eqref{eqn-biasedss}. Indeed, one has 
\[
Y_\lambda=-\frac{1}{\sqrt{\lambda(1-\lambda)}}X_\lambda
\]
and  hence $Y_\lambda$ and $X_\lambda$ share the same optimal hypercontractive constant. Therefore, Theorem~\ref{thm-biased-alpha} provides the optimal hypercontractive constant for  the noise operator defined in \eqref{eqn-nise-operator}.

%\subsection{Discussion on the reverse hypercontractivity } \textcolor{red}{Add GAFA reverse}
%Let $0<q<p<1$, the reverse hypercontractivity inequality involves finding the largest $r$ such that 
%\[
%\|T_rf\|_q\geq \|f\|_p\quad \text{for all }\quad f=1+a\chi+\bar{a}\bar{\chi}\geq 0.
%\]
%Using our method, we can also give the optimal constant of the reverse hypercontractivity on $\mathbb{Z}_3$.  We can also show (with a little effort) that the equation system \eqref{xy-eq} has a unit solution in $(0,1)^2$ when $0<q<p<1$.  Then the optional hypercontractive constant also has the form \eqref{cross-rpq}. 

\subsection{Outlines for the proof of Theorem \ref{thm-main-thm-one-xy}}
The proof is divided into three main parts: 
\begin{enumerate}
\item reducing the problem from the full space $\mathbb{C}^3$ to the segment $[0,1]$; 
\item establishing the equation-system via variational principle; 
\item proving the uniqueness of the solution to this system. 
\end{enumerate}

For convenience, throughout the whole paper,  we denote $r_{p,q}(\mathbb{Z}_3)$ simply by $r_{p,q}$:
\[
r_{p,q} = r_{p,q}(\mathbb{Z}_3). 
\]

\subsubsection{Reducing the full space to a segment}\label{sec-intro-reducing}

The reduction proceeds in two main steps, each narrowing the candidate space for nontrivial critical extremizers.   While the first reduction is relatively routine --- using standard symmetry arguments --- the second reduction is technical and crucial. It effectively \textbf{lowers the dimension of the parameter space} from $\mathbb{C}^3$  to a one-dimensional interval \([0,1]\). This dimensional reduction is what ultimately allows us to formulate and solve a tractable system of equations for the extremal parameters in the subsequent analysis.

\begin{remark*}
An alternative way of reducing the space for candidate nontrivial critical extremizers is to use Wolff's argument Wolff \cite[Theorem 3.1]{Wolff}: assuming their existence, they must be exactly two-valued.  In this paper, we choose to give a self-contained direct reduction which seems to be of independent interest. Moreover,  the proofs provide key strategies and insights for our proof of the uniqueness of the solution to \eqref{xy-eq}. It can therefore be seen as a prelude to that proof. 
\end{remark*}

\textbf{Reduction 1: from $\mathbb{C}^3$ to $\Delta$ via symmetry.}  
Define the {\it defect functional}
\[
D_{p,q,r}(f) = \|f\|_p - \|T_r f\|_q, 
\quad \forall  f \in L^p(\mathbb{Z}_3). 
\]
Via Fourier expansion, any \( f \) can be written as \( f =a_0 + a_1\chi + a_2\overline{\chi} \).  
Thus, the functional \( D_{p,q,r} \) can be viewed as a function of the Fourier coefficients. 
The following  reduction is standard. 
\begin{enumerate}
\item By homogeneity, we may assume \(a_0=1\) without loss of generality. 
\item Since $\|f\|_p=\||f|\|_p$ and $\|T_r f\|_p\leq \|T_r|f|\|_p$, 
it suffices to consider non-negative functions, which are necessarily real.
Hence, we just need to consider 
\[
f_a := 1 + a\chi + \overline{a}\,\overline{\chi},\quad a \in \mathbb{C}.
\] 
 The condition \(f_a \geq 0\) forces \(a\) to lie in the equilateral triangle \(\Delta_0 \subset \mathbb{C}\); see Figure \ref{figtriangle00}. 
\item Then the symmetries  
$\|f_{a} \|_p= \|f_{\overline{a}} \|_p= \|f_{a e^{i2\pi/3}}\|_p$ (see Lemma \ref{lem:symmetry})
allow us to restrict $\rho e^{i\theta}$ to a smaller triangle \(\Delta \subset \Delta_0\); see Figure \ref{figtriangle00} again. 
\end{enumerate}

%Since the operator \(D_{p,q,r}\) is linear, we may assume \(a_0=1\) without loss of generality. 
%By $\|f\|_p=\||f|\|_p$ and $\|T_r f\|_p\leq \|T_r|f|\|_p$ (see Lemma \ref{lemma-convolution}), 
%it suffices to consider non-negative functions, which are necessarily real. 
%Hence, we just need to consider 
%\[
%f_a := 1 + a\chi + \overline{a} \overline{\chi},\quad a \in \mathbb{C}.
%\]
%The non-negative condition \(f_a \geq 0\) forces \(a\) to lie in the equilateral triangle \(\Delta_0 \subset \mathbb{C}\); see Figure \ref{figtriangle00}. 

\begin{figure}[htbp] 
\centering
\includegraphics[width=0.45\textwidth]{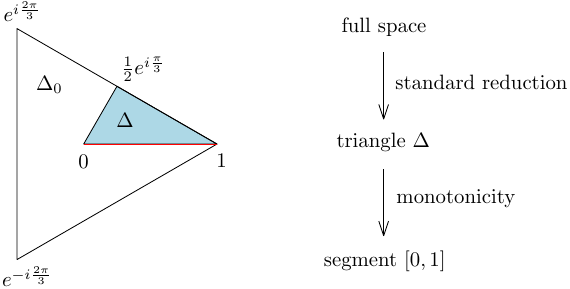}
\caption{Reduction architecture}\label{figtriangle00}
\end{figure}

Thus for any \(r \in [0, 1]\), we have 
$$\|T_r f\|_q \leq \|f\|_p,\ \forall\ f \in L^p(\mathbb{Z}_3) \Longleftrightarrow 
\|T_r f_{a}\|_q \leq \|f_{a}\|_p,\  \forall\  a\in\Delta.$$

%\textbf{Reduction 2: Exploit symmetry to shrink \(\Delta_0\) to \(\Delta\).}  
%Using the polar coordinate $a=\rho e^{i\theta}$, let 
%$$f_{\rho,\theta}=f_{\rho e^{i\theta}} = 1 + \rho e^{i\theta} \chi + \rho e^{-i\theta}\overline{\chi}.$$
%Then the symmetries  
%$\|f_{a} \|_p= \|f_{\overline{a}} \|_p= \|f_{\rho, \theta+\frac{2\pi}{3}}\|_p$
%allow us to restrict $\rho e^{i\theta}$ to a smaller triangle \(\Delta \subset \Delta_0\); see Figure \ref{figtriangle00} again. 
%Now for any \(r \in [0, 1]\), we have 
%$$\|T_r f\|_q \leq \|f\|_p,\ \forall\ f \in L^p(\mathbb{Z}_3) \Longleftrightarrow 
%\|T_r f_{a}\|_q \leq \|f_{a}\|_p,\  \forall\  \rho e^{i\theta}\in\Delta.$$
% That is, the hypercontractive inequality \eqref{hyper-pq} holds for all \(f \in L^p(\mathbb{Z}_3)\) if and only if it holds for all $f_{\rho,\theta}$ with $\rho e^{i\theta}\in\Delta$. 

 \textbf{Reduction 2:  from \(\Delta\) to \([0,1]\) via monotonicity.}  
By \eqref{dual-rpq},
it suffices to determine \(r_{p,q}\) for the  two cases
$$1 < p < q < 2\ \ \text{and}\ \  1 < p \leq 2 < q<\infty.$$

For $r\in[0,1]$ and $(\rho,\theta)$ with $\rho e^{i\theta}\in\Delta$, define the {\it defect function}
\begin{equation}\label{eq:defectfunction}
G(p,q,r,\rho,\theta) = D_{p,q,r}(f_{\rho e^{i\theta}}).
\end{equation}
By expanding \( G(p,q,r,\rho,0) \) at \( \rho=0 \), one easily verifies that
$$r_{p,q} < \sqrt{(p-1)/(q-1)}\text{ (see Lemma \ref{lem-rough})}.$$

In both cases \(1 < p < q < 2\) and \(1 < p \leq 2< q<\infty\), for any \(r \in [0,\sqrt{(p-1)/(q-1)}]\), the defect function $G(p,q,r,\rho,\theta)$
is increasing with respect to $\theta$ (see Lemmas \ref{lemma-reduce-p-q-1} and \ref{lemma-reduce-p-q-2}).    
So it suffices to check functions $f_\rho$ with \(\rho \in [0,1]\). 
That is, for any \(r \in  [0,\sqrt{(p-1)/(q-1)}]\), we have 
$$\|T_r f_{a}\|_q \leq \|f_{a}\|_p,\  \forall\   a\in\Delta
\Longleftrightarrow 
\|T_r f_\rho\|_q \leq \|f_\rho\|_p,\  \forall\  \rho \in [0,1].$$

\subsubsection{Establishing the system of equations}
%We proceed in two steps: first establishing the existence of a nontrivial extremizer, and then deriving the system of equations that characterizes it.
%
%\textbf{Existence of a nontrivial extremizer.}  
For fixed exponents \(1<p<q<\infty \), consider 
\begin{align}\label{def-Grrho}
G(r,\rho) &= G(p,q,r,\rho,0)=\left(\frac{\left(1+2\rho\right)^p+2\left(1-\rho\right)^p}{3}\right)^{\frac{1}{p}}
-\left(\frac{\left(1+2r\rho\right)^q+2\left(1-r\rho\right)^q}{3}\right)^{\frac{1}{q}}
\end{align}
with $(r,\rho) \in [0,1]\times[0,1]$. 
In Proposition \ref{prop-nontrivial-extremizer} below, we establish the existence of a \textbf{nontrivial extremizer}  of the form \( f_{\rho_0} \) with \(\rho_0 \in (0,1) \).  This further allows us to derive an equation-system via variational principle.  See Figure~\ref{Defect} for an illustration.  

\begin{figure}[htbp] 
\centering
\includegraphics{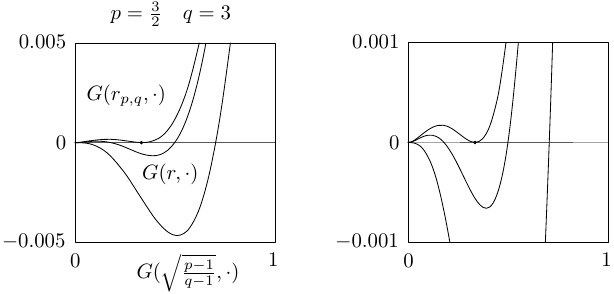}
\caption{Defect functions $G(r,\cdot)$ for different $r$}
\label{Defect}
\end{figure}

More precisely, since \( \rho_0 \in (0,1)\) is an interior point of \( [0,1] \) where  the function \(  G(r_{p,q}, \cdot) \) attains its minimum $G( r_{p,q},\rho_0) = 0$, the pair \( ( r_{p,q}, \rho_0) \) satisfies the system
\begin{align}\label{two-zero-eq}
G(r,\rho) = 0 \text{\,\, and \,\,} \frac{\partial G}{\partial \rho}(r,\rho) = 0.
\end{align}
To simplify the analysis, we introduce the auxiliary variables
\begin{equation*}\label{eq:transform-ar2xy}
x = \frac{1 - \rho}{1 + 2\rho},\quad  
y = \frac{1 - r\rho}{1 + 2 r\rho}. 
\end{equation*}
The mapping \((r,\rho) \mapsto (x,y)\) is injective from \([0,1]\times[0,1]\) to itself. 
Under this transformation, the system \eqref{two-zero-eq} reduces to the equation-system \eqref{xy-eq}. See  \S \ref{sec-system} for the details. 
 
The remainder of the proof is devoted to establishing that \eqref{xy-eq} admits a \textbf{unique solution}  in the interior domain \((0,1)\times(0,1)\).

\subsubsection{Uniqueness of the solution}
We consider the one-parameter family of curves defined by  

\[
\begin{array}{rccc}
h:& (1,\infty)\times(0,1)&\longrightarrow &\mathbb{R}\times \mathbb{R} \\
& (p,x)&\longmapsto&
\Big(
\frac{1}{1+2x}\big(\frac{1+2x^{p}}{3}\big)^{\frac{1}{p}},\;
\frac{(1-x)(1-x^{p-1})}{1+2x^{p}}
\Big).
\end{array}
\]
Observe that the system \eqref{xy-eq} admits a unique solution in \((0,1)^2\) if and only if 
\begin{center}
{\it any two distinct curves $h(p,\cdot)$ and $h(q,\cdot)$ intersect exactly once.}
\end{center} Thus, the problem reduces to analyzing the intersection properties of these curves (see Figure \ref{intersection}).

\begin{figure}[htbp]
\centering
\includegraphics[width=0.5\textwidth]{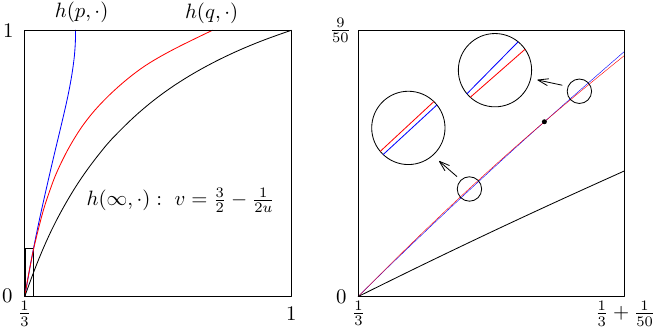}
\caption{Intersection of the curves \(h(p,\cdot)\) and \(h(q,\cdot)\)}
\label{intersection}
\end{figure}

%\[
%h_p(x):=h(p,x)=\Bigl(
%\frac{1}{1+2x}\big(\frac{1+2x^{p}}{3}\big)^{\frac{1}{p}},\;
%\frac{(1-x)(1-x^{p-1})}{1+2x^{p}}
%\Bigr),\quad x\in (0,1)
%\]
%and
%\[
%h_q(y):=h(q,y)\Bigl(
%\frac{1}{1+2y}\big(\frac{1+2y^{q}}{3}\big)^{\frac{1}{q}},\;
%\frac{(1-y)(1-y^{q-1})}{1+2y^{q}}
%\Bigr),\quad y\in (0,1)
%\]

At first glance, the curves \(h(p,\cdot)\) and \(h(q,\cdot)\) appear to be tangent at their endpoints, suggesting no intersection in the interior region. However, a closer examination, particularly after zooming in, reveals that they indeed cross {\bf exactly once} in \((\frac{1}{3},1)\times(0,1)\). To prove this phenomenon rigorously, we will employ {\bf symmetrization} and {\bf blowup analysis}.

%\[
%\begin{CD}
%\left(1, \infty \right) \times (0,1) @>h>> (0,\infty) \times (-\infty, 3/2) \\
%@A{\Phi}AA @VV{\Psi}V \\
%\left(-1, 1\right) \times (0, 1) @>> H> \mathbb{R} \times (-\infty, 3/2).
%\end{CD}
%\]

{\bf Symmetrization.} 
Exploiting inherent symmetries of the curve family $h$, there are suitable changes of coordinates $\Phi$ and $\Psi$ (see \eqref{phi} and \eqref{psi}) such that the curve family 
$$H=(H_1,H_2):=\Psi\circ h\circ\Phi:(-1,1)\times(0,1)
\longrightarrow \mathbb{R}\times \mathbb{R}$$ indexed by $\alpha\in (-1,1)$ possessing the symmetries 
$$H_1(-\alpha,t)=-H_1(\alpha,t),\quad H_2(-\alpha,t)=H_2(\alpha,t).$$
More precisely, we have  (see Figure \ref{fig-H})
$$H(\alpha,t)=
\left(
-\frac{\alpha}{2}\log \frac{1+2 t^2}{3} 
+\frac{1}{2}\log\frac{1+2 t^{1+\alpha}}{1+2 t^{1-\alpha}},\ \  
\frac{(1-t^{1-\alpha})(1-t^{1+\alpha})}{1+2 t^2}
\right).
$$

\begin{figure}[htbp]
\centering
\includegraphics[width=0.5\textwidth]{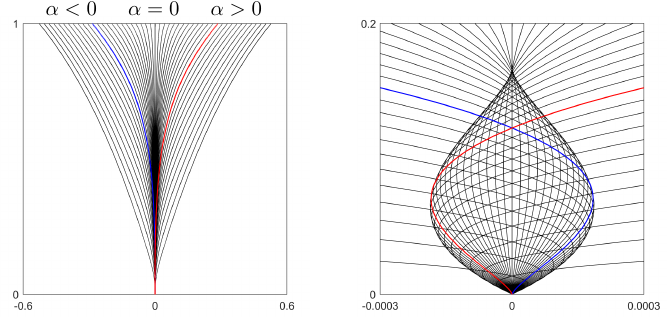}
\caption{The family of curves $H(\alpha,\cdot)$ and a close-up view near $(0,0)$
}
\label{fig-H}
\end{figure}

Coordinate transformations preserve intersection properties. Therefore, to establish that distinct curves $h(p,\cdot)$ have a unique intersection point, it suffices to prove the same for the transformed curves $H(\alpha,\cdot)$. 
The form of $H$ and its symmetries will greatly simplify the subsequent analysis.

\textbf{Blowup analysis.} 
To rigorously prove the unique intersection property of the new curve family $H(\alpha,\cdot)$, we require three properties: injectivity, separation, and transversality, where the latter two are  established via appropriate blowups. 

\begin{itemize}
\item 
\textbf{Injectivity.} 
The restricted map 
\(H: (-1,1)\times(0,\frac12) \to \mathbb{R}^2\) is injective (see Proposition~\ref{restrict-homeo}). 

\item 
\textbf{Separation after a first-order blowup.} 
Applying the following blowup to $H$: 
\begin{equation}
\label{blowup-b}
b(u,v)=\left(\frac{u}{v},\ v\right),
\end{equation}
 we find that the curvilinear triangle \(b\circ H((-1,1)\times(\frac{1}{2},1))\) is separated from the curve $b\circ H((-1,1)\times\{t_0\})$ for some $t_0\in(0,\frac{1}{2})$ (see Proposition \ref{separation}, as illustrated in Figure \ref{fig-blowup-b}). 

\item 
\textbf{Transversality after a $3/2$-order blowup.} 
Define another blowup 
\begin{equation}
\label{blowup-B}
B(u,v) = \left(\frac{u}{v^{3/2}},\ v^{1/2}\right).
\end{equation}
As shown in Figure \ref{fig-blowup-B}, each subarc $B\circ H(\alpha,[\frac{1}{4},1))$ resembles a straight line segment, and any two of them intersect transversally. 
For the details, see Proposition \ref{slope-monotone}. 
\end{itemize}

\begin{figure}[htbp]
\centering
\includegraphics[width=0.45\textwidth]{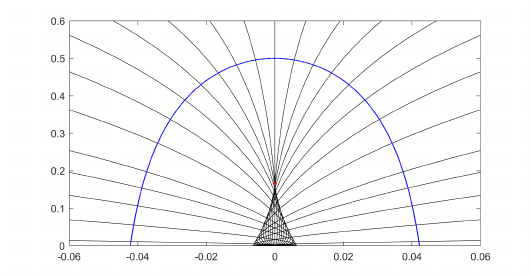}
\caption{Range of \(b\circ H\) 
(the black curves correspond to fixed values of \(\alpha\), the blue curve corresponds to \(t=\frac{1}{4}\), and the red point represents \(b\circ H(0,\frac{1}{2})\))
}
\label{fig-blowup-b}
\end{figure}

\begin{figure}[htbp]
\centering
\includegraphics[width=0.45\textwidth]{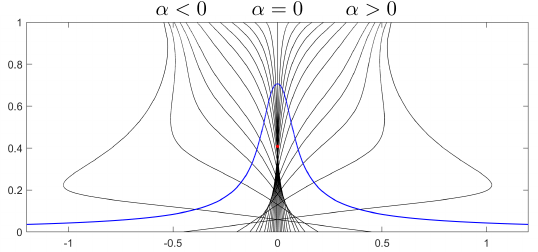}
\caption{Range of $B\circ H$
(the black curves correspond to fixed values of $\alpha$, the blue curve corresponds to $t=\frac{1}{4}$, and the red point represents $B\circ H(0,\frac{1}{2})$)
}
\label{fig-blowup-B}
\end{figure}

Combining these three properties, we show that distinct curves $H(\alpha,\cdot)$ (and consequently $h(p,\cdot)$) have a unique intersection point, 
which implies that the system equation \eqref{xy-eq} has a unique solution in the unit square $(0,1)^2$; see Proposition~\ref{proposition-unique}. 
This completes the proof of Theorem \ref{thm-main-thm-one-xy}.

\begin{remark*}
Readers familiar with the classification of singularities of smooth maps will recognize that the singularity at $(\alpha,t)=(0,\frac{1}{2})$ of $B\circ H$ (or $b\circ H$, $H$) is a Whitney pleat (see  \cite[\S1]{Arnold-Sing}). 
\end{remark*}

{\flushleft\bf Organization of the paper: }
The remaining part of the paper is organized as follows. In Section~\ref{S:preli}, we present the orthogonal basis method and the monotonicity game, which are fundamental to our subsequent analysis. Section \ref{S:parameter-reduction} is devoted to reducing the domain of the function coefficients in the hypercontractive inequalities from \(\mathbb{C}^3\) to the interval \([0,1]\). In Section \ref{S:establishing the e-s}, we first prove the existence of nontrivial critical extremizers for the \(\mathbb{Z}_3\) hypercontractive inequality and then, by applying the variational principle, obtain a characterizing system of equations. The proof of the main theorem is given in Section \ref{S:proof-thm1}, contingent on the uniqueness of the solution to this system, which is addressed in Section~ \ref{sec-unique}.  Section~\ref{sec-biased} is devoted to a sketched proof of Theorem~\ref{thm-biased-alpha}.  Lastly, the appendix provides elementary proofs for several inequalities and for the corollaries of the main result.

%\begin{remark}
%The authors of current paper learnt hypercontractive constant problems for the hypercontractive inequality over \(\mathbb{Z}_3\) in the above mentioned monograph \cite{JPPP} by Junge, Palazuelos, Parcet, and Perrin. Applications of \(r_{p,q}(\mathbb{Z}_3)\) go beyond the scope of this paper, and we encourage interested readers to explore its potential applications in other contexts.
%\end{remark}

{\flushleft\bf Acknowledgements.}  We would like to thank Prof. Javier Parcet and Dr. Gan Yao for helpful discussions. JC is supported by Postdoctoral Daily Funding (No.000141020162), SF is supported by NSFC (No.12331004 and No.12231013), YH is supported by NSFC (No.12131016), YQ is supported by NSFC (No.12595283 and No.12471145), ZW is supported by NSFC (No.12471116) and FRFCU (No.2025CDJ-IAIS YB004).

%Z. Wang is supported by NSF of China (No.12471116) and 2025CDJ-IAIS YB-004 (Chongqing University). 

\section{Preliminaries}\label{S:preli}

In this section, we introduce two elementary methods: the orthogonal basis method and the monotonicity game. Both approaches are used to establish criteria for determining the sign of a function over an interval by utilizing local information about the function---such as its values, derivatives at specific points, or convexity and monotonicity properties on an interval. These methods will be employed repeatedly in the proofs of main theorems. The detailed descriptions and application contexts of these two methods are provided below.

%These methods not only hold independent theoretical significance --- providing a concise and universal approach to addressing function sign determination problems --- but also play an indispensable role in the core arguments of this paper: 

\subsection{Orthogonal basis method}
This simple property of strictly convex/concave functions on an interval plays a crucial role in the proof of our monotonicity lemmas (\S\ref{proof:lemm-23} and \S\ref{sect:mono2}). Although it might be known in the literature, we have not been able to locate an explicit reference.

\begin{lemma}\label{lem-orthogonal}
Let \(A\) be a strictly convex or strictly concave  function on an interval \((a,b) \subset \mathbb{R}\),
and let \((y_1,y_2,y_3) \in \mathbb{R}^3\) be a nonzero vector with \(y_1+y_2+y_3=0\).
If \(x_1,x_2,x_3\in(a,b)\) are distinct and satisfy
$x_1y_1+x_2y_2+x_3y_3=0$, 
then
\[
\sum_{k=1}^{3}A(x_k)y_k\neq 0.
\]
\end{lemma}

\begin{proof}
Assume by contradiction that
$
\sum_{k=1}^{3}A(x_k)y_k=0.
$
That is, the vector \((A(x_1),A(x_2),A(x_3))  \) is orthogonal to the nonzero vector \((y_1,y_2,y_3)\).
By assumption, the vectors \((1,1,1)\) and \((x_1,x_2,x_3)\) are  orthogonal to \((y_1,y_2,y_3)\).
 And, the vectors \((1,1,1)\) and \((x_1,x_2,x_3)\) are linearly independent since \(x_1,x_2,x_3\) are distinct. 
%Since \(x_1,x_2,x_3\) are distinct points in \((a,b)\), the vectors \((1,1,1)\) and \((x_1,x_2,x_3)\) are linearly independent, which are both orthogonal to $(y_1,y_2,y_3)$, by the assumptions.
Thus there exist \(s,t\in\mathbb{R}\) such that 
$$
(A(x_1), A(x_2), A(x_3)) = s (1, 1, 1) + t (x_1, x_2, x_3) = (s + tx_1, s+tx_2, s+ tx_3). 
$$
%$$
%\left(\begin{array}{c}
%A(x_1)\\[1mm]
%A(x_2)\\[1mm]
%A(x_3)
%\end{array}\right)
%= s\left(\begin{array}{c}
%1\\[1mm]
%1\\[1mm]
%1
%\end{array}\right)
%+ t\left(\begin{array}{c}
%x_1\\[1mm]
%x_2\\[1mm]
%x_3
%\end{array}\right).
%$$
In other words, the affine function \(L(x)=s+tx\) coincides with \(A(x)\) at three distinct points \(x_1,x_2,x_3\).
This contradicts strict convexity or strict concavity of \(A\) and thus finishes the proof.
\end{proof}

\begin{remark*}
In our proofs of the monotonicity lemmas, the relevant setting satisfies a stronger condition: 
the vectors
\[
(1,1,1),\quad (x_1,x_2,x_3),\quad (y_1,y_2,y_3)
\]
form an orthogonal basis of \(\mathbb{R}^3\).
\end{remark*}

\subsection{Monotonicity game}
\label{S:monotonicity}

Our proof for the unique intersection property for the curve family $H(\alpha,\cdot)$ relies on establishing numerous inequalities. 
The analysis of sign patterns in derivatives naturally leads to a structured deductive puzzle like Sudoku, which we formalize as the \emph{monotonicity game}. 
The game's objective is to deduce the global sign of a function on an interval from limited information about the signs of its derivatives on the interval and at the endpoints. 

{\bf Rules and objective.} 
Let $f$ be a sufficiently smooth function on a closed interval $[a,b]$. Its $n$-th derivative is denoted by $f^{(n)}$, with the convention that $f^{(0)}=f$. 
Given information $(N,E,I)$: 
\begin{align*}
N &\in \{1,2,3,\dots\},\\
E &= (E_0,\dots,E_{N-1}),\ E_n\in \{+, -, 0,/\} \times \{+, -, 0,/\},\\
I &\in \{+,-\}. 
\end{align*}
The first (resp. second) coordinate of $E_n$ indicates whether $f^{(n)}(a)$ (resp. $f^{(n)}(b)$) is positive, negative, zero, or unknown. 
Here, $I$ indicates whether $f^{(N)}|_{(a,b)}$ is positive or negative. 
The player's goal is to determine the sign $$J\in\{+,-\}$$ of the original function $f$ on the interval $(a,b)$. 

As in a well-posed Sudoku puzzle, the given information $(N,E,I)$ must be appropriate to uniquely determine $J$. 
This logic is illustrated by the basic instance 
$$(N,E,I)=(2,\ ((0,0),(/,/)),\ +),$$ 
for which one can conclude $J=-$, as shown in Figure \ref{monotonicity-game}. 

\begin{figure}[htbp]
\centering
\includegraphics[width=0.4\textwidth]{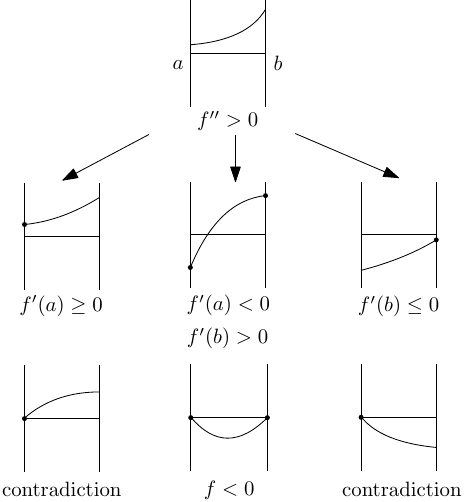}
\caption{Solving a monotonicity game instance through case analysis}
\label{monotonicity-game}
\end{figure}

The following examples are derived from specific analytical contexts within this work, showing the translation of derivative conditions into the game's formal language. While the solution to Example \ref{example-monotone-game} is provided in Figure \ref{derivatives-g}, the solutions to the remaining examples are left as exercises for the reader. 

\begin{example}
[From \eqref{monotone-game-u-1}]
\label{example-monotone-game}
Determine $J$ for 
$$(N,E,I)=(4,\ ((0,0), (/,0), (0,/), (/,/)),\ -).$$
\end{example}

\begin{example}
[From \eqref{monotone-game-u-2}]
\label{ex-monotone-game-u-2}
Determine $J$ for 
$$(N,E,I)=(4,\ ((0,0), (0,/), (0,/), (/,/)),\ -).$$
\end{example}

\begin{example} 
[From the proof of \eqref{P-J}] 
\label{ex-section-P-J}
Determine $J$ for 
$$(N,E,I)=(5,\ ((+,0), (0,-), (/,/),(0,/), (/,/)),\ +).$$
\end{example}

\begin{example}
[From the proof of \eqref{I-W0-P} and \eqref{I-W1-P}]
\label{ex-section-I-W-P}
Determine $J$ for 
$$(N,E,I)=(3,\ ((-,0), (0,/), (/,/)),\ +).$$
\end{example}

\section{Parameter reduction: from $\mathbb{C}^3$ to $[0,1]$}
\label{S:parameter-reduction}

Let \(1 < p < q < \infty\). 
Our goal is to determine the optimal  \(r_{p,q}\) such that for \(r \in (0, r_{p,q}]\), 
\begin{align}\label{eqn-Tt-f-pq}
\|T_r f\|_q \leq \|f\|_p \quad \text{for all } f \in L^p(\mathbb{Z}_3) = L^p(\mathbb{Z}_3, \mathfrak{m}).
\end{align}

It seems that $0< r_{p,q}<\sqrt{(p-1)/(q-1)}$ is well-known. For completeness, we state it as follows. 

\begin{lemma}
\label{lem-rough}
For $1<p<q<\infty$, we have 
$r_{p,q} < \sqrt{(p-1)/(q-1)}$.  
\end{lemma}

For \(1 < p<q < \infty\), let \(p^* = p/(p-1), q^* = q/(q-1)\) denote the conjugate exponents.
By duality, one always has $
r_{p,q}(\mathbb{Z}_3) = r_{q^*,p^*}(\mathbb{Z}_3)$. Hence
it suffices to determine \(r_{p,q}(\mathbb{Z}_3)\) for the cases
$$
% \label{eqn-pq-range}
1 < p < q <2\ \ \text{and}\ \  1 < p \leq 2 < q<\infty.
$$

\begin{proposition}
\label{thm-zhongyaoi}
Let \(1 < p < q < 2\) or \(1 < p \leq 2 < q<\infty\) and  \(r \in (0, \sqrt{(p-1)/(q-1)}]\). Then the inequality 
\eqref{eqn-Tt-f-pq} holds for all functions on $\mathbb{Z}_3$ if and only if it holds for all functions of the form
$$
f_\rho = 1 + \rho \chi + \rho \overline{\chi} \quad \text{with } \rho \in [0,1].
$$
\end{proposition}

The proof of Proposition~\ref{thm-zhongyaoi} is illustrated in Figure~\ref{figtriangle00}: the parameter space of the functions is reduced from $\mathbb{C}^3$
to the triangle $\Delta$, and then further reduced from $\Delta$ to the interval $[0,1]$.
%The proof of Proposition \ref{thm-zhongyaoi} is illustrated by Figure \ref{figtriangle00}. Roughly speaking, the parameters in the functions are reduced from $\mathbb{C}^3$  to the triangles $\Delta$,  and then $\Delta$ to the interval $[0,1]$. 
\subsection{Proof of Lemma \ref{lem-rough}}
Consider the real-valued $f_a=1+a\chi+a\overline{\chi}$ 
with $a\in [-\frac{1}{2},1]$.  As $a\to 0$, 
\begin{equation}\label{eq:taylor}
\|f_a\|_p-\|T_r f_a\|_q =   (p-1 - (q-1)r^2)a^2 
+ \frac{(p-1)(p-2) - (q-1)(q-2) r^3}{3} a^3 + o(a^3).
\end{equation}
If $r > \sqrt{(p-1)/(q-1)}$, then $p-1 - (q-1)r^2 < 0$ and  there exists a constant $\delta>0$ such that $\|f_a\|_p-\|T_r f_a\|_q<0$ for any $0<|a|<\delta$. 
If $r = \sqrt{(p-1)/(q-1)}$, then  
\[
\|f_a\|_p-\|T_r f_a\|_q=\frac{p-1}{3}\left(p-2-(q-2)\sqrt{\frac{p-1}{q-1}}\right)a^3+o(a^3). 
\]
Since the coefficient of $a^3$ is nonzero, $\|f_a\|_p-\|T_r f_a\|_q<0$ for some small $a$. It follows that $r_{p,q} < \sqrt{(p-1)/(q-1)}$, as desired.

\subsection{Reduction from $\mathbb{C}^3$ to $\Delta$ via symmetry}
%The following lemma is standard. 
%\begin{lemma}\label{lemma-convolution}
%For any function $f$ on $\mathbb{Z}_3$, we have
%$
%\|T_r f\|_p\leq \|T_r|f|\|_p. 
%$ Consequently, the inequality \eqref{eqn-Tt-f-pq} holds if and only if
%\begin{align*}\label{eqn-Tt-f-pq-positive}
%\|T_r f\|_q \leq \|f\|_p \quad 
%\text{for all non-negative } f \in L^p(\mathbb{Z}_3).
%\end{align*}
%\end{lemma}
%
%\begin{proof}
%Let $P_r = 1 + r\chi + r\overline{\chi}$. Then $T_r f = P_r * f$, where $*$ denotes  the convolution operation. Noticing that $P_r \geq 0$ and $\int_{\mathbb{Z}_3} P_r(x) \, \mathfrak{m}(\mathrm{d}x) = 1$, we obtain
%\begin{align*}
%|T_r (f)(x)| &= \left|\int_{\mathbb{Z}_3} P_r(x-y)f(y) \,  \mathfrak{m}(\mathrm{d}y)\right|\leq \int_{\mathbb{Z}_3} P_r(x-y)|f(y)| \,  \mathfrak{m}(\mathrm{d}y) = T_r(|f|)(x)
%\end{align*}
%and thus $\|T_r f\|_p\leq \|T_r|f|\|_p$. The last assertion follows by noticing that $\| f\|_p= \| |f|\|_p$. 
%\end{proof}
%

 Any function $f$ on $\mathbb{Z}_3$  has Fourier expansion:
\[
f=a_0+a_1 \chi+a_2 \chi^2  = a_0+a_1 \chi+a_2 \overline{\chi} \quad  \text{with $a_0, a_1, a_2 \in \mathbb{C}$. }
\]

%Therefore, by Lemma \ref{lemma-convolution} 
Note that    $T_r f = P_r * f$ with Poisson kernel $P_r = 1 + r\chi + r\overline{\chi}\ge 0$ for $0\le r\le 1$ and   $\|T_r f\|_p\leq \|T_r|f|\|_p$. Therefore,  the inequality \eqref{eqn-Tt-f-pq} holds if and only if
\begin{align*}\label{eqn-Tt-f-pq-positive}
\|T_r f\|_q \leq \|f\|_p \quad 
\text{for all non-negative } f \in L^p(\mathbb{Z}_3).
\end{align*} 
By homogeneity,  to study the inequality  \eqref{eqn-Tt-f-pq}, we may assume without loss of generality that 
\[
f = 1 + a\chi + b\overline{\chi} \geq 0.
\]
 It follows that \(b = \overline{a}\), hence \(f = 1 + a\chi + \overline{a}\,\overline{\chi}\). We note that \(f = 1 + a\chi + \overline{a}\,\overline{\chi} \geq 0\) if and only if
$$
\left\{
\begin{aligned}
&1 + a + \overline{a} \geq 0, \\
&1 + ae^{i\frac{2\pi}{3}} + \overline{a}e^{-i\frac{2\pi}{3}} \geq 0, \\
&1 + ae^{-i\frac{2\pi}{3}} + \overline{a}e^{i\frac{2\pi}{3}} \geq 0.
\end{aligned}
\right.
$$
Equivalently, $f\geq 0$ if and only if  the complex number \(a\) lies in the triangular region \(\Delta_0\) with vertices at \(1\), \(e^{i 2\pi/3}\), and \(e^{-i 2\pi/3}\); see Figure \ref{figtriangle00}.
Hence, the inequality \eqref{eqn-Tt-f-pq} holds if and only if it holds for all functions of the form
$$f_a = 1 + a\chi + \overline{a}\,\overline{\chi} \quad \text{with } a\in\Delta_0.$$

\begin{lemma} \label{lem:symmetry}
For $p>1$ and $f_a=1+a\chi+\overline{a}\,\overline{\chi}$ with $a\in\mathbb{C}$, we have 
$
\|f_a\|_p = \|f_{\overline{a}}\|_p = \|f_{a e^{i2\pi/3}}\|_p.
$
\end{lemma}

\begin{proof}
By  definition,
\[
3\|f_a\|_p^p = |1+2\operatorname{Re}(a)|^p
+ |1+2\operatorname{Re}(a e^{i 2\pi/3})|^p
+ |1+2\operatorname{Re}(a e^{-i 2\pi/3})|^p.
\]
Since $\operatorname{Re}(\overline{a})=\operatorname{Re}(a)$, replacing $a$ with $\overline{a}$ does not change the three summands.  Moreover, multiplying $a$ by $e^{i2\pi/3}$ cyclically permutes the cubic roots of unity $1$, $e^{i2\pi/3}$, $e^{-i2\pi/3}$, which simply permutes the three terms in the sum and leaves the total sum unchanged.
\end{proof}

Recall that $T_r f_a = f_{ra}$. Using the symmetry in Lemma \ref{lem:symmetry}, we obtain the following reduction.

\begin{proposition}\label{proposition-small-triangle}
The inequality \eqref{eqn-Tt-f-pq} holds for all functions on $\mathbb{Z}_3$ if and only if it holds for all
\[
f_a = 1 + a\chi + \overline{a}\,\overline{\chi}, \quad a \in \Delta,
\]
where $\Delta \subset \mathbb{C}$  is the closed triangle with three vertices given by $0, 1, \frac{1}{2}e^{i \frac{\pi}{3}} \in \mathbb{C}$ (see Figure \ref{figtriangle00}). 
\end{proposition}
%%%%%%%%%%%%%%%%%%%%%%%%%%%%%%%%%%%%%%%%%%%%%%%%%%%%%%%%%%%%%%%%%%%%%%%%%%%%%%%%%%%%%%%%%%%%%%%%%%%%%%%%%%%%%%%%%%%%%%%%%%%%%%%%%%%%%%%%%%%%%%%%%%%%%%%%%%%%%%%%%%
\subsection{Reduction from \(\Delta\) to \([0,1]\) via monotonicity}

%For any $1<p,q<\infty$, note that
%\begin{align}\label{eqn-dual-index-r}
%r_{p,q}(\mathbb{Z}_3) = r_{q^*,p^*}(\mathbb{Z}_3),
%\end{align}
%it suffices to determine \(r_{p,q}(\mathbb{Z}_3)\) for the cases
%\begin{align}\label{eqn-pq-range}
%1 < p < q \leq 2 \quad \text{or} \quad 1 < p \leq 2 < q<\infty.
%\end{align}
%The main result in this section is the following result.

%The first contribution of this paper is to transform the hypercontractive inequality \eqref{eqn-Tt-f-pq-positive} into a more computationally tractable form, with the precise result stated as follows.

%To reduce $\Delta_0$ to the unit interval $[0,1]$, we present several auxiliary technical results in this subsection. 
%The main results here are Lemmas 1, 2, and 3. 
%These proofs are tedious, and for the sake of quickly reaching the proof of Proposition \ref{thm-zhongyaoi}, readers may take these results for granted on first reading and proceed directly to the next subsection.

Write the triangle $\Delta$  in polar coordinates:
$$
\Delta=\bigg\{\rho e^{i\theta} ~\bigg|~ \theta \in \Big[0,\frac{\pi}{3}\Big],\ 
\rho\in\Big[0,\frac{1}{2\sin(\theta+\frac{\pi}{6})}\Big]\bigg.\bigg\}.
$$
For $p>1$ and $(\rho,\theta)$ with $\rho e^{i\theta}\in \Delta$, define 
\begin{equation}\label{eq-Fq}
\begin{aligned}
F(p,\rho,\theta)= 3\|f_{\rho e^{i\theta}}\|_p^p= \big(1+2\rho\cos\theta\big)^p 
+ \big(1+2\rho\cos(\theta+\tfrac{2\pi}{3})\big)^p 
+ \big(1+2\rho\cos(\theta-\tfrac{2\pi}{3})\big)^p.
\end{aligned}
\end{equation}
Consequently, for $q>1$ and $0<r<1$, we have 
\[
\|T_r f_{\rho e^{i\theta}}\|_q 
= \|f_{r\rho e^{i\theta}}\|_q
= \left(\frac{F(q,r\rho,\theta)}{3}\right)^{\frac{1}{q}}.
\]
Hence, the defect function \eqref{eq:defectfunction} can be written as 
\begin{align}\label{eqn-df-H}
G(p,q,r, \rho, \theta) =\left(\frac{F(p,\rho,\theta)}{3}\right)^{\frac{1}{p}}
   -\left(\frac{F(q,r\rho,\theta)}{3}\right)^{\frac{1}{q}}.
\end{align}
%For convenience in the subsequent analysis, we write \( G(p,q,\rho,\theta, r) \) in place of \( G(p,q,\rho e^{i\theta}, r) \).
For completing the proof of Proposition \ref{thm-zhongyaoi}, we need the following monotonicity lemmas.

\begin{lemma}
\label{lemma-reduce-p-q-1}
Let \(1 < p \leq 2 < q < \infty\) and \(0<r\leq 1\).
For any \((\rho, \theta)\) with $\rho e^{i\theta}\in\Delta$, the function \(G(p,q,r, \rho, \theta)\) is  strictly increasing in \(\theta\). Consequently,
$$ 
%\label{eq:hpq12}
G(p,q,r, \rho, \theta) \geq G(p,q,r, \rho, 0).
$$
\end{lemma}

\begin{lemma}
\label{lemma-reduce-p-q-2}
Let \(1 < p < q < 2\) and \(0 < r \leq \sqrt{(p-1)/(q-1)}\).  
For any \((\rho, \theta)\) with $\rho e^{i\theta}\in\Delta$, the function \(G(p,q,r, \rho, \theta)\) is  strictly increasing in \(\theta\). Therefore,
$$
% \label{eq:hpqsmall2}
G(p,q,r, \rho, \theta) \geq G(p,q,r, \rho, 0).$$
\end{lemma}

\begin{remark*} While Lemma~\ref{lemma-reduce-p-q-1} holds for all $r\in (0,1]$, 
the condition $0 < r \leq \sqrt{(p-1)/(q-1)}$ is essential for Lemma~\ref{lemma-reduce-p-q-2}. Note also that, in general, if   $2\leq p<q<\infty$, we do not have a similar monotonicity result as in Lemmas~\ref{lemma-reduce-p-q-1} and \ref{lemma-reduce-p-q-2}. 
\end{remark*}
%\begin{lemma}\label{lemma-H-monotonicity}
%Let $1<p,q<\infty$ and $0<r\leq\sqrt{\frac{p-1}{q-1}}$. 
%Assume either
%\begin{enumerate}
%    \item $1<p\leq2\leq q<\infty$, or
%    \item $1<p\leq q\leq2$.
%\end{enumerate}
%Then for any $(\rho,\theta)\in\Delta$, we have
%\[
%H(p,q,\rho,\theta,r)\geq H(p,q,\rho,0,r).
%\]
%That is, the function $\theta\mapsto H(p,q,\rho,\theta,r)$ attains its minimum at $\theta=0$ on $[0,\pi/3]$.
%\end{lemma}

%\begin{lemma}\label{lemma-reduce-p-q-1}
%Let \(1 < p \leq 2 \leq q < \infty\).  
%For any \((\rho, \theta) \in \Delta\), we have $\frac{\partial G}{\partial\theta}(p,q,r, \rho, \theta)>0.$ Hence
%\begin{align}\label{eq:hpq12}
%G(p,q,r, \rho, \theta) \geq G(p,q,r, \rho, 0).
%\end{align}
%\end{lemma}
%
%\begin{lemma}\label{lemma-reduce-p-q-2}
%Let \(1 < p \leq q \leq 2\) and \(0 < r \leq \sqrt{\frac{p-1}{q-1}}\).  
%For any \((\rho, \theta) \in \Delta\), $\frac{\partial G}{\partial\theta}(p,q,r, \rho, \theta)>0.$  we have
%\begin{align}\label{eq:hpqsmall2}
%G(p,q,r, \rho, \theta) \geq G(p,q,r, \rho, 0)
%\end{align}
%\end{lemma}

%\subsection{The proofs of Monotonicity lemmas}\label{two-key-lemmas}
%These proofs require a careful and detailed analysis, and readers may take these results for granted on first reading and proceed directly to the next section.

\subsection{Proof of Lemma \ref{lemma-reduce-p-q-1}}\label{proof:lemm-23}
Recall from \eqref{eq-Fq} that
\[
F(p, \rho, \theta) = \big(1+2\rho\cos\theta\big)^p 
+ \big(1+2\rho\cos(\theta+\tfrac{2\pi}{3})\big)^p 
+ \big(1+2\rho\cos(\theta-\tfrac{2\pi}{3})\big)^p,  
\]
where $p > 1$ and $(\rho,\theta)$ with $\rho e^{i\theta}\in\Delta$. 
Our goal is to estimate the partial derivative
$\frac{\partial F}{\partial\theta}(p,\rho,\theta)$. 
Using the orthonormal basis method,
we show that $\frac{\partial F}{\partial\theta}(p,\rho,\theta)$
does not vanish for $\rho e^{i\theta}\in\operatorname{int}(\Delta)$
whenever $p\neq 2$.
We then evaluate $\frac{\partial F}{\partial\theta}(p,\rho,\theta)$ at a special choice of $(\rho,\theta)$
to determine its sign.

\begin{lemma}\label{lem-decreasing-zero}
For  $(\rho,\theta)$ with $\rho e^{i\theta}\in\operatorname{int}(\Delta)$, we have 
\begin{enumerate}
    \item \(\frac{\partial F}{\partial\theta}(p,\rho,\theta) > 0\) if \(1 < p < 2\);
    \vspace{1mm}
    \item \(\frac{\partial F}{\partial\theta}(p,\rho,\theta) < 0\) if \(p > 2\). 
\end{enumerate}
\end{lemma}

\begin{proof} 
 For \(p >1\) and \(x\in(-1,+\infty)\), define
\[
A_p(x) = (1 + x)^{p-1}.
\]
Clearly, \(A_p(x)\) is strictly convex/concave on \((-1,+\infty)\) if $p\neq 2$.
By direct computation, we obtain
\[
\frac{\partial F}{\partial\theta}(p,\rho,\theta)=-2p\rho \sum_{\eta \in \{\theta,\,\theta\pm\frac{2\pi}{3}\}}
A_p( 2\rho\cos\eta) \sin\eta.
\]

Observe that the columns of the matrix
\[
\left(\begin{array}{ccc}
1 & 2\rho\cos(\theta) & \sin(\theta)  \\[1mm]
1 & 2\rho\cos(\theta+\frac{2\pi}{3}) & \sin(\theta+\frac{2\pi}{3})  \\[1mm]
1 & 2\rho\cos(\theta-\frac{2\pi}{3}) & \sin(\theta-\frac{2\pi}{3}) 
\end{array}\right)
\]
form an orthogonal basis of $\mathbb{R}^3$ since $\rho\in(0,1)$.
Because \(\rho e^{i\theta}\in \operatorname{int}(\Delta)\), we have
\[
2\rho\cos(\theta),\ 
2\rho\cos(\theta+\tfrac{2\pi}{3}),\ 
2\rho\cos(\theta-\tfrac{2\pi}{3})\in (-1,2)
\]
and $\theta\in (0,\pi/3)$, which implies that 
\[
\cos(\theta) \in (1/2,1),\quad
\cos(\theta+\tfrac{2\pi}{3})\in (-1,-1/2),\quad
\cos(\theta-\tfrac{2\pi}{3})\in (-1/2,1/2)
\]
are distinct. 
By Lemma~\ref{lem-orthogonal}, 
$\frac{\partial F}{\partial\theta}(p,\rho,\theta) \neq 0$
for $\rho e^{i\theta}\in\operatorname{int}(\Delta)$ and $p\neq 2$.

%Note that $F(2,\rho,\theta) = 3(1+2\rho^2)$ is constant in $\theta$,
%so $\frac{\partial F}{\partial\theta}(2,\rho,\theta) = 0$.
%Moreover, $\frac{\partial F}{\partial\theta}(p,\rho,\theta)$ depends continuously on $p$.
By direct computation,
$$
\frac{\partial F}{\partial\theta}\left(\frac32,\frac12,\frac\pi6\right) =\frac{6-3\sqrt{3}}{4}> 0 
\text{\ \ and\ \ }
\frac{\partial F}{\partial\theta}\left(3,\frac12,\frac\pi6\right) =-\frac{9}{4} < 0.
$$
Since $\frac{\partial F}{\partial\theta}$ is continuous and does not vanish for $p\neq 2$,
it follows that 
$$\text{
$\frac{\partial F}{\partial\theta}(p,\rho,\theta) > 0$ for $1 < p < 2$
\text{\ \ and\ \ }
$\frac{\partial F}{\partial\theta}(p,\rho,\theta) < 0$ for $p > 2$. 
}$$
This completes the proof. 
\end{proof}

%
%By Lemma~\ref{lem-orthogonal}, we conclude that \(\frac{\partial F}{\partial\theta}(p,\rho,\theta) \neq 0\) for $\rho e^{i\theta}\in\operatorname{int}(\Delta)$ when $p\neq 2$.
%
%
%
%Note that $F(2, \rho,\theta)= 3(1+2\rho^2)$ is constant with respect to $\theta$. 
%Hence \(\frac{\partial F}{\partial\theta}(2,\rho,\theta) =0.\) 
%By direct computation, 
%\[\frac{\partial F}{\partial\theta}(3/2,1/2,\pi/6)=??>0 \]
%and 
%\[\frac{\partial F}{\partial\theta}(3,1/2,\pi/6)=??<0\]
%It follows that \(\frac{\partial F}{\partial\theta}(p,\rho,\theta) > 0\) for  \(1 < p < 2\)
%and   \(\frac{\partial F}{\partial\theta}(p,\rho,\theta) <0\) for  \(p>2\).

We are now ready to prove the monotonicity of the defect function \(G\) in $\theta$ for $1<p\leq 2<q<\infty$.

\begin{proof}[Proof of Lemma \ref{lemma-reduce-p-q-1}]
Fix $\rho\in (0,1)$. 
For $1<p\leq 2< q<\infty$, Lemma~\ref{lem-decreasing-zero} implies that \(F(p, \rho, \theta)\) is increasing in \(\theta\), while \(F(q, r\rho, \theta)\) is strictly decreasing in \(\theta\).  
It follows that the defect function  
\[
G(p,q,r, \rho, \theta) =\left(\frac{F(p,\rho,\theta)}{3}\right)^{\frac{1}{p}}
   -\left(\frac{F(q,r\rho,\theta)}{3}\right)^{\frac{1}{q}}
\]  
is  strictly increasing in \(\theta\).
\end{proof}

\subsection{Proof of Lemma \ref{lemma-reduce-p-q-2}} \label{sect:mono2}
 For the case $1<p<q<2$, both $F(p,\rho,\theta)$ and $F(q,r\rho,\theta)$ are increasing in $\theta$, so the monotonicity comparison in Lemma \ref{lemma-reduce-p-q-1} no longer applies.
In this regime, one cannot rely on opposite monotonicity to conclude that $G$ is increasing.
Instead, to establish the monotonicity of $G$ with respect to $\theta$, we estimate the partial derivative $\frac{\partial G}{\partial\theta}$ directly.

Let $(\rho,\theta)$ be such that $\rho e^{i\theta}\in \operatorname{int}(\Delta)$.
By a direct computation, we obtain
\begin{align*}
3\frac{\partial G}{\partial\theta}
&=  \left(\frac{F(p,\rho,\theta)}{3}\right)^{\frac{1}{p}-1}
\frac{\partial F}{\partial\theta}(p,\rho,\theta)
- \left(\frac{F(q,r\rho,\theta)}{3}\right)^{\frac{1}{q}-1}
\frac{\partial F}{\partial\theta}(q,r\rho,\theta).
\end{align*}
It follows from Lemma~\ref{lem-decreasing-zero} that
\[
\frac{\partial F}{\partial\theta}(p,\rho,\theta)>0
\quad\text{and}\quad
\frac{\partial F}{\partial\theta}(q,r\rho,\theta)>0.
\]

The desired inequality  $\frac{\partial G}{\partial\theta}>0$  follows immediately once we establish  the following two estimates:
\begin{align}
\left(\frac{F(p,\rho,\theta)}{3}\right)^{\frac{1}{p}-1}
&> r\left(\frac{F(q,r\rho,\theta)}{3}\right)^{\frac{1}{q}-1},
\label{eq:term1-est} \\[4pt]
\frac{\partial F}{\partial\theta}(p,\rho,\theta)
&> \frac{1}{r}\,\frac{\partial F}{\partial\theta}(q,r\rho,\theta).
\label{eq:term2-est}
\end{align}
We emphasize that  the second inequality requires the key condition $r\leq\sqrt{(p-1)/(q-1)}$.

%Thus a direct termwise comparison fails to yield $\partial G/\partial\theta>0$, and we instead establish the following refined estimates:
%\begin{align}
%\left(\frac{F(p,\rho,\theta)}{3}\right)^{\frac{1}{p}-1}
%&> r\left(\frac{F(q,r\rho,\theta)}{3}\right)^{\frac{1}{q}-1},
%\label{eq:term1-est} \\[4pt]
%\frac{\partial F}{\partial\theta}(p,\rho,\theta)
%&> \frac{1}{r}\,\frac{\partial F}{\partial\theta}(q,r\rho,\theta).
%\label{eq:term2-est}
%\end{align}
%Once these inequalities are verified, they immediately imply that
%\[
%\frac{\partial G}{\partial\theta}>0,
%\]
%as desired.

%For simplification, we now define two auxiliary functions
%\[
%U(q,\rho,\theta,r)=r\left(\frac{F(q,r\rho,\theta)}{3}\right)^{\frac{1}{q}-1}
%\]
%and
%\[
%V(q,\rho,\theta,r)= \frac{1}{r}\,\frac{\partial F}{\partial\theta}(q,r\rho,\theta)= \sum_{\eta}(1+2r\rho\cos\eta)^{q-1} (-\sin\eta).
%\]
%
%

\begin{proof}
[Proof of \eqref{eq:term1-est}]
By $1<p<q$ and $\|f_a\|_q > \|f_a\|_p > \|f_a\|_1 = 1$, we have 
\[
\left(\frac{F(p,\rho,\theta)}{3}\right)^{\frac{1}{p}-1}  = \| f_a\|_p^{1-p} > \| f_a\|_q^{1-p} 
>  \| f_a\|_{q}^{1-q} = 
\left(\frac{F(q,\rho,\theta)}{3}\right)^{\frac{1}{q}-1}.  
\]
Now to prove the inequality \eqref{eq:term1-est}, it suffice to show that   
\begin{align}\label{eq:rfmono}
\left(\frac{F(q,\rho,\theta)}{3}\right)^{\frac{1}{q}-1}\geq r\left(\frac{F(q,r\rho,\theta)}{3}\right)^{\frac{1}{q}-1}=:U(q,r,\rho,\theta).
\end{align}
That is, 
\[U(q,1,\rho,\theta)\geq U(q,r,\rho,\theta).\]
Differentiating $U$ with respect to $r$ yields
\begin{align*}
&\left(\frac{F(q,r\rho,\theta)}{3}\right)^{2-\frac{1}{q}}\frac{\partial U}{\partial r}
= F(q,r\rho,\theta)
     + r\left(\frac{1}{q} - 1\right)\frac{\partial F(q,r\rho,\theta)}{\partial r}\\
&= \sum_{\eta \in \{\theta,\,\theta\pm\frac{2\pi}{3}\}} \Big( (1 + 2r\rho\cos\eta)^q
   + r(1 - q)(1 + 2r\rho\cos\eta)^{q-1} 2\rho\cos\eta \Big)\\
&= \sum_{\eta \in \{\theta,\,\theta\pm\frac{2\pi}{3}\}} (1 + 2r\rho\cos\eta)^{q-1}
   \Big(q - 1 + (2 - q)(1 + 2r\rho\cos\eta)\Big)\geq0,  
\end{align*}
where the last inequality follows from $1<q<2$ and $1 + 2r\rho\cos\eta\geq0$ for each $\eta \in \{\theta,\,\theta\pm\frac{2\pi}{3}\}$. 
Then $\frac{\partial U}{\partial r}\geq 0$ implies that $U$ is increasing in $r$. Therefore \eqref{eq:rfmono} holds, and \eqref{eq:term1-est} follows.
\end{proof}

\begin{remark*}
The inequality \eqref{eq:rfmono} combined with  Lemma~\ref{lem:symmetry} implies the 
 reverse estimation:
\[
\|T_{r}f\|_q \geq r^{\frac{1}{q-1}}\|f\|_q, \quad\forall\, f\colon \mathbb{Z}_3\to\mathbb{R}_{\geq 0}.
\]
\end{remark*}

\begin{proof}
[Proof of \eqref{eq:term2-est}] 
To prove \eqref{eq:term2-est}, we introduce the function
\[
V(q,r,\rho,\theta)
:= \frac{1}{2q r \rho}\,\frac{\partial F}{\partial\theta}(q,r\rho,\theta)
= \sum_{\eta \in \{\theta,\,\theta\pm\frac{2\pi}{3}\}}  (1+2r\rho\cos\eta)^{q-1} (-\sin\eta).
\]
The inequality \eqref{eq:term2-est} is equivalent to
\[
V(p,1,\rho,\theta) > V(q,r,\rho,\theta).
\]
We first show that $V(q,r,\rho,\theta)$ is strictly increasing in $r$ by means of the orthogonal basis method.

Differentiating $V$ with respect to $r$ yields
\begin{align*}
\frac{\partial V}{\partial r}(q,r,\rho,\theta)
&= \frac{1-q}{r} \sum_{\eta \in \{\theta,\,\theta\pm\frac{2\pi}{3}\}} (1 + 2r\rho\cos\eta)^{q-2} \cdot 2r \rho\cos\eta \cdot \sin\eta.
\end{align*}
Define $A_0(x)=(1+x)^{q-2}x$, which is strictly concave on $(-1,+\infty)$.
By Lemma~\ref{lem-orthogonal}, as in the proof of Lemma~\ref{lem-decreasing-zero},
we obtain
\[
\sum_{\eta \in \{\theta,\,\theta\pm\frac{2\pi}{3}\}} A_0(2r \rho\cos\eta) \sin\eta \neq 0
\]
for all $1<q<2$ and $\rho e^{i\theta} \in \operatorname{int}(\Delta)$.
Since
\begin{equation*}
\frac{\partial V}{\partial r}\bigg(\frac{3}{2},\frac{\sqrt{3}}{2},\frac{1}{2},\frac{\pi}{6}\bigg)
= \frac{\sqrt{3}}{4}\bigg(1 - \frac{\sqrt{7}}{7}\bigg) > 0,
\end{equation*}
we conclude that $\frac{\partial V}{\partial r} > 0$.

%Then, under the assumption $r\leq \sqrt{\frac{p-1}{q-1}}$, it suffices to verify the inequality
%\begin{equation}
%\label{G-theta-3}
%V\!\left( p,1,\rho,\theta \right)
%>
%V\!\left( q,\sqrt{\frac{p-1}{q-1}},\rho,\theta \right).
%\end{equation}
%Observe that
%\[
%V\!\left( p,1,\rho,\theta \right)
%=
%V\!\left( p,\sqrt{\frac{p-1}{p-1}},\rho,\theta \right).
%\]
%Therefore, in order to establish \eqref{G-theta-3}, it is enough to show that the function
%\[
%q\mapsto
%V\!\left( q,\sqrt{\frac{p-1}{q-1}},\rho,\theta \right)
%\]
%is strictly decreasing in $q$ for $p<q<2$.

Then, under the assumption $r\leq \sqrt{(p-1)/(q-1)}$, it suffices to verify the inequality
\begin{equation}
\label{G-theta-3}
V\!\left( p,1,\rho,\theta \right) = V\!\left( p,\sqrt{\frac{p-1}{p-1}},\rho,\theta \right)
>
V\!\left( q,\sqrt{\frac{p-1}{q-1}},\rho,\theta \right).
\end{equation}
In order to establish the inequality \eqref{G-theta-3}, it is enough to show that 
\[
W(p,q,\rho,\theta) := \frac{\partial }{\partial q} \left[ V\left( q,\sqrt{\frac{p-1}{q-1}},\rho,\theta \right)\right]<0.
\]
Expanding this partial derivative, we obtain
\begin{align*}
W=\sum_{\eta \in \{\theta,\,\theta\pm\frac{2\pi}{3}\}} A(q,2r\rho\cos\eta) (-\sin\eta), 
\end{align*}
where \( r = \sqrt{(p-1)/(q-1)} \) and 
\[
A(q,t):= (1+t)^{q-1}\left(\log(1+t)-\frac{t}{2(1+t)}\right),\ t>-1.
\]  
We next express $W$ in Cartesian coordinates. Let 
$x=r\rho\cos\theta$ and $y=r\rho\sin\theta$. 
Then $x+iy\in \operatorname{int}(\Delta)$, and 
\begin{align*}
&2r\rho W(p,q,\rho,\theta) = \sum_{\eta \in \{\theta,\,\theta\pm\frac{2\pi}{3}\}} A(q,2r\rho\cos\eta) (-2r\rho\sin\eta)\\
&= A(q,2x)(-2y) + A(q,-x-\sqrt{3}y)(-\sqrt{3}x+y) + A(q,-x+\sqrt{3}y)(\sqrt{3}x+y)\\
&=:\psi(q,x,y).
\end{align*}
Thus $W<0$ is equivalent to 
$$\psi(q,x,y)< 0.$$
The proof of this inequality is given in the appendix; see Lemma \ref{lemma-ineq-psi}. 
\end{proof}

\subsection{Proof of Proposition \ref{thm-zhongyaoi}}
Fix $r\in [0,\sqrt{(p-1)/(q-1)}]$. 
By Proposition \ref{proposition-small-triangle}, we only need to deal with $f_a$ with $a\in\Delta$. 
By Lemmas \ref{lemma-reduce-p-q-1} and \ref{lemma-reduce-p-q-2}, we have that 
$$
G(p,q,r,\rho,\theta)\geq 0,\ \forall\ \rho e^{i\theta}\in\Delta
\Longleftrightarrow
G(p,q,r,\rho,0)\geq 0,\ \forall\ \rho\in (0,1).$$
Recall the definition \eqref{eqn-df-H}, we get that 
$$\|T_r f_{a}\|_q \leq \|f_{a}\|_p,\  \forall\   a\in\Delta
\Longleftrightarrow 
\|T_r f_\rho\|_q \leq \|f_\rho\|_p,\  \forall\  \rho \in [0,1].$$
This completes the proof of Proposition \ref{thm-zhongyaoi}.

\section{Establishing  the equation-system via variational principle}\label{S:establishing the e-s}
%By Proposition \ref{thm-zhongyaoi}, the $(p,q)$-hypercontractivity problem for $\mathbb{Z}_3$ is reduced to finding the optimal $r_{p,q}(\mathbb{Z}_3)\in [0,1]$ such that 
%\begin{align}\label{eq-hyper}
%\left\|T_rf_a\right\|_q\leq \|f_a\|_p
%\end{align}
%for all functions $f_a$ of the form
%$$f_a=1+a\chi+ a  \overline{\chi}, \quad a\in [0,1].$$
 Proposition \ref{thm-zhongyaoi} reduces the hypercontractivity for $\mathbb{Z}_3$ to finding the optimal $r_{p,q}(\mathbb{Z}_3)$ such that 
$$%\begin{align}\label{eq-hyper} 
\left\|T_rf_\rho\right\|_q\leq \|f_\rho\|_p \quad \text{for all $f_\rho=1+\rho\chi+ \rho\overline{\chi}$ with $\rho\in [0,1]$ and all $0\le r\le r_{p,q}(\mathbb{Z}_3)$.}
$$%\end{align}

The main goal of this section is to show that the problem of determining $r_{p,q}(\mathbb{Z}_3)$ can be transformed into the problem of the existence of the unique solution in $(0,1)^2$ to  equation-system \eqref{xy-eq}.

We proceed in two steps: 
\begin{itemize}
\item  establishing the existence of a nontrivial critical extremizer,
\item  deriving equation-system \eqref{xy-eq} via variational principle.
\end{itemize}

\subsection{Existence of a nontrivial critical extremizer}
For given exponents \( 1<p<q<\infty \), recall the defect function  defined in \eqref{def-Grrho}: 
\[
G(r,\rho) = \|f_\rho\|_p - \|T_r f_\rho\|_q, \qquad (r,\rho) \in [0,1]\times[0,1].
\]
 
The existence of a nontrivial critical extremizer is given in the following

\begin{proposition}
\label{prop-nontrivial-extremizer}
Let \(1 < p < q < 2\) or \(1 < p \leq 2 <  q<\infty\). There exists  \( \rho_0 \in (0,1) \) such that \( G(r_{p,q}, \rho_0) = 0 \). 
\end{proposition}

To prove this, we first show that \( \rho = 1 \) is not a zero point of \( G(r_{p,q}, \cdot) \).

\begin{lemma}
\label{lemma-GAR-dayuling}
We have \( G(r_{p,q}, 1) > 0 \).
\end{lemma}

\begin{proof}
By the definition of $r_{p,q}$, we have $G(r_{p,q}, \rho) \ge 0$ for all $\rho \in [0,1]$.
In particular,  $G(r_{p,q}, 1) \geq 0$.
We now prove that $G(r_{p,q}, 1) \neq 0$. Suppose by contradiction that $G(r_{p,q}, 1) = 0$. Then the left derivative of $G(r_{p,q}, \cdot)$ at $\rho=1$ is non-positive. Therefore,
$$ G(r_{p,q},1) = 0 \text{\,\, and \,\,}
 \frac{\partial G}{\partial \rho}(r_{p,q},1^-) \leq 0.$$
Direct computation gives 
\begin{align*}
\left\{
\begin{aligned}
&3^{1+q-\frac{q}{p}} = (1+2r_{p,q})^q+2(1-r_{p,q})^q,\\
&3^{q-\frac{q}{p}} \leq r_{p,q} \big[(1+2  r_{p,q})^{q-1}-(1- r_{p,q})^{q-1}\big]. 
\end{aligned}
\right.
\end{align*}
It follows that 
\[
(1+2r_{p,q})^q+2(1-r_{p,q})^q = 3\cdot 3^{q-\frac{q}{p}}\leq
3 r_{p,q} \big[(1+2  r_{p,q})^{q-1}-(1- r_{p,q})^{q-1}\big], 
\]
and hence
$$
(1+2  r_{p,q})^{q-1}(1-r_{p,q}) + (1-  r_{p,q})^{q-1}(2+r_{p,q})\leq0, 
$$
which is clearly impossible since $r_{p,q}\in[0,1)$. 
\end{proof}

\begin{proof}
[Proof of Proposition \ref{prop-nontrivial-extremizer}]
By Lemma \ref{lem-rough}, we may  fix a  \( r_0 \in (r_{p,q}, \sqrt{(p-1)/(q-1)} ) \). 
Then by $r_0<\sqrt{(p-1)/(q-1)}$ and the expansion \eqref{eq:taylor} of \( G(r,\rho) \) at \( \rho=0 \), there exists $\rho_1=\rho_1(r_0)\in (0,1)$ such that 
\( G(r_0, \rho) \geq 0 \) for all \( \rho \in [0,\rho_1] \). 
Observe that \( \|T_r f_{\rho}\|_q \) is increasing in \( r \in [0,1] \), which implies that \( G(r, \rho) \) is decreasing with respect to \( r \).
Hence, for any $r\in[0, r_0]$ (and in particular for $r\in (r_{p,q}, r_0]$), we have 
\begin{align}\label{near-0}
G(r, \rho)\geq G(r_0, \rho)\geq 0,\ \forall\  \rho\in[0,\rho_1].
\end{align}
Now for any $r\in(r_{p,q},r_0]$, by the definition of \( r_{p,q} \) and Proposition \ref{thm-zhongyaoi},  $\min\{G(r,\rho)\mid\rho\in[0,1]\}<0$ and hence 
$$\min\{G(r,\rho)\mid\rho\in[\rho_1,1]\}<0,$$
since  one already has $\min\{G(r, \rho) \mid \rho \in [0,\rho_1]\}\ge 0$ by \eqref{near-0}.  
Therefore, on the one hand,  by the continuity of \( G(r,\rho) \), letting $r\to r_{p,q}^+$ gives 
$$\min\{G(r_{p,q},\rho)\mid\rho\in[\rho_1,1]\}\leq 0,$$
and on the other hand, again by the definition of \( r_{p,q} \), 
\[
 \min\{G(r_{p,q},\rho)\mid\rho\in[\rho_1,1]\} \ge   \min\{G(r_{p,q},\rho)\mid\rho\in[0,1]\}   \ge 0.
\]
It follows that  $\min\{G(r_{p,q},\rho)\mid\rho\in[\rho_1,1]\}   = 0$. 
That is, there exists \( \rho_0 \in [\rho_1, 1] \) such that \( G(r_{p,q}, \rho_0) = 0 \).
Finally, Lemma \ref{lemma-GAR-dayuling} implies that  \( \rho_0 \in [\rho_1, 1) \subset (0,1) \), which finishes the proof.
%
%By Lemma \ref{lem-rough}, we have \( r_{p,q} < \sqrt{\frac{p-1}{q-1}} \). For any \( r \in \bigl(r_{p,q}, \sqrt{\frac{p-1}{q-1}} \bigr) \), there exist constants \( 0 < \rho_1 < \rho_2 < 1 \) such that  
%\begin{enumerate}
%    \item \( G(r, \rho) > 0 \) for all \( \rho \in (0,\rho_1] \) (since \( r < \sqrt{\frac{p-1}{q-1}} \));
%    \item \( G(r, \rho_2) < 0 \) (since \( r > r_{p,q} \)).
%\end{enumerate}
%Statement (1) follows by the expansion \eqref{eq:taylor} of  $G(r,\rho)$ at $\rho=0$.  Statement (2) follows by Proposition \ref{thm-zhongyaoi} and  the definition of $r_{p,q} $
%
%
%Note that \( \|T_r f_{\rho}\| \) is increasing in \( r \in (0,1] \), which implies that \( G(r, \rho) \) is decreasing with respect to \( r \) on this interval. By the continuity of \( G(r, \rho) \), there exists a point \( \rho_0 \in (\rho_1, 1] \) such that \( G(r_{p,q}, \rho_0) = 0 \). 
%Finally, Lemma \ref{lemma-GAR-dayuling} further implies that \( \rho_0 \in (\rho_1, 1) \subset (0,1) \), completing the proof.
\end{proof}

\subsection{Derivation of the extremal system}\label{sec-system}
Since \( \rho_0 \in (0,1)\) is an interior point of \( [0,1] \) where the function \(  G(r_{p,q}, \cdot) \) attains its minimum ($G(r_{p,q}, \rho_0) = 0$),  the pair \( (r_{p,q}, \rho_0) \) satisfies 
\begin{align}\label{sec-degree-zero}
 G(r,\rho) = 0 \text{\,\, and \,\,} \frac{\partial G}{\partial \rho}(r,\rho) = 0.
\end{align}
%\begin{equation}\label{eq:extremalsystem}
%\begin{cases}
%G(r,\rho) = 0,\\[4pt]
%\displaystyle \frac{\partial G}{\partial \rho}(r,\rho) = 0.
%\end{cases}
%\end{equation}
By direct computation, it can be rewritten as 
\begin{align*}%\label{eqn-pair-solution}
\left\{
\begin{aligned}
&\Big(\frac{(1+2\rho)^p+2(1-\rho)^p}{3}\Big)^{\frac{1}{p}} =\Big(\frac{(1+2 r\rho)^q+2(1-r\rho)^q}{3}\Big)^{\frac{1}{q}},\\
&\frac{ (1+2\rho)^{p-1}-(1-\rho)^{p-1}}{ \Big(\frac{(1+2\rho)^p+2(1-\rho)^p}{3}\Big)^{1-\frac{1}{p}} }=r\frac{ (1+2 r\rho)^{q-1}-(1-r\rho)^{q-1}}{\Big(\frac{(1+2 r\rho)^q+2(1-r\rho)^q}{3}\Big)^{1-\frac{1}{q}}}, 
\end{aligned}
\right.
\end{align*}
which is equivalent to
\begin{equation} \label{eq:original}
\left\{
\begin{aligned}
&\Big(\frac{(1+2\rho)^p+2(1-\rho)^p}{3}\Big)^{\frac{1}{p}} =\Big(\frac{(1+2 r\rho)^q+2(1-r\rho)^q}{3}\Big)^{\frac{1}{q}},\\
&\rho\frac{(1+2\rho)^{p-1}-(1-\rho)^{p-1}}{(1+2\rho)^p+2(1-\rho)^p}=r\rho\frac{(1+2r\rho)^{q-1}-(1-r\rho)^{q-1}}{(1+2r\rho )^q+2(1-r\rho)^q}.
\end{aligned}
\right.
\end{equation}
Applying the change of variables 
\begin{equation}\label{eq:transform-ar2xy}
(x ,y)= \Big(\frac{1 - \rho}{1 + 2\rho},  \frac{1 - r\rho}{1 + 2r \rho}\Big),
\end{equation} 
we obtain the desired equation-system \eqref{xy-eq}, i.e.,
 \begin{align*}
\left\{
\begin{aligned}
&\frac{1}{1+2x}\Big(\frac{1+2x^p}{3}\Big)^{\frac{1}{p}}=\frac{1}{1+2y}\Big(\frac{1+2y^q}{3}\Big)^{\frac{1}{q}},\\
&\frac{(1-x)(1-x^{p-1})}{1+2x^p}=\frac{(1-y)(1-y^{q-1})}{1+2y^q}.
\end{aligned}
\right.
\end{align*}

\section{Proof of Theorem \ref{thm-main-thm-one-xy}}\label{S:proof-thm1}

To complete the proof of Theorem \ref{thm-main-thm-one-xy}, we need to show that \eqref{xy-eq} indeed has a unique solution, and we will carry out this work starting in  \S \ref{sec-unique}. For now, {\it assuming the uniqueness}, we can quickly complete the proof of Theorem \ref{thm-main-thm-one-xy}.

\subsection{Self-dual formulation: proof of Lemma \ref{lem-self-dual}}

Using the elementary identity 
\begin{align}\label{elem-id}
\frac{(1+ 2z^{\kappa-1})(1+2z)}{1+2z^\kappa} = 3  - 2\frac{(1-z^{\kappa-1})(1-z)}{1+2z^{\kappa}},
\end{align}
 the second equation in the equation-system  \eqref{xy-eq} is equivalent to 
\begin{align}\label{no-minus-eq}
\frac{(1+2x)(1+2x^{p-1})}{1+2x^p}=\frac{(1+2y)(1+2y^{q-1})}{1+2y^q}.
\end{align}
The first equation \eqref{xy-eq} multiplied by \eqref{no-minus-eq} yields
\[
(1+2x^{p-1}) \Big(\frac{1+2x^p}{3}\Big)^{\frac{1}{p}-1}=(1+2y^{q-1}) \Big(\frac{1+2y^q}{3}\Big)^{\frac{1}{q}-1},
\]
which can be rewritten as 
\begin{align}\label{dual-conj-eq}
\frac{1}{1+2x^{p-1}} \Big(\frac{1+2x^p}{3}\Big)^{\frac{1}{p^*}}=\frac{1}{1+2y^{q-1}} \Big(\frac{1+2y^q}{3}\Big)^{\frac{1}{q^*}}.
\end{align}
By $(p-1)p^* = p$ and $(q-1)q^*=q$, the equality 
\eqref{dual-conj-eq} is equivalent to $\ell(p^*, x^{p-1})= \ell(q^*, y^{q-1})$. 
Reversing the above procedure, we obtain the equivalence between   \eqref{xy-eq} and \eqref{self-dual-form}. 

Now we turn to the proof of the equality \eqref{cross-sym}. 
Indeed, the identity \eqref{elem-id} implies also 
\begin{align*}
  \frac{(1+2y^{q-1}) (1+2y)}{(1 - y^{q-1}) (1- y)} &=  \frac{3(1+2y^{q})}{(1 - y^{q-1})(1- y)} -2,
 \\
   \frac{(1+2x^{p-1}) (1+2x)}{(1 - x^{p-1}) (1- x)} & =  \frac{3(1+2x^{p})}{(1 - x^{p-1})(1- x)} -2.
\end{align*}
Therefore, if $(x,y)\in (0,1)^2$ is the solution to \eqref{xy-eq}, then   
\[
  \frac{(1+2y^{q-1}) (1+2y)}{(1 - y^{q-1}) (1- y)}  = \frac{(1+2x^{p-1}) (1+2x)}{(1 - x^{p-1}) (1- x)}
\]
and hence 
\[
   \frac{  (y^{q-1}+1/2)  (x^{p-1}-1)}{ ( y^{q-1}-1) (x^{p-1}+1/2)}= \frac{ (x+1/2) (y-1) }{  (x-1) (y+1/2)}.
\]
This is exactly the desired equality $ (y^{q-1}, x^{p-1}; -1/2, 1) =  (x,y; -1/2, 1) $.

\subsection{Proof of Theorem \ref{thm-main-thm-one-xy} by assuming uniqueness}\label{sec-ass-unique}
Let $1<p<q<\infty$.  We give the proof in two cases. 

{\flushleft Case 1:  $1<p<q<2$ or $1<p\le 2 < q <\infty$. }

In this case, we show in \S \ref{sec-system} that  there exists $\rho_{p,q} \in (0,1)$ such that $(r_{p,q}, \rho_{p,q})$ satisfies  \eqref{sec-degree-zero}  and under  the changing-of-variable \eqref{eq:transform-ar2xy}, the pair 
\[
(x_{p,q},y_{p,q}) = \Big(\frac{1 - \rho_{p,q}}{1 + 2\rho_{p,q}},  \frac{1 - r_{p,q} \rho_{p,q} }{1 + 2 r_{p,q} \rho_{p,q}}\Big) \in (0,1)^2 
\]
 satisfies the equation-system \eqref{xy-eq}.  Therefore, by the uniqueness of the solution of \eqref{xy-eq} in $(0,1)^2$, the pair $(x_{p,q}, y_{p,q})\in (0,1)^2$ is uniquely determined by the equation-system \eqref{xy-eq}. Hence,   reversing   changing-of-variable $(x ,y)= (\frac{1 - \rho}{1 + 2\rho},  \frac{1 - r \rho }{1 + 2 r \rho})$, we obtain the desired relations: 
\[
r_{p,q}=\frac{(1 + 2x_{p,q})(1 - y_{p,q})}{(1 + 2y_{p,q})(1 - x_{p,q})} \text{\, and \,}
\rho_{p,q}=\frac{1 - x_{p,q}}{1 + 2x_{p,q}}.
\]

{\flushleft Case 2:  $2< p<q<\infty$ or $1<p<q=2$. }

In this case, let $(x_{p,q}, y_{p,q})$ (resp. $(x_{q^*, p^*}, y_{q^*, p^*})$) denote the unique solution in the open unit square $(0,1)^2$ to \eqref{xy-eq} with respect to the parameter $(p,q)$ (resp. $(q^*, p^*)$).  

Since  $1<q^*<p^*<2$ or $q^* =2 <p^*<\infty$,  by the conclusion in case 1, we obtain 
\[
r_{q^*, p^*} =  \frac{(1 + 2x_{q^*, p^*})(1 - y_{q^*, p^*})}{(1 + 2y_{q^*,p^*})(1 - x_{q^*, p^*})} = (x_{q^*, p^*}, y_{q^*, p^*}; -1/2, 1). 
\]

By the elementary  symmetry of duality  \eqref{dual-rpq} of $r_{p,q}(\mathbb{Z}_3)$ and  the symmetry of duality for the cross ratio \eqref{cross-sym} proved  in Lemma \ref{lem-self-dual},   we obtain  the desired relation for $r_{p,q}(\mathbb{Z}_3)$: 
\begin{align}\label{grpq-eq}
r_{p,q}= r_{q^*, p^*} =  (x_{q^*, p^*}, y_{q^*, p^*}; -1/2, 1)=  (x_{p,q}, y_{p,q}; -1/2, 1) =  \frac{(1 + 2x_{p,q})(1 - y_{p,q})}{(1 + 2y_{p,q})(1 - x_{p,q})}.
\end{align}

Finally,  for the nontrivial critical extremizer, we need to  consider the following modified  extremal  problem: Let $\widetilde{r_{p,q}}$ be the optimal constant such that the hypercontractive inequality \eqref{eqn-Tt-f-pq} holds for all functions $f_\rho= 1 + \rho\chi + \rho\overline{\chi}$ with $\rho\in [0,1]$.  Replacing $r_{p,q}$ by $\widetilde{r_{p,q}}$, then Proposition \ref{prop-nontrivial-extremizer}  holds for all $1<p<q<\infty$ (note that now Proposition~\ref{thm-zhongyaoi} is not needed for the modified extremal problem).  Then we may repeat the arguments in \S \ref{sec-system} and prove the existence of 
$\widetilde{\rho_{p,q}} \in (0,1)$ such that $(\widetilde{r_{p,q}}, \widetilde{\rho_{p,q}})$ satisfies the equation \eqref{sec-degree-zero} and hence after changing-of-variable \eqref{eq:transform-ar2xy} and by uniqueness of the solution of the equation-system \eqref{xy-eq} in $(0,1)^2$, we obtain again 
\[
\widetilde{r_{p,q}} = \frac{(1 + 2x_{p,q})(1 - y_{p,q})}{(1 + 2y_{p,q})(1 - x_{p,q})} \text{\,\, and \,\,}  \widetilde{\rho_{p,q}}  = \frac{1-x_{p,q}}{1+2x_{p,q}}. 
\]
With  $f_{\widetilde{\rho_{p,q}}} = 1 + \widetilde{\rho_{p,q}} \chi + \widetilde{\rho_{p,q}}\overline{\chi}$ being a nontrivial critical extremizer for the modified extremal problem: 
\begin{align}\label{mod-ext}
\| T_{\widetilde{r_{p,q}}} f_{\widetilde{\rho_{p,q}}}\|_q = \| f_{\widetilde{\rho_{p,q}}}\|_p. 
\end{align}
Since we have already proved the equality  \eqref{grpq-eq} in this case, we obtain 
\[
\widetilde{r_{p,q}}= r_{p,q}. 
\]
Therefore, the equation \eqref{mod-ext} becomes 
\[
\| T_{r_{p,q}} f_{\widetilde{\rho_{p,q}}}\|_q = \| f_{\widetilde{\rho_{p,q}}}\|_p. 
\]
It follows that $f_{\widetilde{\rho_{p,q}}} = 1 + \widetilde{\rho_{p,q}} \chi + \widetilde{\rho_{p,q}}\overline{\chi}$ is also an extremal function for the original hypercontractivity inequality. This completes the whole proof of Theorem \ref{thm-main-thm-one-xy} by assuming the uniqueness of the solution to \eqref{xy-eq} inside $(0,1)^2$.

\begin{remark*}
The above argument shows that  we always have $\widetilde{r_{p,q}} = r_{p,q}$ for all $1<p<q<\infty$, although  the reduction in Proposition~\ref{thm-zhongyaoi}  is only proved in $1<p<q<2$ or $1<p\le 2 < q<\infty$. 
\end{remark*}

\section{Uniqueness of the solution}\label{sec-unique}

This section is devoted to the proof of the uniqueness of solution in the open unit square to  \eqref{xy-eq}.  While Figures~\ref{fig-H} and \ref{fig-blowup-B}, obtained via symmetrization and blowup techniques, make the uniqueness of the solution in $(0,1)^2$ appear so self-evident as to seem almost trivial, a rigorous proof of this fact requires a highly nontrivial argument.

\begin{proposition}\label{proposition-unique}
Given $1<p<q<\infty$. The equation-system \eqref{xy-eq}
has a unique solution in $(0,1)^2$. 
\end{proposition}

Define 
$$
\begin{array}{rccc}
h=(h_1,h_2):& (1,\infty)\times(0,1)&\longrightarrow &\mathbb{R}\times \mathbb{R}\\
& (p,x)&\longmapsto&
\Big(
\frac{1}{1+2x}\big(\frac{1+2x^{p}}{3}\big)^{\frac{1}{p}},\;
\frac{(1-x)(1-x^{p-1})}{1+2x^{p}}
\Big).
\end{array}
$$
Then $h$ can be regarded as a one-parameter family $\{h(p,\cdot)\}_{p\in(1,\infty)}$ of curves, with end points
$$h(p,0) = ((1/3)^{1/p},\  1),
\quad h(p,1) = (1/3,\ 0).$$
By taking derivatives, we obtain
$$\frac{\partial h_1}{\partial x} = \frac{2(x^{p-1}-1)}{(1+2x)^2(1+2x^p)}\left(\frac{1+2x^p}{3}\right)^{\frac{1}{p}}<0.$$
Since $1-x$, $1 - x^{p-1}$, $1/(1 + 2 x^p)$ are all strictly decreasing in $x$, we have $\frac{\partial h_2}{\partial x}<0$. 
Therefore both the abscissa and ordinate of the curve $h(p,\cdot)$ monotonically decrease with respect to $x$. 
In particular, the curve $h(p,\cdot)$ has no self-intersecting point. 
This implies the following lemma. 

\begin{lemma}
\label{lemma-no-self-intersecting}
Given $1<p<q<\infty$. The equation system \eqref{xy-eq} has a unique solution in $(0,1)^2$, if and only if the curves $h(p,\cdot)$ and $h(q,\cdot)$ intersect exactly once. 
\end{lemma}

%\begin{proof}
%By taking derivatives, we obtain
%\begin{align*}
%\frac{\partial h_1}{\partial x} &= \frac{2(x^{p-1}-1)}{(1+2x)^2(1+2x^p)}\left(\frac{1+2x^p}{3}\right)^{\frac{1}{p}}<0,\\
%\frac{\partial h_2}{\partial x} &= \frac{(x^{p-2}-1)(1+2x^p)+px^{p-2}(x-1)(1+2x)}{(1+2x^p)^2}\\
%&=\frac{(x^{p-1}-1)(2x^{p-1}+1) + (p-1)x^{p-2}(x-1)(1+2x)}{(1+2x^p)^2}<0. 
%\end{align*}
%This implies that both the abscissa and ordinate of the curve $h(p,\cdot)$ monotonically decrease with respect to $x$. 
%Therefore, the curve $h(p,\cdot)$ has no self-intersecting point, and similarly for $h(q,\cdot)$. The equivalence in the lemma follows immediately. 
%\end{proof}

By the derivation of the equation system \eqref{xy-eq}, the existence of a solution in the unit square $(0,1)^2$ is automatic. 
To prove its uniqueness, we will first symmetrize $h$, and then establish three crucial properties for the symmetrical family.  
Combining these three properties, we will give the proof of Proposition \ref{proposition-unique} in \S\ref{S:transver}.

\subsection{Symmetrization}

Observe that $h_2(p,x) = h_2(p^*,x^{p-1})$
with $p^*=p/(p-1)$, we define a transformation
\begin{align}\label{phi}
\begin{array}{rccc}
\Phi:& (-1,1)\times (0,1)&\longrightarrow &(1,\infty)\times(0,1)\\[1mm]
& (\alpha,t)&\longmapsto&\big(\frac{2}{1-\alpha}, t^{1-\alpha}\big),
\end{array}
\end{align}
See the left of Figure \ref{fig-phi-psi}. 
When $\Phi(\alpha,t)=(p,x)$, we have $\Phi(-\alpha,t)=(p^*,x^{p-1})$. 
It follows that $$h_2\circ\Phi(\alpha,t) = h_2\circ\Phi(-\alpha,t).$$

\begin{figure}[htbp]
\centering
\includegraphics{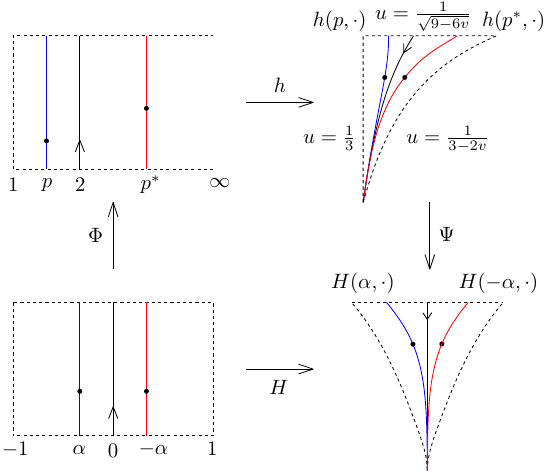}
\caption{Symmetrization through coordinate exchanges}
\label{fig-phi-psi}
\end{figure}

Having examined symmetry in the domain of $h$, we now turn to symmetry in its range. 
Let $(u,v)=h(p,x)$. If $p=2$, then $v = 3/2-1/(6 u^2)$. Letting $p\to\infty$, we get $v = 3/2-1/(2u)$; letting $p\to1$, we get $u=1/3$. 
That is, $p=1,2,\infty$ correspond to 
$u=1/3$, $u=1/\sqrt{9-6v}$, 
$u=1/(3-2v)$ respectively. 
See the upper-right corner of Figure \ref{fig-phi-psi}.  
Note that $1/\sqrt{9-6v}$ is exactly the geometric mean of $1/3$ and $1/(3-2v)$.  
Define another transformation
\begin{align}\label{psi}
\begin{array}{rccc}
\Psi:& (0,\infty)\times(0,1)&\longrightarrow &\mathbb{R}\times(0,1)\\[1mm]
& (u,v)&\longmapsto&\big(\log u-\log\frac{1}{\sqrt{9-6v}},v\big).
\end{array}
\end{align}

Through the coordinate exchanges $\Phi$ and $\Psi$, we define a new curve family 
$$H = \Psi\circ h\circ \Phi,$$ 
as shown in Figure \ref{fig-phi-psi}. 
Direct computation gives  
\begin{equation}
\label{eq-H-alpha-t}
\begin{aligned}
H(\alpha,t)&=\Psi
\left(
\frac{1}{1+2 t^{1-\alpha}}\left(\frac{1+2 t^2}{3}\right)^{\frac{1-\alpha}{2}},\ \  
\frac{(1-t^{1-\alpha})(1-t^{1+\alpha})}{1+2 t^2}
\right)\\
&=
\left(
-\frac{\alpha}{2}\log \frac{1+2 t^2}{3} 
+\frac{1}{2}\log\frac{1+2 t^{1+\alpha}}{1+2 t^{1-\alpha}},\ \  
\frac{(1-t^{1-\alpha})(1-t^{1+\alpha})}{1+2 t^2}
\right).
\end{aligned}
\end{equation}
Observe that $H$ satisfies the following symmetry with respect to $\alpha=0$. Denote $H=(H_1,H_2)$. Then $$H_1(-\alpha,t)=-H_1(\alpha,t),\quad H_2(-\alpha,t)=H_2(\alpha,t).$$
Particularly, $H_1(0,t)\equiv0$. 
See the bottom of Figure \ref{fig-phi-psi}. 
 
Since both $\Phi$ and $\Psi$ are homeomorphisms, 
to prove that distinct curves $h(p,\cdot)$ have a unique intersection point, it suffices to prove that distinct $H(\alpha,\cdot)$ have a unique intersection point. 
As shown in Figure \ref{fig-H}, after zooming in at the origin $(0,0)$, distinct curves of the family $H(\alpha,\cdot)$ appear to intersect at exactly one point. 
To rigorously prove this uniqueness, we rely on three independent key properties for the symmetrical family $H$: {\bf injectivity}, {\bf separation}, and {\bf transversality}.

\subsection{Injectivity}\label{S:injectivity}
Recall that $$
H=(H_1,H_2)=\left(
-\frac{\alpha}{2}\log \frac{1+2 t^2}{3} 
+\frac{1}{2}\log\frac{1+2 t^{1+\alpha}}{1+2 t^{1-\alpha}},\ \ 
\frac{(1-t^{1-\alpha})(1-t^{1+\alpha})}{1+2 t^2}
\right).
$$
For $t=\frac{1}{2}$, we have $ H_1(\alpha,\frac{1}{2})\equiv0$ and
$$H_2(\alpha,\tfrac{1}{2})=\frac{(2-2^\alpha)(2-2^{-\alpha})}{6}\in(0,\tfrac{1}{6}].$$
It follows that $H((-1,1)\times\{\frac{1}{2}\}) = \{(0,v)\mid 0<v\leq\frac{1}{6}\} =: L$.
Let
$$\Omega = \{(u,v)\in\mathbb{R}^2\mid 0<v<1,\ \tfrac{1}{2}\log\tfrac{3-2v}{3}<u< -\tfrac{1}{2}\log\tfrac{3-2v}{3}\}\setminus L.$$
That is, $\Omega$ is the domain bounded by the dashed line in the bottom-right corner of Figure \ref{fig-phi-psi}, minus the vertical line segment $L$.

\begin{proposition}
\label{restrict-homeo}
The restriction $H:(-1,1)\times(0,\frac{1}{2})\to\Omega$ is a homeomorphism.
\end{proposition}

\begin{proof}
Let $D_0=(-1,1)\times(0,1)$ and $D=(-1,1)\times(0,\frac{1}{2})$. From the behavior of $h$, we know that as $(\alpha,t)\to\partial D_0$, $H(\alpha,t)\to\partial\Omega$. Moreover, $H(\alpha,\frac{1}{2})\in L\subset\partial\Omega$. 
Therefore, as $(\alpha,t)\to\partial D$, $ H(\alpha,t)\to\partial\Omega$, which implies that $ H|_D$ is a proper map.

By Lemma \ref{detJH<0} below, $H|_D$ is a local homeomorphism. Combined with the properness, $ H:D\to\Omega$ is surjective. Since both $D$ and $\Omega$ are simply connected, it follows that $H:D\to\Omega$ is injective. Hence, $H:D\to\Omega$ is a homeomorphism.
\end{proof}

\begin{lemma}
\label{detJH<0}
For $(\alpha,t)\in(-1,1)\times(0,\frac{1}{2})$, the Jacobian matrix $J_H$ of $H$ satisfies $\det(J_H)<0$. 
\end{lemma} 

\begin{proof}
A computation followed by simplification yields 
\begin{align*}
\det(J_H) =
& \frac{t}{2(1+2t^2)^2}\bigg[-\left(
2 - (1+\alpha)t^{1-\alpha} -(1-\alpha)t^{1+\alpha}
\right) 4\log t\\
&+\left(
2 -2 (1+\alpha)t^{1-\alpha} -2(1-\alpha)t^{1+\alpha} + \frac{1+\alpha}{t^{1-\alpha}} + \frac{1-\alpha}{t^{1+\alpha}}
\right) \log\frac{1+2t^2}{3} \bigg]. 
\end{align*}
Note that $$(1+\alpha)t^{1-\alpha} + (1-\alpha)t^{1+\alpha}<1+\alpha+1-\alpha=2.$$
The inequality $\det(J_H)<0$ is equivalent to 
\begin{equation}
\label{I-alpha-t}
I(\alpha,t):=\frac{\frac{1+\alpha}{t^{1-\alpha}} + \frac{1-\alpha}{t^{1+\alpha}}-2}
{2 - (1+\alpha)t^{1-\alpha} - (1-\alpha)t^{1+\alpha}}
> \frac{4\log t}{\log\frac{1+2t^2}{3}}-2.
\end{equation}

The proof of the lemma proceeds in two steps. First, we show that 
\begin{equation}
\label{I-alpha-t-I-0-t}
I(\alpha,t)\geq I(0,t), \quad \forall\  (\alpha,t)\in(-1,1)\times(0,1).
\end{equation} 
The function $I(\alpha,t)$ constructed here is monotonic in $\alpha\in(0,1)$ and in $\alpha\in(-1,0)$, which can be verified by differentiation. 
Second, we prove that 
\begin{equation}
\label{I-0-t}
I(0,t) = \frac{1}{t} > \frac{4\log t}{\log\frac{1+2t^2}{3}}-2, \quad \forall\  t\in(0,\tfrac{1}{2}).
\end{equation}

We now show that $\frac{\partial I}{\partial \alpha}>0$ 
for $(\alpha,t)\in(0,1)\times(0,1)$. 
Direct computation gives 
\begin{align*}
&\frac{(2 - (1+\alpha)t^{1-\alpha} - (1-\alpha)t^{1+\alpha})^2}{2}\ 
\frac{\partial I}{\partial \alpha}
\\
&=\Big[
4\alpha - (1+\alpha)(t^{1-\alpha}+\tfrac{1}{t^{1-\alpha}})+(1-\alpha)(t^{1+\alpha}+\tfrac{1}{t^{1+\alpha}})
\Big]\log\tfrac{1}{t} 
- (\tfrac{1}{t^\alpha} - t^\alpha)
(1-t^{1-\alpha})(1-t^{1+\alpha}) \tfrac{1}{t}\\
&=:  J_1\log\tfrac{1}{t} - J_2.
\end{align*}
Thus, it suffices to prove that for $(\alpha,t)\in(0,1)\times(0,1)$, 
$$J_1\log\tfrac{1}{t} > J_2.$$
Clearly, $J_2>0$. 
Since $J_1(\alpha,1)\equiv 0$ and $$\frac{\partial J_1}{\partial t}=(1-\alpha^2)(t^\alpha-\tfrac{1}{t^\alpha})(1+\tfrac{1}{t^2})<0,$$
we have $J_1>0$. Therefore, the problem reduces to proving 
$$J := - \frac{J_2}{J_1} + \log \tfrac{1}{t} > 0.$$
Clearly, $J(\alpha,1)\equiv0$. 
Computation and simplification yield 
\begin{align*}
&\frac{\partial J}{\partial t} J_1^2 = 
-\frac{\partial J_2}{\partial t} J_1 + \frac{\partial J_1}{\partial t} J_2 -\frac{1}{t} J_1^2
=-\frac{1}{t^2}(1-t^{1-\alpha})(1-t^{1+\alpha})\\
&\Big[
(t+\tfrac{1}{t})(t^\alpha-\tfrac{1}{t^\alpha})^2
-2(t-\tfrac{1}{t})(t^{2\alpha}-\tfrac{1}{t^{2\alpha}})\alpha + \left( (t+\tfrac{1}{t})(t^\alpha+\tfrac{1}{t^\alpha})^2
+4(t+\tfrac{1}{t}) - 8(t^\alpha+\tfrac{1}{t^\alpha})
\right)\alpha^2
\Big].
\end{align*}
Since $(t+\frac{1}{t})(t^\alpha+\frac{1}{t^\alpha})>4$, we have 
$$(t+\tfrac{1}{t})(t^\alpha+\tfrac{1}{t^\alpha})^2
+4(t+\tfrac{1}{t}) - 8(t^\alpha+\tfrac{1}{t^\alpha})
> 4(t+\tfrac{1}{t}) - 4(t^\alpha+\tfrac{1}{t^\alpha})
=\frac{4(1-t^{1-\alpha})(1-t^{1+\alpha})}{t}>0. 
$$
To prove $\frac{\partial J}{\partial t}<0$, it suffices to show that the discriminant of the quadratic in $\alpha$ inside the square brackets is negative. That is, 
$$(t-\tfrac{1}{t})^2(t^{\alpha}+\tfrac{1}{t^{\alpha}})^2
<
(t+\tfrac{1}{t})\Big[
(t+\tfrac{1}{t})(t^\alpha+\tfrac{1}{t^\alpha})^2
+4(t+\tfrac{1}{t}) - 8(t^\alpha+\tfrac{1}{t^\alpha})
\Big].$$
Rearranging and simplifying gives 
$$2(t+\tfrac{1}{t})(t^\alpha+\tfrac{1}{t^\alpha})
<(t+\tfrac{1}{t})^2 + (t^\alpha+\tfrac{1}{t^\alpha})^2.$$
Completing the square shows that this inequality holds. Since $\alpha\neq1$, the equality is never attained. 

From $J(\alpha,1)\equiv0$ and $\frac{\partial J}{\partial t}<0$, it follows that $J>0$, and hence $\frac{\partial I}{\partial \alpha}>0$. 
It is worth noting that the above proves $\frac{\partial I}{\partial \alpha}>0$ for $\alpha\in(0,1)$; by symmetry, $\frac{\partial I}{\partial \alpha}<0$ for $\alpha\in(-1,0)$. Therefore, \eqref{I-alpha-t-I-0-t} holds. 

We now verify \eqref{I-0-t}. That is, 
\begin{equation}
\label{st-log}
s(t):=\log\frac{1+2t^2}{3}-\frac{4 t}{1+2t}\log t<0,\quad \forall\  t\in (0, \tfrac{1}{2}).
\end{equation}
Note that $s(\frac{1}{2})=0$. It suffices to show that 
$$s'(t) = 4\frac{-(1+t-2t^2)+(1+2t^2)\log \frac{1}{t}}
{(1+2t^2)(1+2t)^2}>0.$$
That is, we need to prove 
$$r(t):=\log\frac{1}{t} - \frac{1+t-2t^2}{1+2t^2}>0.$$
This follows from $r(\frac{1}{2})=\log 2-\frac{2}{3}>0$ and $$r'(t) = -\frac{1}{t}\frac{(1-t)(1+2t)(1-2t^2)}
{(1+2t^2)^2}<0, \quad\forall\   t\in(0, \tfrac{1}{2}).$$
This completes the proof that \eqref{I-alpha-t} holds for all $(\alpha,t)\in(-1,1)\times(0,\frac{1}{2})$. 
\end{proof}

\subsection{Separation}\label{S:sepa}
Recall from \eqref{blowup-b} that 
\[
b(u,v) = \left(\frac{u}{v},\ v\right).
\]
The action of \(b\circ H\) is illustrated in Figure \ref{fig-blowup-b}.  
As hinted in the figure, the image \(b\circ H((-1,1)\times(\frac{1}{2},1))\) is a curvilinear triangle with a cusp.

\begin{proposition}
\label{separation}
The image $b\circ H((-1,1)\times(\frac{1}{2},1))$ does not intersect the curve $b\circ H((-1,1)\times\{\frac{2}{5}\})$. 
\end{proposition} 

For the proof, we need to estimate 
\begin{align*}
H_1&=-\frac{\alpha}{2}\log \frac{1+2 t^2}{3} 
+\frac{1}{2}\log(1+2 t^{1+\alpha})-\frac{1}{2}\log(1+2 t^{1-\alpha}), \\
H_2&=\frac{(1-t^{1-\alpha})(1-t^{1+\alpha})}{1+2 t^2}.
\end{align*}

\begin{lemma}
[Estimates for $H_1$]
\label{estimate-H1}
For $t>0$, we have the limits 
\begin{align*}
\lim\limits_{\alpha\to0} \frac{H_1}{\alpha(1-\alpha^2)} 
&= - \frac{1}{2}\log\frac{1+2t^2}{3}+\frac{2t}{1+2t}\log t=:U_0,\\
\lim\limits_{\alpha\to1} \frac{H_1}{\alpha(1-\alpha^2)} 
&= \frac{1}{4}\log\frac{1+2t^2}{3} - \frac{1+5t^2}{6(1+2t^2)}\log t=:U_1. 
\end{align*}
For $\alpha\in(-1,0)\cup(0,1)$ and $t\in(\frac{1}{2},1)$, $$U_0< \frac{H_1}{\alpha(1-\alpha^2)}<U_1<0;$$ 
For $\alpha\in(-1,0)\cup(0,1)$ and $t\in[\frac{1}{4},\frac{1}{2})$, $$U_0>\frac{H_1}{\alpha(1-\alpha^2)}>U_1>0.$$
\end{lemma}

\begin{remark*}
The graph shows that $$\max\limits_{t\in[\frac{1}{4},\frac{1}{2}]} |U_0(t)-U_1(t)|<10^{-2},\quad\max\limits_{t\in[\frac{1}{2},1]} |U_0(t)-U_1(t)|<2\times10^{-5}.$$
Hence, $U_0$ and $U_1$ differ very little. A sketch is shown in Figure \ref{fig-U0U1}. 
\begin{figure}[htbp]
\centering
\includegraphics{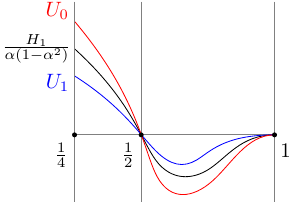}
\caption{Sketch of the estimates for $H_1$}
\label{fig-U0U1}
\end{figure}
\end{remark*}

\begin{lemma}
[Estimates for $H_2$]
\label{estimate-H2}
For $t>0$, we have the limit
$$\lim\limits_{\alpha\to1} \frac{H_2}{1-\alpha^2} 
= \frac{t^2 - 1}{2(1+2 t^2)} \log t=:V_1. $$
For $(\alpha,t)\in(-1,1)\times(0,1)$, $$V_0:=\frac{(1-t)^2}{1+2t^2}\leq \frac{H_2}{1-\alpha^2}<V_1,$$
where equality holds if and only if $\alpha=0$. 
\end{lemma}

Before showing lemmas \ref{estimate-H1} and \ref{estimate-H2}, let us prove Proposition \ref{separation}. 

\begin{proof}[Proof of Proposition \ref{separation}]
Let $\alpha\in(0,1)$ and $t\in(\frac{1}{2},1)$. 
By Lemmas \ref{estimate-H1} and \ref{estimate-H2}, we have 
$$0<\frac{-H_1}{H_2} < \frac{-\alpha(1-\alpha^2)U_0}{(1-\alpha^2)V_0}=\frac{-\alpha U_0}{V_0}<\frac{\alpha}{100}.$$
Here, the inequality $\frac{-U_0}{V_0}<\frac{1}{100}$ is equivalent to $100 U_0+V_0>0$, 
which can be proved similarly to \eqref{st-log}. 
Since $H_2$ is monotonically decreasing in $t$,  
$$H_2(\alpha,t) < H_2(\alpha,\tfrac{1}{2}) = \frac{5-2(2^\alpha+2^{-\alpha})}{6}.$$
By symmetry, for any $(\alpha,t)\in (-1,1)\times(\frac{1}{2},1)$, we have $$\left|\frac{H_1}{H_2}\right|\leq\frac{|\alpha|}{100},\quad H_2<\frac{5-2(2^\alpha+2^{-\alpha})}{6}.$$

Now consider $\alpha\in(0,1)$ and $t=\frac{2}{5}$. 
By Lemmas \ref{estimate-H1} and \ref{estimate-H2}, we obtain 
\begin{align*}
\frac{H_1}{H_2} &> \frac{\alpha(1-\alpha^2)U_1}{(1-\alpha^2)V_1}=\frac{\alpha U_1}{V_1}
=\left(\frac{5}{7}-\frac{11}{14}
\frac{\log\frac{11}{25}}
     {\log\frac{2}{5}}\right)\alpha>\frac{\alpha}{100}, \\ 
H_2 &> (1-\alpha^2)V_0 = \frac{3}{11}(1-\alpha^2).
\end{align*}
% 5/7-11/14*log(11/25)/log(2/5)
By symmetry again, for any $(\alpha,t)\in (-1,1)\times\{\frac{2}{5}\}$, we have $$\left|\frac{H_1}{H_2}\right|\geq\frac{|\alpha|}{100},\quad H_2\geq\frac{3}{11}(1-\alpha^2).$$

For every $\alpha\in(-1,1)$, one easily checks that $\frac{3}{11}(1-\alpha^2) > \frac{5-2(2^\alpha+2^{-\alpha})}{6}$, the proposition follows. 
\end{proof}

\begin{proof}
[Proof of Lemma \ref{estimate-H1}]
The limits can be obtained by direct computation. We now prove the inequalities. 
Let $g=H_1-\alpha(1-\alpha^2) U_1$. 
Define 
\begin{align*}
&r_1=\frac{x}{1+2x},\hspace{6.15em}
r_2=xr'_1=\frac{x}{(1+2x)^2},\\
&r_3=xr'_2=\frac{x(1-2x)}{(1+2x)^3},\qquad
r_4=xr'_3=\frac{x(1-8x+4x^2)}{(1+2x)^4}.
\end{align*}
Then \begin{align*}
\frac{\partial g}{\partial \alpha} 
&= -\frac{1}{2}\log\frac{1+2t^2}{3}+
(r_1(t^{1+\alpha})+r_1(t^{1-\alpha}))\log t - (1-3\alpha^2) U_1
,\\
\frac{\partial^2 g}{\partial \alpha^2} 
&= (r_2(t^{1+\alpha})-r_2(t^{1-\alpha}))(\log t)^2+6\alpha U_1
,\\
\frac{\partial^3 g}{\partial \alpha^3} 
&= (r_3(t^{1+\alpha})+r_3(t^{1-\alpha}))(\log t)^3+6 U_1
,\\
\frac{\partial^4 g}{\partial \alpha^4} 
&= (r_4(t^{1+\alpha})-r_4(t^{1-\alpha}))(\log t)^4
=
\frac{-4 t^3(4t^2-1)(\frac{1}{t^\alpha}-t^\alpha)(\log t)^4}{(1+2t^{1+\alpha})^4(1+2t^{1-\alpha})^4}I,\\
I &:= 24+2 (2t+\tfrac{1}{2t} ) (t^\alpha+\tfrac{1}{t^\alpha} ) - (2t+\tfrac{1}{2t}-t^\alpha-\tfrac{1}{t^\alpha} )^2. 
% matlab验证代码
%syms t positive
%syms alpha real
%f = ( t^(1+alpha)*(1-8*t^(1+alpha)+4*(t^(1+alpha))^2)/(1+2*t^(1+alpha))^4 - t^(1-alpha)*(1-8*t^(1-alpha)+4*(t^(1-alpha))^2)/(1+2*t^(1-alpha))^4 )* log(t)^4 - ( -4*t^3*(4*t^2-1)*(1/t^alpha-t^alpha)*(24+2*(2*t+1/(2*t))*(t^alpha+1/t^alpha)- (2*t+1/(2*t)-t^alpha-1/t^alpha)^2)*log(t)^4 )/((1+2*t^(1+alpha))^4*(1+2*t^(1-alpha))^4);
%f = simplify(f)
\end{align*}
For $(\alpha,t)\in(-1,1)\times[\frac{1}{4},1)$, we have $2t+\frac{1}{2t}, t^\alpha+\frac{1}{t^\alpha}\in[2,5)$, 
hence $I>0$. 
Therefore, the sign of $\frac{\partial^4 g}{\partial \alpha^4}$ is determined by $(4t^2-1)(\frac{1}{t^\alpha}-t^\alpha)$, which is equivalent to being determined by $(t-\frac{1}{2})\alpha$. 

For $(\alpha,t)\in(0,1)\times(\frac{1}{2},1)$, we have $\frac{\partial^4 g}{\partial \alpha^4}<0$. 
% \footnote{直接判定$\frac{\partial^4 g}{\partial \alpha^4}$的正负是证明的关键.}
Hence, $\frac{\partial^3 g}{\partial \alpha^3}$ is decreasing in $\alpha$. Note that 
\begin{equation}
\label{monotone-game-u-1}
g(0,t)\equiv0,\quad 
g(1,t)\equiv0,\quad
\frac{\partial g}{\partial \alpha}(1,t)\equiv0,\quad
\frac{\partial^2 g}{\partial \alpha^2}(0,t)\equiv0.
\end{equation}
We can then conclude from Figure \ref{derivatives-g} that $g<0$, i.e., $\frac{H_1}{\alpha(1-\alpha^2)}<U_1$. 
\begin{figure}[htbp]
\centering
\includegraphics{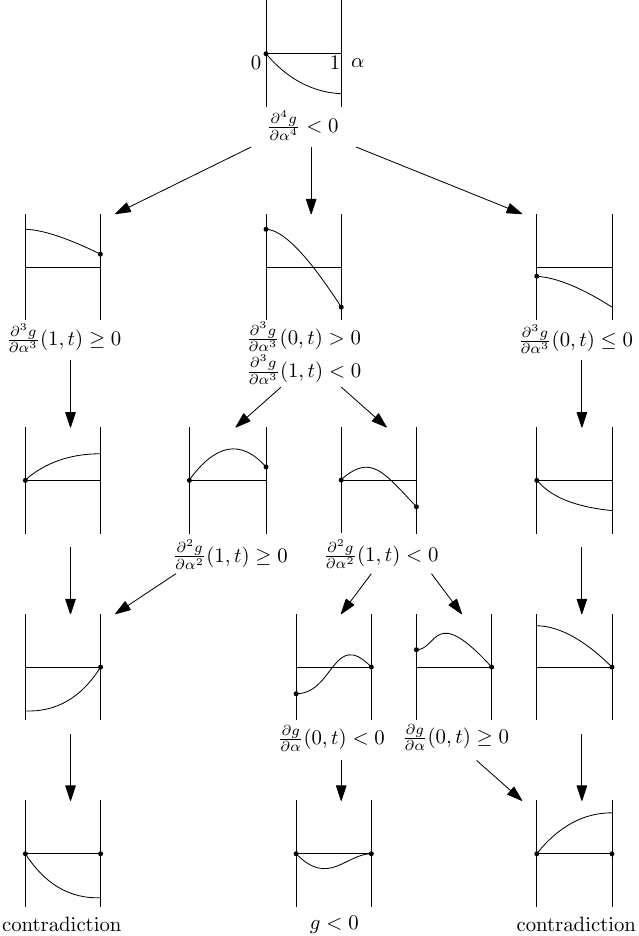}
\caption{Determining the sign of the partial derivatives of $g$ with respect to $\alpha$}
\label{derivatives-g}
\end{figure}
For $(\alpha,t)\in(0,1)\times[\frac{1}{4},\frac{1}{2})$, we have $\frac{\partial^4 g}{\partial \alpha^4}>0$. By reversing all signs in Figure \ref{derivatives-g}, we obtain $g>0$, i.e., $\frac{H_1}{\alpha(1-\alpha^2)}>U_1$. 

To establish the inequalities for $U_0$, let $g=H_1-\alpha(1-\alpha^2) U_0$. 
Then 
\begin{equation}
\label{monotone-game-u-2}
g(0,t)\equiv0,\quad 
g(1,t)\equiv0,\quad
\frac{\partial g}{\partial \alpha}(0,t)\equiv0,\quad
\frac{\partial^2 g}{\partial \alpha^2}(0,t)\equiv0.
\end{equation}
A similar argument shows that for $(\alpha,t)\in(0,1)\times(\frac{1}{2},1)$, we have $\frac{H_1}{\alpha(1-\alpha^2)}>U_0$; and for $(\alpha,t)\in(0,1)\times[\frac{1}{4},\frac{1}{2})$, we have $\frac{H_1}{\alpha(1-\alpha^2)}<U_0$.  

Finally, the case $\alpha\in(-1,0)$ follows by symmetry. 
\end{proof}

\begin{remark*}
As corollaries of Figure \ref{derivatives-g}, 
we have 
$$\frac{\partial^3 g}{\partial\alpha^3}(0,t)>0, \quad 
\frac{\partial^3 g}{\partial\alpha^3}(1,t)<0, \quad
\frac{\partial^2 g}{\partial\alpha^2}(1,t)<0, \quad
\frac{\partial g}{\partial\alpha}(0,t)<0. 
$$
These are univariate inequalities involving logarithms, which can be proved directly by methods similar to \eqref{st-log} or using the techniques in \S\ref{section-log} in the appendix. 
\end{remark*}

\begin{remark*}
The proof of Lemma \ref{estimate-H2} is analogous to that of Lemma \ref{estimate-H1} and is therefore omitted. 
\end{remark*}

\subsection{Transversality}\label{S:transver}

We turn to  the tangents of the curves in Figure \ref{fig-blowup-B}. 
For simplicity, denote 
\begin{align*}
u&=H_1= -\frac{\alpha}{2}\log \frac{1+2 t^2}{3} 
+\frac{1}{2}\log(1+2 t^{1+\alpha}) - \frac{1}{2}\log(1+2 t^{1-\alpha}),\\
v&=H_2= \frac{(1-t^{1-\alpha})(1-t^{1+\alpha})}{1+2 t^2}. 
\end{align*} 
Recall from \eqref{blowup-B} that 
\[
B(u,v) = \left(\frac{u}{v^{3/2}},\ v^{1/2}\right).
\]

Given $(\alpha,w)\in(-1,1)\times(0,1)$. 
Since $v=v(\alpha,t)$ is monotonically decreasing in $t$, the equation $w=\sqrt{v}$ determines a unique $t=t(\alpha,w)$. 
Treating $\alpha$ as the parameter and $w$ as the independent variable, we now obtain a one-parameter family of functions $uv^{-\frac{3}{2}}$. Rotating Figure \ref{fig-blowup-B} counterclockwise by $90^\circ$ and applying a reflection across the vertical axis yields the graph of $uv^{-\frac{3}{2}}$ as a function of $w$. Its derivative with respect to $w$ is shown in Figure \ref{fig-slope}. 
\begin{figure}[ht]
\centering
\includegraphics[width=0.45\textwidth]{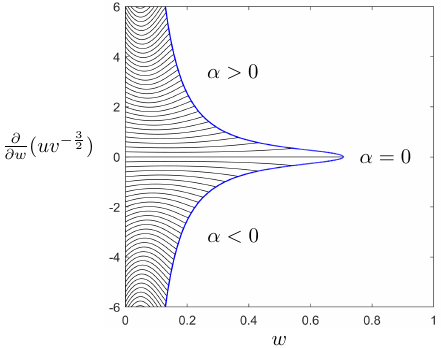}
\caption{Graphs of $\frac{\partial}{\partial w} (uv^{-\frac{3}{2}})$ for different values of $\alpha$ (the blue curve corresponds to $t=\frac{1}{4}$) 
}
\label{fig-slope}
\end{figure}

\begin{proposition}
\label{slope-monotone}
For $(\alpha,t)\in(-1,1)\times(\frac{2}{5},1)$, we have $\frac{\partial}{\partial \alpha}\frac{\partial}{\partial w} (uv^{-\frac{3}{2}})>0$. 
\end{proposition}

Before proving Proposition \ref{slope-monotone}, we use it to prove Proposition \ref{proposition-unique}.

\begin{proof}
[Proof of Proposition \ref{proposition-unique}]
It suffices to show that for $\alpha\neq\beta$, the curves $H(\alpha,\cdot)$ and $H(\beta,\cdot)$ intersect at most once. 

Let $t_0=\frac{2}{5}$. By Proposition \ref{slope-monotone}, the curve $B\circ H(\alpha,(t_0,1))$ can cross $B\circ H(\beta,(t_0,1))$ at most once. Hence, they have at most one intersection point, which implies that $H(\alpha,(t_0,1))$ and $H(\beta,(t_0,1))$ also have at most one intersection. 

By Proposition \ref{restrict-homeo}, the arcs $H(\alpha,(0,t_0])$ and $H(\beta,(0,\frac{1}{2}])$ are disjoint. 
By Proposition \ref{separation}, the images $b\circ H(\alpha,(0,t_0])$ and $b\circ H(\beta,(\frac{1}{2},1))$ are disjoint, i.e., $H(\alpha,(0,t_0])$ and $H(\beta,(\frac{1}{2},1))$ are disjoint. 
Therefore, $H(\alpha,(0,t_0])$ and $H(\beta,\cdot)$ do not intersect. Similarly, $H(\beta,(0,t_0])$ and $H(\alpha,\cdot)$ do not intersect. 

In summary, $H(\alpha,\cdot)$ and $H(\beta,\cdot)$ have at most one intersection point. Since both $\Phi$ and $\Psi$ are homeomorphisms, distinct curves $h(p,\cdot)$ also have at most one intersection point. 
By Lemma \ref{lemma-no-self-intersecting}, the equation system \eqref{xy-eq} has a unique solution in $(0,1)^2$. 
\end{proof}

\begin{proof}
[Proof of Proposition \ref{slope-monotone}]
First, we compute $\frac{\partial}{\partial \alpha}\frac{\partial}{\partial w} (uv^{-\frac{3}{2}})$. 
Treating $(\alpha,w)$ as independent variables, and differentiating both sides of $$w=\sqrt{v}=\sqrt{v(\alpha,t(\alpha,w))}$$ with respect to $\alpha$ and $w$ yields $$
0 = v_\alpha + v_t t_\alpha, \quad
1 = \frac{v_t t_w}{2\sqrt{v}}.$$
Hence, 
$$
t_\alpha = -\frac{v_\alpha}{v_t}, \quad
t_w = \frac{2\sqrt{v}}{v_t}.$$
Therefore,  
\begin{align*}
\frac{\partial}{\partial \alpha}\frac{\partial}{\partial w} (w^{-3} u)
=\ & \frac{\partial}{\partial \alpha}(-3 w^{-4} u + w^{-3} u_t t_w )
= \frac{\partial}{\partial \alpha} \left(-3 w^{-4} u + 2w^{-2} \frac{u_t}{v_t} \right)\\
=\ & -3 w^{-4}(u_\alpha+u_t t_\alpha) + 2w^{-2}\frac{(u_{t\alpha}+u_{tt}t_\alpha)v_t - (v_{t\alpha}+v_{tt}t_\alpha)u_t}{v_t^2}. 
\end{align*}
Further simplifying the inequality $\frac{\partial}{\partial \alpha}\frac{\partial}{\partial w} (w^{-3} u)>0$ leads to 
$$\frac{(u_{t\alpha}v_t-u_{tt}v_\alpha)v_t - (v_{t\alpha}v_t-v_{tt}v_\alpha)u_t}{v}\ \frac{v^2}{v_t^2}<\frac{3}{2}(u_\alpha v_t- u_t v_\alpha).$$
Note that since $v_t<0$, multiplying by $v_t$ reverses the inequality. Direct computation gives 
\begin{align*}
I &:=
\frac{(u_{t\alpha}v_t-u_{tt}v_\alpha)v_t - (v_{t\alpha}v_t-v_{tt}v_\alpha)u_t}{v} \\
&= \frac{4}{(1+2t^2)^3}\Big[
(1-\alpha^2)\left(t^\alpha+\tfrac{1}{t^\alpha} \right)\log t 
+\alpha\left(t^\alpha-\tfrac{1}{t^\alpha}\right)
+\frac{1-2t^2}{1+2t^2}\left(t^\alpha+\tfrac{1}{t^\alpha} \right)
+\frac{2t}{1+2t^2}\Big], \\
J &:= \frac{3}{2}(u_\alpha v_t- u_t v_\alpha )\\
&= \frac{3t}{4(1+2t^2)^2}\Big[
A\left(t^\alpha+\tfrac{1}{t^\alpha}\right)
+B\alpha\left(t^\alpha-\tfrac{1}{t^\alpha}\right)+C
\Big],
\end{align*}
where 
\begin{align*}
A&= 4t\log t+\frac{1-2t^2}{t}\log\frac{1+2t^2}{3}<0,\\
B&=-4t\log t+\frac{1+2t^2}{t}\log\frac{1+2t^2}{3}<0,\\
C&=-8\log t+2\log\frac{1+2t^2}{3}>0. 
\end{align*}
With these notations, the inequality to be proved becomes 
$$I\frac{v^2}{v_t^2}<J.$$

Recall that the Jacobian $\det(J_H)=u_\alpha v_t- u_t v_\alpha=\frac{2}{3}J$ has already been computed in the proof of Lemma \ref{detJH<0}. 
Recall also that an estimate for $v=H_2$ is provided by Lemma \ref{estimate-H2}. 
In \S\ref{sec-vt-vt-v}, we will further estimate $v_t$ and $\frac{v_t}{v}$. Equipped with these estimates, we now proceed to the main argument. 

%\begin{align*}
%I_1 &:=\frac{(u_{t\alpha}v_t-u_{tt}v_\alpha)v_t - (v_{t\alpha}v_t-v_{tt}v_\alpha)u_t}{v}\\
%&=\frac{4}{(1+2t^2)^3} \left( (1-\alpha^2)\left(t^\alpha+\frac{1}{t^\alpha} \right)\log t - (1+2t^2)v_t \right), \\
%I_2 &:= \frac{3}{2}(u_\alpha v_t- u_t v_\alpha )= \frac{3}{2}\det(J_H)\\
%&= \frac{-6t\log t}{(1+2t^2)^2}v + 
%\left(\frac{3t^2}{1+2t^2}\log t - \frac{3}{4}\log\frac{1+2t^2}{3}\right)v_t. 
%\end{align*}

Next, we heuristically construct an intermediate function between $I\frac{v^2}{v_t^2}$ and $J$ by plotting the graphs of these one-parameter functions for observation. 
Further computation gives 
\begin{align*}
I &=\frac{4}{(1+2t^2)^3} \Big[ (1-\alpha^2)\left(t^\alpha+\tfrac{1}{t^\alpha} \right)\log t - (1+2t^2)v_t \Big], \\
J &= \frac{-6t\log t}{(1+2t^2)^2}v + 
\left(\frac{3t^2}{1+2t^2}\log t - \frac{3}{4}\log\frac{1+2t^2}{3}\right)v_t. 
\end{align*}
Then by Lemmas \ref{estimate-H2} and \ref{estimate-vt}, the functions $v$, $v_t$, $I$, $I \frac{v^2}{v_t^2}$ and $J$ share an asymptotic factor $1-\alpha^2$. 
We now multiply $I \frac{v^2}{v_t^2}$ and $J$ by $\frac{(1+2t^2)^3}{1-\alpha^2}$, and compute their Taylor expansions at $t=1$: 
\begin{align*}
I'\frac{v^2}{v_t^2}&:=\frac{(1+2t^2)^3}{(1-\alpha^2)(t-1)^4} I \frac{v^2}{v_t^2} = \frac{1}{3} + \frac{2}{3}(1+\alpha^2)(t - 1) + o(t-1), \\
J'&:=\frac{(1+2t^2)^3}{(1-\alpha^2)(t-1)^4} J = \frac{1}{3} + \frac{1}{2}(1+\alpha^2)(t - 1) + o(t-1).
\end{align*}
After these transformations, 
$$\frac{I'\frac{v^2}{v_t^2}-\frac{1}{3}}{(1+\alpha^2)(t-1)}\text{\ \ and\ \ }\frac{J'-\frac{1}{3}}{(1+\alpha^2)(t-1)}$$ have an $\alpha$-independent intermediate function 
\begin{equation}
\label{mid-function-1-t-2}
1-\frac{t}{2},\quad t\in\left(\frac{2}{5},\ 1\right).
\end{equation}
Therefore, we aim to prove 
\begin{equation}
\label{I1-P-I2}
I\frac{v^2}{v_t^2}< P < J,
\end{equation}
where 
$$P:=\frac{(t-1)^4}{(1+2t^2)^3}(1-\alpha^2)\left[\frac{1}{3}+\left(1-\frac{t}{2}\right)(t-1)(1+\alpha^2)\right].$$
It is worth noting that the intermediate function $P$ is a polynomial of degree $4$ in $\alpha$, and thus its fifth partial derivative with respect to $\alpha$ vanishes. 

Finally, we transform the left-hand side of \eqref{I1-P-I2} into more tractable inequalities. 
In Lemma \ref{estimate-H2}, we proved that $0<V_0\leq \frac{v}{1-\alpha^2}<V_1$. Denote the derivatives of $V_0$ and $V_1$ by $V'_0$ and $V'_1$ respectively. 
By Lemma \ref{estimate-vt}, we have $0>V'_0\geq \frac{v_t}{1-\alpha^2}> V'_1$. It follows that $\frac{V'_1}{V_0}< \frac{v_t}{v} <\frac{V'_0}{V_1}<0$. Furthermore, Lemma \ref{estimate-vtv} provides the optimal $\alpha$-independent upper and lower bounds for $\frac{v_t}{v}$. Therefore, 
$$%\begin{equation}
%\label{V0V1-vtv}
\frac{V'_1}{V_0}< \frac{V'_1}{V_1}<
\frac{v_t}{v} 
\leq \frac{V'_0}{V_0}
<\frac{V'_0}{V_1}<0.
$$%\end{equation}
Define 
\begin{align*}
W_0&=\frac{V'_0}{V_0} = \frac{2(1+2t)}{(1+2t^2)(t-1)},\\
W_1&=\frac{2}{3}t-2+\frac{2}{t-1} = \frac{2(t^2-4t+6)}{3(t-1)}. 
\end{align*}
Here, $W_1$ serves as an intermediate function between $\frac{V'_1}{V_0}$ and $\frac{V'_1}{V_1}$. Both $W_0$ and $W_1$ are chosen to be rational functions to facilitate the proof of the inequality \eqref{log-Gk}. The construction of $W_1$ is analogous to that of the intermediate function \eqref{mid-function-1-t-2}. Since we are dealing with univariate functions here, the construction is more straightforward. The rigorous proof of the inequality $\frac{V'_1}{V_0}<W_1<\frac{V'_1}{V_1}$ follows the same method as in \S\ref{section-log} and is omitted. 
Although $P$ takes both positive and negative values in the domain under consideration, using $W_1<\frac{v_t}{v}\leq W_0<0$, we can decompose \eqref{I1-P-I2} into the following three inequalities to be proved: 
\begin{align}
\label{P-J}
&P<J,\\
\label{I-W0-P}
&I < P W_0^2,\\
\label{I-W1-P}
&I < P W_1^2.
\end{align}
These three inequalities will be proved later. From them, we obtain \eqref{I1-P-I2}, which directly implies $I\frac{v^2}{v_t^2}<J$ and thus completes the proof. 
\end{proof}
 
\begin{proof} 
[Proof of \eqref{P-J}]

Let 
\begin{align*}
g &= \frac{4(1+2t^2)^2}{3t}(J-P) = A\left(t^\alpha+\tfrac{1}{t^\alpha}\right)
+B\alpha\left(t^\alpha-\tfrac{1}{t^\alpha}\right)+C
+D(1-\alpha^2)+E(1-\alpha^4),
\end{align*}
where $$D=-\frac{4(t-1)^4}{9t(1+2t^2)},\quad
E=-\frac{4(t-1)^5(1-\frac{t}{2})}{3t(1+2t^2)}.$$
Then 
\begin{align*}
\frac{\partial g}{\partial\alpha}&=
(A\log t +B)\left(t^\alpha-\tfrac{1}{t^\alpha}\right)
+B\alpha\left(t^\alpha+\tfrac{1}{t^\alpha}\right)\log t
-2D\alpha-4E\alpha^3,\\
\frac{\partial^2 g}{\partial\alpha^2}&=
(A(\log t)^2 + 2 B\log t)\left(t^\alpha+\tfrac{1}{t^\alpha}\right)
+B\alpha\left(t^\alpha-\tfrac{1}{t^\alpha}\right)(\log t)^2
-2D-12E\alpha^2,\\
\frac{\partial^3 g}{\partial\alpha^3}&=
(A(\log t)^3 + 3 B(\log t)^2)\left(t^\alpha-\tfrac{1}{t^\alpha}\right)
+B\alpha\left(t^\alpha+\tfrac{1}{t^\alpha}\right)(\log t)^3
-24E\alpha,\\
\frac{\partial^4 g}{\partial\alpha^4}&=
(A(\log t)^4 + 4 B(\log t)^3)\left(t^\alpha+\tfrac{1}{t^\alpha}\right)
+B\alpha\left(t^\alpha-\tfrac{1}{t^\alpha}\right)(\log t)^4
-24E,\\
\frac{\partial^5 g}{\partial\alpha^5}&=
(A(\log t)^5 + 5 B(\log t)^4)\left(t^\alpha-\tfrac{1}{t^\alpha}\right)
+B\alpha\left(t^\alpha+\tfrac{1}{t^\alpha}\right)(\log t)^5.  
\end{align*}

Let $\alpha\in(0,1)$ and $t\in(\frac{2}{5},1)$. 
We now verify that $\frac{\partial^5 g}{\partial\alpha^5}>0$. 
This is equivalent to showing that $\alpha\frac{1+t^{2\alpha}}{1-t^{2\alpha}}>\frac{A}{B} + \frac{5}{\log t}$. 
By Lemma \ref{basic-monotone-t-2alpha}, we have 
$\alpha\frac{1+t^{2\alpha}}{1-t^{2\alpha}}>\frac{1}{-\log t}$. Thus, it suffices to prove 
\begin{equation}
\label{A-B-6-logt}
\frac{A}{B} + \frac{6}{\log t}<0.
\end{equation}
Additionally, we need to prove 
\begin{align}
\label{g00t}
g(0,t)& = 2A + C+D+E>0,\\
\label{g11t}
\frac{\partial g}{\partial\alpha}(1,t)
&=(A\log t +B)\left(t-\tfrac{1}{t}\right)
+B\left(t+\tfrac{1}{t}\right)\log t
-2D-4E
<0.
\end{align}
Combining these with $$
g(1,t)\equiv0,\quad 
\frac{\partial g}{\partial\alpha}(0,t)\equiv0,\quad
\frac{\partial^3 g}{\partial\alpha^3}(0,t)\equiv0,
$$
we can conclude, in a manner analogous to Figure \ref{derivatives-g}, that $g>0$. 
See also Example \ref{ex-section-P-J}. 
The proofs of the univariate inequalities \eqref{A-B-6-logt}, \eqref{g00t}, and \eqref{g11t} are given in \S\ref{section-log}. 

For $\alpha\in(-1,0)$, we have $g(\alpha,t)=g(-\alpha,t)>0$. 
For $\alpha=0$, the result follows from \eqref{g00t}. 
This completes the proof of \eqref{P-J}. 
\end{proof}

\begin{proof}
[Proof of \eqref{I-W0-P} and \eqref{I-W1-P}]

Let $G = \frac{(1+2t^2)^3}{4}I$. 
Then 
\begin{align*}
G &=
(1-\alpha^2)\left(t^\alpha+\tfrac{1}{t^\alpha} \right)\log t 
+\alpha\left(t^\alpha-\tfrac{1}{t^\alpha}\right)
+\frac{1-2t^2}{1+2t^2}\left(t^\alpha+\tfrac{1}{t^\alpha} \right)
+\frac{2t}{1+2t^2}, \\
\frac{\partial G}{\partial\alpha}&=
(1-\alpha^2)\left(t^\alpha-\tfrac{1}{t^\alpha} \right)(\log t)^2
-\alpha\left(t^\alpha+\tfrac{1}{t^\alpha} \right)\log t 
+t^\alpha-\tfrac{1}{t^\alpha} 
+\frac{1-2t^2}{1+2t^2}\left(t^\alpha-\tfrac{1}{t^\alpha} \right)\log t
,\\
\frac{\partial^2 G}{\partial\alpha^2}&=
(1-\alpha^2)\left(t^\alpha+\tfrac{1}{t^\alpha} \right)(\log t)^3
-3\alpha\left(t^\alpha-\tfrac{1}{t^\alpha} \right)(\log t)^2 
+\frac{1-2t^2}{1+2t^2}\left(t^\alpha+\tfrac{1}{t^\alpha} \right)(\log t)^2
,\\
\frac{\partial^3 G}{\partial\alpha^3}&=
(1-\alpha^2)\left(t^\alpha-\tfrac{1}{t^\alpha} \right)(\log t)^4
-5\alpha\left(t^\alpha+\tfrac{1}{t^\alpha} \right)(\log t)^3 \\
&\ \ \ \ 
-3\left(t^\alpha-\tfrac{1}{t^\alpha} \right)(\log t)^2
+\frac{1-2t^2}{1+2t^2}\left(t^\alpha-\tfrac{1}{t^\alpha} \right)(\log t)^3\\
&=
\left(t^\alpha-\tfrac{1}{t^\alpha} \right)(\log t)^3
\left(
(1-\alpha^2)\log t + 5\alpha\frac{1+t^{2\alpha}}{1-t^{2\alpha}}-\frac{3}{\log t}+\frac{1-2t^2}{1+2t^2}
\right).
%
%\frac{\partial^4 G}{\partial\alpha^4}&=
%(1-\alpha^2)\left(t^\alpha+\frac{1}{t^\alpha} \right)(\log t)^5
%-7\alpha\left(t^\alpha-\frac{1}{t^\alpha} \right)(\log t)^4 \\
%&\ \ \ \ 
%-8\left(t^\alpha+\frac{1}{t^\alpha} \right)(\log t)^3
%+\frac{1-2t^2}{1+2t^2}\left(t^\alpha+\frac{1}{t^\alpha} \right)(\log t)^4
%,\\
%\frac{\partial^5 G}{\partial\alpha^5}&=
%(1-\alpha^2)\left(t^\alpha-\frac{1}{t^\alpha} \right)(\log t)^6
%-9\alpha\left(t^\alpha+\frac{1}{t^\alpha} \right)(\log t)^5 \\
%&\ \ \ \ 
%-15\left(t^\alpha-\frac{1}{t^\alpha} \right)(\log t)^4
%+\frac{1-2t^2}{1+2t^2}\left(t^\alpha-\frac{1}{t^\alpha} \right)(\log t)^5\\
%&=
%\left(t^\alpha-\frac{1}{t^\alpha} \right)(\log t)^5
%\left(
%(1-\alpha^2)\log t + 9\alpha\frac{1+t^{2\alpha}}{1-t^{2\alpha}}-\frac{15}{\log t}+\frac{1-2t^2}{1+2t^2}
%\right).
\end{align*}

Let $\alpha\in(0,1)$ and $t\in(\frac{2}{5},1)$. 
By Lemma \ref{basic-monotone-t-2alpha}, we have $\alpha\frac{1+t^{2\alpha}}{1-t^{2\alpha}}>\frac{1}{-\log t}$. 
By Lemma \ref{basic-monotone-t-alpha}, we have $\frac{1}{t^\alpha}-t^\alpha>-2\alpha\log t$. 
Therefore,  
\begin{align*}
\frac{\partial^3 G}{\partial\alpha^3}
&>\left(t^\alpha-\tfrac{1}{t^\alpha} \right)(\log t)^3
\left( 
\log t-\frac{8}{\log t}+\frac{1-2t^2}{1+2t^2}
\right)\\
&>2\alpha(\log t)^4\left( 
\log t-\frac{8}{\log t}+\frac{1-2t^2}{1+2t^2}
\right). 
\end{align*}

For $k=0, 1$, define $G_k=\frac{(1+2t^2)^3}{4}(I - PW_k^2)$. Then $$
\frac{\partial^3 G_k}{\partial\alpha^3}=  
\frac{\partial^3 G}{\partial\alpha^3}+ 6(t-1)^5\left(1-\frac{t}{2}\right)\alpha W_k^2.
$$
We now verify $\frac{\partial^3 G_k}{\partial\alpha^3}>0$. 
By the estimate for $\frac{\partial^3 G}{\partial\alpha^3}$ and the fact that $W_1^2> W_0^2$, it suffices to prove 
\begin{equation}
\label{log-W1}
(\log t)^4\left( 
\log t-\frac{8}{\log t}+\frac{1-2t^2}{1+2t^2}
\right)+3(t-1)^5\left(1-\frac{t}{2}\right) W_1^2>0.
\end{equation}
% Here, the left-hand side of \eqref{log-W1} is exactly $\frac{1}{2}\frac{\partial^4 G_1}{\partial\alpha^4}(0,t)$. 
Additionally, we need to prove 
\begin{equation}
\label{log-Gk}
G_k(0,t) = 2\log t + \frac{2(1+2t)(1-t)}{1+2t^2} 
- \frac{(t-1)^4(\frac{1}{3}+(1-\frac{t}{2})(t-1))}{4}W_k^2
<0.
\end{equation}
Combining these with 
$$
G_k(1,t)\equiv0,\quad 
\frac{\partial G_k}{\partial\alpha}(0,t)\equiv0,
$$
we can conclude, in a manner analogous to Figure \ref{derivatives-g}, that $G_k<0$. 
See also Example \ref{ex-section-I-W-P}. 
The proofs of the univariate inequalities \eqref{log-W1} and \eqref{log-Gk} are given in \S\ref{section-log}. 

For $\alpha\in(-1,0)$, we have $G_k(\alpha,t)=G_k(-\alpha,t)<0$. 
For $\alpha=0$, the result follows from \eqref{log-Gk}. 
This completes the proof of \eqref{I-W0-P} and \eqref{I-W1-P}. 
\end{proof}

\subsection{Estimates for $v_t$ and $\frac{v_t}{v}$}
\label{sec-vt-vt-v}

Recall that 
\begin{align*}
v   &= \frac{(1-t^{1-\alpha})(1-t^{1+\alpha})}{1+2 t^2} =
\frac{t}{1+2 t^2}\Big[
t+\tfrac{1}{t}-\left(t^\alpha+\tfrac{1}{t^\alpha}\right)
\Big],\\
v_t &=
\frac{-1}{1+2t^2}\left[
\frac{1 - 2 t^2}{1 + 2 t^2}\left(t^\alpha+\tfrac{1}{t^\alpha}\right) 
+ \alpha \left(t^\alpha-\tfrac{1}{t^\alpha} \right)
+ \frac{2 t}{ 1 + 2 t^2 }
\right].
\end{align*}
Recall also that $V_0$ and $V_1$ are defined in Lemma \ref{estimate-H2}, whose derivatives are denoted by $V'_0$ and $V'_1$ respectively. 

\begin{lemma}
[Estimates for $v_t$]
\label{estimate-vt}
For $t>0$, we have the limit  
$$\lim\limits_{\alpha\to1} \frac{v_t}{1-\alpha^2} 
= \frac{t^2-1}{2t(1+2 t^2)} + \frac{3t}{(1+2t^2)^2} \log t=V'_1. $$
For $(\alpha,t)\in(-1,1)\times(\frac{1}{8},1)$, $$V'_0=\frac{2(2t+1)(t-1)}{(1+2t^2)^2}\geq \frac{v_t}{1-\alpha^2}>V'_1,$$
where equality holds if and only if $\alpha=0$. 
\end{lemma}

\begin{proof}
The limit can be obtained directly, and the inequality is proved similarly to Lemma \ref{estimate-H1}. 
Let $g=v_t-(1-\alpha^2) V'_1$. 
Then \begin{align*}
\frac{\partial g}{\partial \alpha} 
&= \frac{-\log t}{1+2t^2}\left[
\left(\frac{1}{\log t}+\frac{1 - 2 t^2}{1 + 2 t^2}\right)\left(t^\alpha-\tfrac{1}{t^\alpha}\right) 
+ \alpha \left(t^\alpha+\tfrac{1}{t^\alpha} \right)
\right]
 + 2\alpha V'_1
,\\
\frac{\partial^2 g}{\partial \alpha^2} 
&= \frac{-(\log t)^2}{1+2t^2}\left[
\left(\frac{2}{\log t}+\frac{1 - 2 t^2}{1 + 2 t^2}\right)\left(t^\alpha+\tfrac{1}{t^\alpha}\right) 
+ \alpha \left(t^\alpha-\tfrac{1}{t^\alpha} \right)
\right]
 + 2 V'_1
,\\
\frac{\partial^3 g}{\partial \alpha^3} 
&= \frac{-(\log t)^3}{1+2t^2}\left[
\left(\frac{3}{\log t}+\frac{1 - 2 t^2}{1 + 2 t^2}\right)\left(t^\alpha-\tfrac{1}{t^\alpha}\right) 
+ \alpha \left(t^\alpha+\tfrac{1}{t^\alpha} \right)
\right]\\
&=\frac{-(\log t)^3}{1+2t^2}\left(\tfrac{1}{t^\alpha}-t^\alpha\right) \left(
-\frac{3}{\log t}-\frac{1 - 2 t^2}{1 + 2 t^2}
+ \alpha\frac{1+t^{2\alpha}}{1-t^{2\alpha}}
\right). 
\end{align*}
For $(\alpha,t)\in(0,1)\times(\frac{1}{8},1)$, since $-\frac{3}{\log t}-\frac{1 - 2 t^2}{1 + 2 t^2}>-\frac{3}{\log t}-1>0$, we have $\frac{\partial^3 g}{\partial \alpha^3} >0$. 
Thus, $\frac{\partial^2 g}{\partial \alpha^2}$ is monotonically increasing in $\alpha$. Note that 
$$
%\label{monotone-game-vt-1}
g(1,t)\equiv0,\quad 
\frac{\partial g}{\partial \alpha}(0,t)\equiv0,\quad
\frac{\partial g}{\partial \alpha}(1,t)\equiv0.
$$
Similarly to Figure \ref{derivatives-g}, we conclude $g>0$ (i.e., $\frac{v_t}{1-\alpha^2}>V'_1$) and $g(0,t)>0$ (i.e., $V'_0>V'_1$).  

To obtain the inequality for $V'_0$, let $g=v_t-(1-\alpha^2) V'_0$. 
Then 
$$
% \label{monotone-game-vt-2}
g(0,t)\equiv0,\quad 
g(1,t)\equiv0,\quad
\frac{\partial g}{\partial \alpha}(0,t)\equiv0.
$$
A similar argument shows that for $(\alpha,t)\in(0,1)\times(\frac{1}{8},1)$, we have $g<0$, i.e., $\frac{v_t}{1-\alpha^2}<V'_0$. 

Finally, the case $\alpha=0$ follows from $V'_0>V'_1$ and the case $\alpha\in(-1,0)$ follows by symmetry. 
\end{proof}

\begin{lemma}
[Estimates for $\frac{v_t}{v}$]
\label{estimate-vtv}
For $t\in(0,1)$, we have the limit 
$$\lim\limits_{\alpha\to1} \frac{v_t}{v} 
= \frac{6t}{(1+2t^2)(t^2-1)} + \frac{1}{t\log t} =\frac{V'_1}{V_1}. $$
For $(\alpha,t)\in(-1,1)\times(\frac{1}{5},1)$, $$\frac{V'_1}{V_1}<\frac{v_t}{v}\leq \frac{V'_0}{V_0}=\frac{2(1+2t)}{(1+2t^2)(t-1)},$$
where equality holds if and only if $\alpha=0$. 
\end{lemma}

\begin{proof}
The limits can be obtained directly. To prove the inequality, it suffices to show that for any $t\in(\frac{1}{5},1)$, the function $\frac{v_t}{v}$ is monotonically decreasing in $\alpha\in(0,1)$. 
From 
\begin{align*}
\frac{v_t}{v}&=\frac{\partial}{\partial t}\log v
= \frac{\partial}{\partial t}( \log(1-t^{1+\alpha}) +\log(1-t^{1-\alpha}) - \log(1+2 t^2))\\
&=
-\frac{(1+\alpha)t^{\alpha}}{1-t^{1+\alpha}}
-\frac{(1-\alpha)t^{-\alpha}}{1-t^{1-\alpha}}
-\frac{4t}{1+2t^2}\\
&= -\frac{1}{t}\left(\frac{t^{1+\alpha}}{1-t^{1+\alpha}}+\frac{t^{1-\alpha}}{1-t^{1-\alpha}}
+\alpha\left(\frac{t^{1+\alpha}}{1-t^{1+\alpha}}-\frac{t^{1-\alpha}}{1-t^{1-\alpha}}\right)
+\frac{4t^2}{1+2t^2}
\right), 
\end{align*}
we obtain 
\begin{align*}
\frac{\partial}{\partial \alpha}\frac{-t v_t}{v}&=
\left(\frac{t^{1+\alpha}}{(1-t^{1+\alpha})^2}-\frac{t^{1-\alpha}}{(1-t^{1-\alpha})^2}\right)\log t
+ \frac{t^{1+\alpha}}{1-t^{1+\alpha}}-\frac{t^{1-\alpha}}{1-t^{1-\alpha}}\\
&\ \ \ \ + \alpha \left(\frac{t^{1+\alpha}}{(1-t^{1+\alpha})^2}+\frac{t^{1-\alpha}}{(1-t^{1-\alpha})^2}\right)\log t. 
\end{align*}
Multiplying by $\frac{(1-t^{1+\alpha})^2(1-t^{1-\alpha})^2}{t^2}$ yields 
$$I:=
A\left(t^\alpha-\tfrac{1}{t^\alpha}\right)
-\left(t^{2\alpha}-\tfrac{1}{t^{2\alpha}}\right)
+B\alpha\left(t^\alpha+\tfrac{1}{t^\alpha}\right)
-4\alpha\log t, 
$$
where 
$$
A = t+\tfrac{1}{t}-\left(t-\tfrac{1}{t}\right)\log t, \quad B = \left(t+\tfrac{1}{t}\right)\log t. 
$$
% matlab验证
%syms a t
%v = (1-t^(1+a))*(1-t^(1-a))/(1+2*t^2);
%I = diff(diff(v,t)/v *(-t),a) * (1-t^(1+a))^2*(1-t^(1-a))^2/ t^2;
%A = t+1/t-(t-1/t)*log(t);
%B = (t+1/t)*log(t);
%I2 = A*(t^a-1/t^a) - (t^(2*a)-1/t^(2*a))+B*a*(t^a+1/t^a)-4*a*log(t);
%simplify(I-I2)
We aim to prove $I>0$. 
Differentiating $I$ with respect to $\alpha$ successively gives 
\begin{align*}
\frac{\partial I}{\partial\alpha}&=
(A\log t+B)\left(t^\alpha+\tfrac{1}{t^\alpha}\right)
-\left(t^{2\alpha}+\tfrac{1}{t^{2\alpha}}\right)(2\log t)+B\alpha\left(t^\alpha-\tfrac{1}{t^\alpha}\right)\log t
-4\log t, \\
\frac{\partial^2 I}{\partial\alpha^2}&=
(A(\log t)^2+2B\log t)\left(t^\alpha-\tfrac{1}{t^\alpha}\right)
-\left(t^{2\alpha}-\tfrac{1}{t^{2\alpha}}\right)(2\log t)^2
 +B\alpha\left(t^\alpha+\tfrac{1}{t^\alpha}\right)(\log t)^2,\\
\frac{\partial^3 I}{\partial\alpha^3}&=
(A(\log t)^3+3B(\log t)^2)\left(t^\alpha+\tfrac{1}{t^\alpha}\right)
-\left(t^{2\alpha}+\tfrac{1}{t^{2\alpha}}\right)(2\log t)^3
 +B\alpha\left(t^\alpha-\tfrac{1}{t^\alpha}\right)(\log t)^3,\\
\frac{\partial^4 I}{\partial\alpha^4}&=
(A(\log t)^4+4B(\log t)^3)\left(t^\alpha-\tfrac{1}{t^\alpha}\right)
-\left(t^{2\alpha}-\tfrac{1}{t^{2\alpha}}\right)(2\log t)^4
 +B\alpha\left(t^\alpha+\tfrac{1}{t^\alpha}\right)(\log t)^4.
\end{align*}
By Lemma \ref{basic-monotone-t-2alpha}, we have $\alpha\frac{1+t^{2\alpha}}{1-t^{2\alpha}}<\frac{1+t^2}{1-t^2}$. 
Using $t^\alpha+\frac{1}{t^\alpha}\geq2$ and $t\in(\frac{1}{5},1)$, we obtain 
\begin{align*}
&\ \ \ \ \frac{1}{(\log t)^4 (t^\alpha-\frac{1}{t^\alpha})} \frac{\partial^4 I}{\partial\alpha^4}\\
&=5\left(t+\tfrac{1}{t}\right) -\left(t-\tfrac{1}{t}\right)\log t
-16\left(t^\alpha+\tfrac{1}{t^\alpha}\right)
+\alpha\frac{1+t^{2\alpha}}{1-t^{2\alpha}}\left(t+\tfrac{1}{t}\right)(-\log t)\\
&<5\left(t+\tfrac{1}{t}\right) -\left(t-\tfrac{1}{t}\right)\log t
-32
+\frac{1+t^2}{1-t^2}\left(t+\tfrac{1}{t}\right)(-\log t)\\
&< 5\times5.2 -\left(t-\tfrac{1}{t}\right)\log t
-32
+5.2\frac{1+t^2}{1-t^2}(-\log t)\\
&= \frac{26}{5}\frac{1+t^2}{1-t^2}(-\log t) - \left(t-\tfrac{1}{t}\right)\log t - 6<0. 
\end{align*}
Here, the last inequality holds for all $t>0$, whose proof is straightforward and omitted here.  
Since $t^\alpha-\frac{1}{t^\alpha}<0$, we have $\frac{\partial^4 I}{\partial\alpha^4}>0$. Hence, $\frac{\partial^3 I}{\partial\alpha^3}$ is monotonically increasing in $\alpha$. 
Note that 
$$I(0,t)\equiv0,\quad
I(1,t)\equiv0,\quad
\frac{\partial I}{\partial\alpha}(1,t)\equiv0,\quad
\frac{\partial^2 I}{\partial\alpha^2}(0,t)\equiv0,$$
which is the same as \eqref{monotone-game-u-1}. 
Reversing all signs in Figure \ref{derivatives-g} yields $I>0$, i.e., $\frac{\partial}{\partial \alpha}\frac{v_t}{v}<0$. 
The case $\alpha=0$ is clear and the case $\alpha\in(-1,0)$ follows by symmetry.  
\end{proof}

\begin{lemma}
\label{basic-monotone-t-2alpha}
For any $\alpha,t\in(0,1)$, we have $$\frac{1}{-\log t}<\alpha\frac{1+t^{2\alpha}}{1-t^{2\alpha}}<\frac{1+t^2}{1-t^2}.$$
\end{lemma}

\begin{proof}
Since $\lim\limits_{\alpha\to0}\alpha\frac{1+t^{2\alpha}}{1-t^{2\alpha}}= \frac{1}{-\log t}$, it suffices to show that $\frac{\partial}{\partial\alpha}\frac{\alpha(1+t^{2\alpha})}{1-t^{2\alpha}}>0$, i.e., $$I:=t^{-2\alpha}-t^{2\alpha}+4\alpha\log t>0.$$
This follows from $I(0,t)\equiv0$ and $\frac{\partial I}{\partial\alpha}=2(t^{-2\alpha}+t^{2\alpha}-2)(-\log t)>0$. 
\end{proof}

\begin{lemma}
\label{basic-monotone-t-alpha}
For any $\alpha,t\in(0,1)$, we have $$-2\log t<\frac{t^{-\alpha}-t^\alpha}{\alpha}<t^{-1}-t.$$
\end{lemma}

\begin{proof}
Since $\lim\limits_{\alpha\to0}\frac{t^{-\alpha}-t^\alpha}{\alpha}= -2\log t$, it suffices to show that $\frac{\partial}{\partial\alpha}\frac{t^{-\alpha}-t^\alpha}{\alpha}>0$. 
After differentiation and simplification, it is equivalent to 
$$\alpha\frac{1+t^{2\alpha}}{1-t^{2\alpha}}>\frac{1}{-\log t},$$
which follows from Lemma \ref{basic-monotone-t-2alpha}.  
\end{proof}

\section{The biased two-point hypercontractivity: Proof of Theorem~\ref{thm-biased-alpha}}\label{sec-biased}

Recall that the $\lambda$-biased Bernoulli random variable $X=X_\lambda$ $(0<\lambda<1)$ defined in \eqref{eqn-biased-rw}:
\[
\mathbb{P}(X=1-\lambda)=\lambda,\quad \mathbb{P}(X=-\lambda)=1-\lambda.
\]
Then
\[
\|1+\rho X\|_p=\left(\lambda |1+ (1-\lambda) \rho |^p
+(1-\lambda)|1-\lambda \rho|^p\right)^{\frac{1}{p}}.
\]
We want to find the optimal constant  \(\sigma_{p,q}(\lambda)\in [0,1]\) such that
\[
\|1+r\rho X\|_q\leq \|1+\rho X\|_p \quad \text{for all $0\le r\leq \sigma_{p,q}(\lambda)$ and all $\rho\in \mathbb{R}$}.
\]
We only need to consider $\rho$'s such that $1+\rho X\geq 0$.
%In this case,
%\[
%\|1+rbX\|_q=\Big(\lambda |1+rb(1-\lambda)|^q+(1-\lambda)|1-rb\lambda|^q\Big)^{\frac{1}{q}}
%\]
Thus we only need to consider 
$$
% \label{eqn-bias-gao-fangcheng-1}
\Big(\lambda (1+(1-\lambda)r\rho)^q+(1-\lambda)(1-\lambda r\rho)^q\Big)^{\frac{1}{q}}\leq  \Big(\lambda (1+(1-\lambda)\rho)^p+(1-\lambda)(1-\lambda\rho)^p\Big)^{\frac{1}{p}}$$
with 
\begin{align}\label{eqn-bias-gao-fangcheng-2}
-\frac{1}{1-\lambda}\leq \rho\leq \frac{1}{\lambda}.
\end{align}
%\[
%-\frac{1}{1-\lambda}\leq b\leq \frac{1}{\lambda}
%\]
%\section{Log-Sobolev inequality}
As in \eqref{def-Grrho}, we define the defect function
\begin{align*}
G(\lambda,r,\rho) := G(\lambda, p,q,r,\rho)=
&\Big(\lambda (1+(1-\lambda)\rho)^p+(1-\lambda)(1-\lambda\rho)^p\Big)^{\frac{1}{p}}\\
-&\Big(\lambda (1+(1-\lambda) r\rho)^q+(1-\lambda)(1-\lambda r\rho)^q\Big)^{\frac{1}{q}}
\end{align*}
with $r\in (0,1)$ and $\rho$ satisfying \eqref{eqn-bias-gao-fangcheng-2}.
\begin{lemma}\label{lemma-non-trivial-ffgg-cao-fan-han-qiu-wang}
For $\lambda\in (0,1/2)$ and $1<p<q<\infty$, one has
$
0<\sigma_{p,q}(\lambda)<\sqrt{(p-1)/(q-1)}.
$
\end{lemma}
\begin{proof}
The proof is almost the same as that for Lemma~\ref{lem-rough}. We only need to take the Taylor expansion of $G(\lambda,r,\rho)$ at $\rho=0$:
\[
G = \tfrac{\lambda(1-\lambda)}{2}\Big[(p-1) - (q-1)r^2\Big] \rho^2 + \tfrac{\lambda(1-\lambda)(1-2\lambda)}{6}\Big[(p-1)(p-2) - (q-1)(q-2)r^3\Big] \rho^3 + O(\rho^4). 
\]
The remaining proof is the same as that in Lemma~\ref{lem-rough}.
\end{proof}

\begin{lemma}
\label{lemma-non-trivial-fg-cao-fan-han-qiu-wang}
For $\lambda\in (0,1)\setminus\{1/2\}$ and $1<p<q<\infty$, one has
\[
G(\lambda,\sigma_{p,q}(\lambda),\tfrac{1}{\lambda})>0\quad\text{and}\quad G(\lambda,\sigma_{p,q}(\lambda),-\tfrac{1}{1-\lambda})>0.
\]
\end{lemma}

\begin{proof}
By the definition of $\sigma:=\sigma_{p,q}(\lambda)$, we have
$G(\lambda,\sigma,\frac{1}{\lambda})\geq 0$. 
Now we prove that $$G(\lambda,\sigma,\tfrac{1}{\lambda})>0.$$ 
Assume by contradiction that $G(\lambda,\sigma,\tfrac{1}{\lambda})=0$. Then the left derivative of $G(\lambda,\sigma,\cdot)$ at $\rho=\tfrac{1}{\lambda}$ is non-positive. 
That is,
\[
G(\lambda,\sigma,\tfrac{1}{\lambda})=0
\quad \text{and} \quad 
\tfrac{\partial G}{\partial \rho}(\lambda,\sigma,(\tfrac{1}{\lambda})^-)\leq 0.
\]
Directive computation gives 
$$\left\{
\begin{aligned}
&\lambda^{\frac{q}{p}-q}=\lambda\left(1+\tfrac{1-\lambda}{\lambda}\sigma\right)^q+(1-\lambda)(1-\sigma)^q,\\
&\lambda^{\frac{q}{p}-q}\leq \sigma \left[\left(1+\tfrac{1-\lambda}{\lambda}\sigma\right)^{q-1}-(1-\sigma)^{q-1}\right].
\end{aligned}
\right.$$
It follows that 
$$\lambda\left(1+\tfrac{1-\lambda}{\lambda}\sigma\right)^q+(1-\lambda)(1-\sigma)^q
\leq \sigma \left[\left(1+\tfrac{1-\lambda}{\lambda}\sigma\right)^{q-1}-(1-\sigma)^{q-1}\right],  
$$
and hence 
\[
\left(1+\tfrac{1-\lambda}{\lambda}\sigma\right)^{q-1}\lambda(1-\sigma)+(1-\sigma)^{q-1}[(1-\lambda)(1-\sigma)+\sigma]\leq 0.
\]
This is clearly impossible, since $\lambda,\sigma\in(0,1)$. Therefore $G(\lambda,\sigma_{p,q}(\lambda),\tfrac{1}{\lambda})>0$. 

Notice that 
$$G(\lambda,r,\rho)=G(1-\lambda,r,-\rho)\quad\text{and}\quad 
\sigma_{p,q}(\lambda)=\sigma_{p,q}(1-\lambda).$$
Combing the previous discussion, one has
\[
G(\lambda,\sigma_{p,q}(\lambda),-\tfrac{1}{1-\lambda})=G\left(1-\lambda,\sigma_{p,q}(1-\lambda),\tfrac{1}{1-\lambda}\right)>0.
\]
This completes the whole proof.
\end{proof}

\begin{lemma}
\label{lemma-cao-fan-han-qiu-wang-lemma}
There exists $\rho_0\in (-\frac{1}{1-\lambda},\frac{1}{\lambda})\setminus\{0\}$ such that 
\[
G(\lambda,\sigma_{p,q}(\lambda),\rho_0)=0
\quad\text{and}\quad 
\tfrac{\partial G}{\partial\rho} (\lambda,\sigma_{p,q}(\lambda),\rho_0)=0.
\]
\end{lemma}

\begin{proof}
Using Lemmas~\ref{lemma-non-trivial-ffgg-cao-fan-han-qiu-wang} and \ref{lemma-non-trivial-fg-cao-fan-han-qiu-wang}, the proof is exactly the same as Proposition~\ref{prop-nontrivial-extremizer}. 
\end{proof}

\begin{proof}
[Proof of Theorem~\ref{thm-biased-alpha}]
By Lemma~\ref{lemma-cao-fan-han-qiu-wang-lemma}, the pair $(r,\rho)=(\sigma_{p,q}(\lambda),\rho_0)$ satisfies the following system:
\begin{align*}
\left\{
\begin{aligned}
&\Big(\lambda (1+(1-\lambda)\rho)^p+(1-\lambda)(1-\lambda\rho)^p\Big)^{\frac{1}{p}} = \Big(\lambda (1+(1-\lambda)r\rho)^q+(1-\lambda)(1-\lambda r\rho)^q\Big)^{\frac{1}{q}}, \\
&\rho\frac{(1+(1-\lambda)\rho)^{p-1} - (1-\lambda\rho)^{p-1}}{\lambda (1+(1-\lambda)\rho)^p+(1-\lambda)(1-\lambda\rho)^p} 
= r\rho \frac{(1+(1-\lambda)r\rho)^{q-1} - (1-\lambda r\rho)^{q-1}}{\lambda (1+(1-\lambda)r\rho)^q+(1-\lambda)(1-\lambda r\rho)^q}.
\end{aligned}
\right.
\end{align*}
Define the following change of variables:
\begin{align}
\label{new-cov}
 x=\frac{1-\lambda\rho}{1+(1-\lambda)\rho}, \quad y=\frac{1-\lambda r\rho}{1+(1-\lambda)r\rho}.
\end{align}
When
\[
0<r<1,\quad -\frac{1}{1-\lambda}<\rho<\frac{1}{\lambda},
\]
the range of $(x,y)$ is
\[
\begin{cases}
x<y<1, & \text{when } 0 < x<1, \\
y=1, & \text{when } x=1, \\
1<y<x, & \text{when } x>1.
\end{cases}
\]
Then $(x,y)$ satisfies the following equation system:
\begin{align*}
\left\{
\begin{aligned}
&\frac{\big(\lambda + (1-\lambda) x^p\big)^{\frac{1}{p}}}{\lambda + (1-\lambda)x} = \frac{\big(\lambda + (1-\lambda) y^q\big)^{\frac{1}{q}}}{\lambda + (1-\lambda)y},\\
&\frac{(1-x)(1 - x^{p-1})}{\lambda + (1-\lambda) x^p} = \frac{(1-y)(1 - y^{q-1})}{\lambda + (1-\lambda) y^q},
\end{aligned}
\right.
\end{align*}
which is the desired equation-system \eqref{eqn-biased-duiou-system}.  Reversing the change-of-variable  \eqref{new-cov}, one obtains
\[
\sigma_{p,q}(\lambda)=\frac{1-y}{1-x} \cdot \frac{\lambda + (1-\lambda)x}{\lambda + (1-\lambda)y}, 
\]
as desired.
\end{proof}

%%%%%%%%%%%%%%%%%%%%%%%%%%%%%%%%%%%%%%%%%%%%%%%%%%%%%%
%%%%%%%%%%%%%%%%%%%%%%%%%%%%%%%%%%%%%%%%%%%%%%%%%%%%%%
%%%%%%%%%%%%%%%%%%%%%%%%%%%%%%%%%%%%%%%%%%%%%%%%%%%%%%
%%%%%%%%%%%%%%%%%%%%%%%%%%%%%%%%%%%%%%%%%%%%%%%%%%%%%%
%%%%%%%%%%%%%%%%%%%%%%%%%%%%%%%%%%%%%%%%%%%%%%%%%%%%%%
\appendix

\section{A three‑variable inequality}
\label{section-psi}

In this part, we will prove a three‑variable inequality from the monotonicity lemma (see Lemma~\ref{lemma-reduce-p-q-2}). 
The main idea of the proof is to reduce multivariate inequalities to inequalities involving fewer variables. 
Recall that 
$$
\Delta=\bigg\{x+iy  ~\bigg|~ y\in\Big[0,\frac{\sqrt{3}}{4}\Big],\ x\in\Big[\frac{\sqrt{3}}{3}y, 1-\sqrt{3}y\Big]\bigg.\bigg\}. 
$$

\begin{lemma}
\label{lemma-ineq-psi}
For $1<q<2$ and $t>-1$, define 
$$A(q,t) = (1+t)^{q-1}\left(\log(1+t)-\frac{t}{2(1+t)}\right).$$
For $1<q<2$ and $x+i y\in \Delta$, let 
\begin{align*}
\psi(q,x,y) = A(q,2x)\ (-2y) + A(q,-x-\sqrt{3}y)\ (-\sqrt{3}x+y) + A(q,-x+\sqrt{3}y)\ (\sqrt{3}x+y).
\end{align*}
Then for any $1<q<2$ and $x+i y\in \operatorname{int}(\Delta)$, we have $$\psi(q,x,y) < 0.$$
\end{lemma}

{\bf Observations.} Plotting shows that $\psi$ is not necessarily monotone in $q$ and $y$, but we have $\psi(q,\frac{\sqrt{3}}{3}y,y)\equiv0$ and $\psi$ is decreasing in $x\in[\frac{\sqrt{3}}{3}y, 1-\sqrt{3}y ]$. See Figure \ref{fig-region-partition}. 
This motivates an analysis of the partial derivatives of $\psi$ with respect to $x$. 

\begin{figure}[htbp]
\centering
\includegraphics{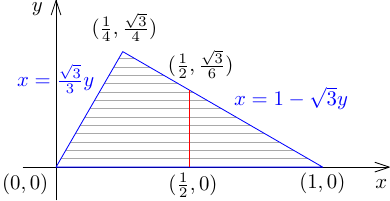}
\caption{Partition of $\operatorname{int}(\Delta)$ by the line $x=\frac{1}{2}$}
\label{fig-region-partition}
\end{figure}

{\bf Facts.}
For $k\geq1$ one can inductively obtain 
\begin{align*}
\frac{\partial^k \psi}{\partial x^k}
&=  \alpha_k(q,x,y) + \beta_k(q,x,y) + \gamma_k(q,x,y),\\
\alpha_k&:=(2^k-(-1)^k) \frac{\partial^k A}{\partial t^k}(q,2x)(-2y),\\
\beta_k&:= (-1)^k\bigg(\frac{\partial^k A}{\partial t^k}(q,2x)(-2y) 
+ \frac{\partial^k A}{\partial t^k}(q,-x-\sqrt{3}y)(-\sqrt{3}x+y)\\
&\ \ \ \ 
+ \frac{\partial^k A}{\partial t^k}(q,-x+\sqrt{3}y)(\sqrt{3}x+y)\bigg),\\
\gamma_k&:= (-1)^k k\sqrt{3}\left(\frac{\partial^{k-1} A}{\partial t^{k-1}}(q,-x-\sqrt{3}y)
- \frac{\partial^{k-1} A}{\partial t^{k-1}}(q,-x+\sqrt{3}y)\right). 
\end{align*}
The specific form of $\beta_k$ is designed to facilitate the application of the orthogonal basis method (Lemma \ref{lem-orthogonal}). 
Furthermore, a second induction gives 
\begin{align*}
A_k(q,t)&:=(1+t)^{k+1-q}
\left(\prod_{j=1}^k\frac{1}{q-j}\right) \frac{\partial^k A}{\partial t^k}\\
&= \log(1+t) + \frac{1}{2(1+t)} - \frac{1}{2} 
- \frac{k}{2(1+t)(q-1)}
+ \sum_{j=1}^k\frac{1}{q-j}. 
\end{align*}

\begin{lemma}
\label{lemma-Ak}
Let $k$ be a positive integer and let $q\in(1,2)$. 
Then for every $$t\in\Big(-1,\frac{k}{2}-1\Big],$$ 
we have $\frac{\partial^k A}{\partial t^k}(q,t)<0$ if $k$ is odd, and $\frac{\partial^k A}{\partial t^k}(q,t)>0$ if $k$ is even. 
\end{lemma}

\begin{proof}
Since $A_k$ is increasing in $t$, 
\begin{align*}
A_k(q,t)&\leq A_k\left(q,\frac{k}{2}-1\right) 
% = \log\frac{k}{2} + \frac{q-1-k}{(q-1)k}-\frac{1}{2}+ \sum_{j=1}^k\frac{1}{q-j}
 = \log\frac{k}{2} + \frac{1}{k}-\frac{1}{2}+ \sum_{j=2}^k\frac{1}{q-j}\\
& < \log k + \sum_{j=2}^k\frac{1}{1-j} = \log k - \sum_{j=1}^{k-1}\frac{1}{j}\leq 0.  
\end{align*}
If $k$ is odd, then $\prod_{j=1}^k\frac{1}{q-j}>0$ implies $\frac{\partial^k A}{\partial t^k}(q,t)<0$; 
if $k$ is even, then $\prod_{j=1}^k\frac{1}{q-j}<0$ implies $\frac{\partial^k A}{\partial t^k}(q,t)>0$. 
\end{proof}

\begin{lemma}
\label{lemma-beta-k}
Let $k\geq 4$ be a positive integer and let $q\in(1,2)$. 
Then for every $x+i y\in \operatorname{int}(\Delta)$, we have $\beta_k(q,x,y)<0$. 
\end{lemma}

\begin{proof}
By $k\geq4$ and Lemma \ref{lemma-Ak}, we have 
$\frac{\partial^{k+2} A}{\partial t^{k+2}}(q,t)>0$ for any $t\in(-1,2)$ 
or $\frac{\partial^{k+2} A}{\partial t^{k+2}}(q,t)<0$ for any $t\in(-1,2)$. It follows that $\frac{\partial^k A}{\partial t^k}$ is strictly convex/concave on $t\in(-1,2)$. 
Because $x+i y\in \operatorname{int}(\Delta)$, the numbers $2x,-x-\sqrt{3}y,-x+\sqrt{3}y\in(-1,2)$ are pairwise different. 
Observe that the columns of the matrix
\[
\left(\begin{array}{ccc}
1 & 2x & -2y \\[1mm]
1 & -x-\sqrt{3}y & -\sqrt{3}x+y \\[1mm]
1 & -x+\sqrt{3}y &  \sqrt{3}x+y
\end{array}\right)
\]
are mutually orthogonal.
Using the orthogonal basis method (Lemma \ref{lem-orthogonal}) we obtain $\beta_k(q,x,y)\neq0$. 
Fixing $q\in(1,2)$ and $y\in(0,\frac{\sqrt{3}}{4})$, 
we have 
\begin{align*}
\lim_{x\to1-\sqrt{3}y}\beta_k(q,x,y) 
&= C + \lim_{x\to1-\sqrt{3}y}(-1)^k \frac{\partial^k A}{\partial t^k}(q,-x-\sqrt{3}y)(-\sqrt{3}x+y)\\
&= C + (-1)^k (4y-\sqrt{3}) \lim_{t\to-1}\frac{\partial^k A}{\partial t^k}(q,t) = -\infty,
\end{align*}
where $C=C(k,q,y)$ is a constant. 
Therefore $\beta_k<0$. 
\end{proof}

\begin{proof}
[Proof of Lemma \ref{lemma-ineq-psi}]
We prove $\psi<0$ by considering two cases $2x\leq 1$ and $2x\geq 1$, as shown in Figure \ref{fig-region-partition}.

Case 1: $2x\leq 1$. First verify $\frac{\partial^4 \psi}{\partial x^4} < 0$. 
By Lemma \ref{lemma-Ak}, we have $\frac{\partial^4 A}{\partial t^4}>0$ for $t\in(-1,1]$, so $\alpha_4 < 0$ and $\gamma_4 < 0$. 
Combining this with $\beta_4<0$ (see Lemma \ref{lemma-beta-k}) gives $\frac{\partial^4 \psi}{\partial x^4}=\alpha_4+\beta_4+\gamma_4 < 0$. This means $\frac{\partial^3 \psi}{\partial x^3}$ is strictly decreasing in $x$. Proceeding inductively, together with 
\begin{equation}
\label{psi-3-2-1}
\frac{\partial^k \psi}{\partial x^k}\left(q,\frac{\sqrt{3}}{3}y,y\right) < 0,\quad k=3,2,1
\end{equation}
and $\psi(q,\frac{\sqrt{3}}{3}y,y)\equiv0$, we obtain $\psi(q,x,y) < 0$.
The proof of \eqref{psi-3-2-1} will be given later. 

Case 2: $2x\geq 1$. First verify $\frac{\partial^6 \psi}{\partial x^6}<0$. By Lemma \ref{lemma-Ak}, we have $\frac{\partial^6 A}{\partial t^6}>0$ for $t\in(-1,2]$, so $\alpha_6<0$ and $\gamma_6<0$. 
Combining this with $\beta_6<0$ (see Lemma \ref{lemma-beta-k}) gives $\frac{\partial^6 \psi}{\partial x^6}=\alpha_6+\beta_6+\gamma_6 < 0$, which implies that $\frac{\partial^5 \psi}{\partial x^5}$ is strictly decreasing in $x$. 
From Case 1 we already know $\frac{\partial^k \psi}{\partial x^k}\left(q,\frac{1}{2},y\right) < 0$ for $k=4,3,2,1,0$. 
Combined with 
\begin{equation}
\label{psi-5}
\frac{\partial^5 \psi}{\partial x^5}\left(q,\frac{1}{2},y\right)<0, 
\end{equation}
we again obtain $\psi(q,x,y)<0$. 
Inequality \eqref{psi-5} will also be given later. 

Summarizing Cases 1 and 2 we have proved $\psi<0$
\end{proof}

In the following, we will prove the two‑variable inequalities \eqref{psi-3-2-1} and \eqref{psi-5}. 

\begin{proof}
[Proof of \eqref{psi-3-2-1}]
For any positive integer $k$, using $\beta_{k}(q,\frac{\sqrt{3}}{3}y,y)\equiv0$ we obtain 
\begin{align*}
&\frac{\sqrt{3}}{3}\frac{\partial^k \psi}{\partial x^k}\left(q,\frac{\sqrt{3}}{3}y,y\right) = \frac{\sqrt{3}}{3}\alpha_k\left(q,\frac{\sqrt{3}}{3}y,y\right) + \frac{\sqrt{3}}{3}\gamma_k\left(q,\frac{\sqrt{3}}{3}y,y\right)\\
& = (-2^k+(-1)^k) s \frac{\partial^k A}{\partial t^k}(q,s) 
+ (-1)^k k\left(\frac{\partial^{k-1} A}{\partial t^{k-1}}(q,-2s)
- \frac{\partial^{k-1} A}{\partial t^{k-1}}(q,s)\right)\\
&=: \Psi_k(q,s),
\end{align*}
where $q\in(1,2)$ and $s = \frac{2\sqrt{3}}{3}y\in(0,\frac{1}{2})$. 
To prove \eqref{psi-3-2-1}, we need to check $\Psi_k(q,s)<0$ for $k=3,2,1$ successively. 

First we prove $\Psi_3(q,s)<0$. 
Direct calculation gives 
\begin{align*}
\Psi_3
&= -9 s \frac{\partial^3 A}{\partial t^3}(q,s) 
- 3 \frac{\partial^2 A}{\partial t^2}(q,-2s)
+ 3 \frac{\partial^2 A}{\partial t^2}(q,s),\\
\frac{\partial\Psi_3}{\partial s}
&= -9 s \frac{\partial^4 A}{\partial t^4}(q,s) 
+ 6 \frac{\partial^3 A}{\partial t^3}(q,-2s)
- 6 \frac{\partial^3 A}{\partial t^3}(q,s). 
\end{align*}
By Lemma \ref{lemma-Ak}, we have $\frac{\partial^4 A}{\partial t^4}(q,t)>0$ for $t\in(-1,1]$, so 
$\frac{\partial^3 A}{\partial t^3}(q,-2s)< \frac{\partial^3 A}{\partial t^3}(q,s)$ for $s\in(0,\frac{1}{2})$, hence $\frac{\partial\Psi_3}{\partial s}<0$. 
Thus $\Psi_3$ is strictly decreasing in $s$. Together with $\Psi_3(q,0)\equiv0$ this yields $\Psi_3(q,s)< 0$. 

Next we show $\Psi_2(q,s)< 0$. 
We compute 
\begin{align*}
\Psi_2
&= -3 s \frac{\partial^2 A}{\partial t^2}(q,s) 
+ 2 \frac{\partial A}{\partial t}(q,-2s)
- 2 \frac{\partial A}{\partial t}(q,s),\\
\frac{\partial\Psi_2}{\partial s}
&= -3 s \frac{\partial^3 A}{\partial t^3}(q,s) 
- 4 \frac{\partial^2 A}{\partial t^2}(q,-2s)
- 5 \frac{\partial^2 A}{\partial t^2}(q,s),\\
\frac{\partial^2\Psi_2}{\partial s^2}
&= -3 s \frac{\partial^4 A}{\partial t^4}(q,s) 
+ 8 \frac{\partial^3 A}{\partial t^3}(q,-2s)
- 8 \frac{\partial^3 A}{\partial t^3}(q,s). 
\end{align*}
As for $\frac{\partial\Psi_3}{\partial s}<0$, we obtain $\frac{\partial^2\Psi_2}{\partial s^2}<0$. 
By Lemma \ref{lemma-Ak} again, we have $\frac{\partial\Psi_2}{\partial s}(q,0) = -9\frac{\partial^2 A}{\partial t^2}(q,0)<0$, so $\frac{\partial\Psi_2}{\partial s}<0$. 
Then it follows from $\Psi_2(q,0)\equiv0$ that $\Psi_2(q,s)< 0$. 

For $\Psi_1$, 
\begin{align*}
\Psi_1 
&= -3 s \frac{\partial A}{\partial t}(q,s) 
-  h(q,-2s) + h(q,s),\\
\frac{\partial \Psi_1}{\partial s}
&= -3 s \frac{\partial^2 A}{\partial t^2}(q,s) 
+ 2 \frac{\partial A}{\partial t}(q,-2s)
- 2 \frac{\partial A}{\partial t}(q,s).
\end{align*}
From $\frac{\partial \Psi_1}{\partial s} = \Psi_2<0$ and $\Psi_1(q,0)\equiv0$ we get $\Psi_1(q,s) < 0$. 

Consequently, \eqref{psi-3-2-1} holds for $k=3,2,1$. 
\end{proof}

\begin{proof}
[Proof of \eqref{psi-5}]
By Lemma \ref{lemma-beta-k}, we have $\beta_5<0$. 
Therefore 
\begin{align*}
&\frac{\partial^5 \psi}{\partial x^5}\left(q,\frac{1}{2},y\right) < \alpha_5\left(q,\frac{1}{2},y\right) + \gamma_5\left(q,\frac{1}{2},y\right)\\
&=-66y\frac{\partial^5 A}{\partial t^5}(q,1)
-5\sqrt{3} \frac{\partial^4 A}{\partial t^4}\left(q,-\frac{1}{2}-\sqrt{3}y\right) 
+5\sqrt{3} \frac{\partial^4 A}{\partial t^4}\left(q,-\frac{1}{2}+\sqrt{3}y\right)\\
&=:\Psi(q,y),
\end{align*}
where $q\in(1,2)$ and $y\in(0,\frac{\sqrt{3}}{6})$. 
Differentiating with respect to $y$ yields 
\begin{align*}
\frac{\partial \Psi}{\partial y} 
&= -66\frac{\partial^5 A}{\partial t^5}(q,1)
+15 \frac{\partial^5 A}{\partial t^5}\left(q,-\frac{1}{2}-\sqrt{3}y\right) 
+15 \frac{\partial^5 A}{\partial t^5}\left(q,-\frac{1}{2}+\sqrt{3}y\right).
\end{align*}
By Lemma \ref{lemma-Ak}, we have $\frac{\partial^6 A}{\partial t^6}>0$ for $t\in(-1,2]$, so $\frac{\partial^5 A}{\partial t^5}(q,t)$ is strictly increasing in $t$. Hence 
\begin{equation}
\label{ineq-A5}
\frac{\partial \Psi}{\partial y}< -66\frac{\partial^5 A}{\partial t^5}(q,1)
+30\frac{\partial^5 A}{\partial t^5}(q,0). 
\end{equation}
Note that $\frac{\partial^5 A}{\partial t^5}(q,0)<\frac{\partial^5 A}{\partial t^5}(q,1)<0$; see Lemma \ref{lemma-Ak}. 
Computations give 
\begin{align*}
\frac{\partial^5 A}{\partial t^5}(q,0) \prod_{j=1}^5\frac{1}{q-j}
& = -\frac{5}{2(q-1)} + \sum_{j=1}^5 \frac{1}{q-j}\\
&< -\frac{3}{2(q-1)} + \frac{1}{q-2} + \sum_{j=3}^5 \frac{1}{1-j}\\
&= -\frac{3}{2(q-1)} + \frac{1}{q-2} - \frac{13}{12}, 
\\
\frac{\partial^5 A}{\partial t^5}(q,1) \prod_{j=1}^5\frac{1}{q-j}
& > \frac{1}{16}\left(
\log 2 -\frac{1}{4} - \frac{5}{4(q-1)} + \sum_{j=1}^5 \frac{1}{q-j}\right)\\
& > \frac{1}{16}\left( -\frac{1}{4} - \frac{1}{4(q-1)} + \frac{1}{q-2}
+ \sum_{j=3}^5 \frac{1}{2-j}\right)\\
& = \frac{1}{16}\left( - \frac{1}{4(q-1)} + \frac{1}{q-2}
- \frac{25}{12} \right).
\end{align*}
Consequently, 
$$
- \frac{11}{16}\left( - \frac{1}{4(q-1)} + \frac{1}{q-2} - \frac{25}{12} \right)
+ 5\left(-\frac{3}{2(q-1)} + \frac{1}{q-2} - \frac{13}{12}\right)<0
$$
implies $-11\frac{\partial^5 A}{\partial t^5}(q,1)
+5\frac{\partial^5 A}{\partial t^5}(q,0)<0$. 
Combining this with \eqref{ineq-A5} gives $\frac{\partial \Psi}{\partial y}<0$. 
Together with $\Psi(q,0)\equiv0$ we obtain $\Psi(q,y)<0$. 
Therefore \eqref{psi-5} holds. 
\end{proof}

\section{Univariate inequalities involving logarithms}
\label{section-log}

In this section, we prove several previously stated univariate inequalities involving logarithms. 
While their validity is readily confirmed graphically, we establish them rigorously using two principal strategies.
The first approach successively differentiates and transforms the expressions to eliminate the logarithmic terms, ultimately reducing the problem to determining the sign of a polynomial. The second approach transforms both sides of a given inequality to obtain two well-separated functions, then heuristically constructs a lower-degree polynomial that lies between them, thereby decomposing the original inequality into two more tractable ones. 
Owing to substantial computations involved, both methods benefit from the use of symbolic computation software.

\begin{proof}
[Proof of \eqref{A-B-6-logt}]
Let $t\in(\frac{2}{5},1)$. We need to show
$$\frac{4t\log t+\frac{1-2t^2}{t}\log\frac{1+2t^2}{3}}
{-4t\log t+\frac{1+2t^2}{t}\log\frac{1+2t^2}{3}}+
\frac{6}{\log t}<0.$$
Let $x=t^2$. Then it suffices to show that
\begin{equation}
\label{12-logx}
\frac{2x\log x+(1-2x)\log\frac{1+2x}{3}}
{-2x\log x+(1+2x)\log\frac{1+2x}{3}}+
\frac{12}{\log x}<0
\end{equation}
for any $x\in(\frac{1}{10},1)$. 
It is easy to verify that 
\begin{align*}
-2x\log x+(1+2x)\log\frac{1+2x}{3}&<0, \\
12+24x+(1-2x)\log x&>0.
\end{align*}
Hence, the equality \eqref{12-logx} can be simplified to
$$f(x):=\frac{-24x\log x+2x(\log x)^2}{12+24x+(1-2x)\log x}+\log\frac{1+2x}{3}<0.$$
Since $f(1)=0$, we just need to show 
$$f'(x)=
\frac{2
\left((\log x)^2 + 2\frac{1-2x}{1+2x}\log x - 96\right)
\log x
}{(12+24x+(1-2x)\log x)^2} 
>0.$$
It follows from $|\log x|<3$ and $\left|\frac{1-2x}{1+2x}\right|<1$ that 
$$(\log x)^2 + 2\frac{1-2x}{1+2x}\log x - 96<0.$$
So $f'(x)>0$. This proves \eqref{A-B-6-logt}. 
\end{proof}

\begin{proof}
[Proof of \eqref{g00t}]
Let $t\in(\frac{2}{5},1)$. Multiplying the inequality \eqref{g00t} by $\frac{t}{2(1-t)(1+2t)}$, we have 
$$ f_1(t):=\log \frac{1+2t^2}{3} - \frac{4t}{1+2t}\log t + \frac{2(t-1)^3 (1+3(1-\frac{t}{2})(t-1))}{9(1+2t^2)(1+2t)}>0.$$
Since $f_1(1)=0$, it suffices to show that 
\begin{align*}
f_2(t) &:= \frac{(1+2t)^2}{4} f'_1(t) 
=  
\frac{(1-t)(24 t^6 - 30 t^5 - 78 t^4 - 67 t^3 - 75 t^2 - 33 t - 65)}{36(1+2t^2)^2} - \log t
<0.
\end{align*}
After two further differentiations, we obtain
\begin{align*}
f_3(t)
&:= \frac{9t (1 + 2 t^2)^3}{(1 - t)^3 (1+2t)}f'_2(t)
 = 18 t^5 + 18 t^4 + 27 t^3 + 27 t^2 - t - 9,\\
f'_3(t) &= 90 t^4 + 72 t^3 + 81 t^2 + 54 t - 1. 
\end{align*}
Since $f'_3$ is positive when $t\in(\frac{2}{5},1)$, the function $f_3$ monotonically increases. 
Moreover, since $f_3(\frac{2}{5}) = -\frac{8459}{3125}<0$ and $f_3(1) = 80>0$, the function $f_2$ must first decrease and then increase. Combining this with the facts that $f_2(\frac{2}{5}) = -\frac{5087}{5500} - \log\frac{2}{5} < 0$ and $f_2(1)=0$, we conclude that $f_2(t)<0$ for all $t\in(\frac{2}{5},1)$. 
This completes the proof of \eqref{g00t}. 
\end{proof}

\begin{proof}
[Proof of \eqref{g11t}]
Let $t\in(\frac{2}{5},1)$. 
The inequality \eqref{g11t} simplifies to
\begin{align*}
f_1(t)&:=a_0 + a_1\log \frac{1+2t^2}{3}<0, \\
a_0 &:= -8(\log t)^2 + 4(1-t^2)\log t + 
\frac{8(t-1)^4(1+6(t-1)(1-\frac{t}{2}))}{9t(1+2t^2)},\\
a_1 &:= 6\log t + \frac{(1+2t^2)(t^2-1)}{t^2} < 0. 
\end{align*}
Differentiation yields 
\begin{align*}
\frac{d}{dt} \frac{f_1}{a_1} &= 
\frac{8(1-t)}{9 t^4 (1 + 2t^2)^2 a_1 ^2} f_2,\\
f_2(t) &:= b_0 + b_1\log t + b_2(\log t)^2,\\
b_0 &:= (t - 1)^3 (1 + 2 t^2) (6 t^6 + 24 t^5 - 79 t^4 + 34 t^3 - 30 t^2 + 38 t - 5),\\
b_1 &:= 3t (t - 1) (36 t^7 - 108 t^6 + 14 t^5 - 20 t^4 + 24 t^3 - 35 t^2 - 16 t - 3),\\
b_2 &:= 18t (1 + t) (1 + 2 t^2) (1 + 2 t^4) > 0. 
\end{align*}
A further differentiation gives 
\begin{align*}
\frac{d}{dt} \frac{f_2}{b_2} 
&= \frac{d}{dt} \frac{b_0}{b_2} + \frac{b_1}{b_2} \frac{1}{t} 
+ \left(\frac{d}{dt} \frac{b_1}{b_2} + \frac{2}{t}\right)\log t = \frac{18}{b_2^2} f_3,\\
f_3(t)&:= c_0 + c_1\log t,\\
c_0 &:= (t - 1) (2 t^2 + 1)  \big(72 t^{15} + 216 t^{14} + 8 t^{13} - 1632 t^{12} + 1010 t^{11}- 1298 t^{10} + 806 t^9\\
&\ \ \ \   - 260 t^8 - 185 t^7 - 659 t^6  + 817 t^5 - 772 t^4 + 166 t^3 - 244 t^2 + 6 t + 5\big),\\
c_1 &:= 6 t (72 t^{15} + 240 t^{14} - 232 t^{13} + 184 t^{12} - 134 t^{11}+ 396 t^{10} + 98 t^9 \\
&\ \ \ \  + 158 t^8 - 235 t^7 + 404 t^6 - 55 t^5 + 63 t^4 - 53 t^3 + 43 t^2 + 17 t + 6)\\
&>6t[ (- 232 t^{13} + 396 t^{10})
+( - 134 t^{11} + 158 t^8)
+( - 235 t^7 + 404 t^6) \\
&\ \ \ \ + (- 55 t^5 + 63 t^4) + (- 53 t^3 + 43 t^2 + 17 t)]>0. 
\end{align*}
Differentiating once more, we obtain 
\begin{align*}
\frac{d}{dt} \frac{f_3}{c_1} &= 
\frac{12 (t - 1)^3 (t + 1) (2 t^2 + 1) (2 t^4 + 1)}{c_1^2} f_4,\\
f_4(t) &:= 2592 t^{23} + 20736 t^{22} + 44640 t^{21} + 111648 t^{20} + 140288 t^{19}  + 236448 t^{18} + 513408 t^{17} \\
&\ \ \ \ 
  + 399872 t^{16} + 599322 t^{15} + 363480 t^{14} + 478582 t^{13} + 485286 t^{12}  - 22311 t^{11}+ 390104 t^{10}\\
&\ \ \ \ 
  - 132111 t^9 + 66438 t^8 - 35150 t^7 - 23445 t^6 - 27777 t^5 - 10135 t^4  \\
&\ \ \ \ 
- 3198 t^3 - 2079 t^2 - 223 t - 15. 
\end{align*}

Since the coefficients from $t^{12}$ to $t^{23}$ in $f_4$ are all positive, it follows that $\frac{d^{12} f_4}{dt^{12}} > 0$. 
By examining the signs of the successive derivatives at $t=\frac{2}{5}$ and $t=1$ (see Table \ref{table-bd-value-fk-g11t}), working backwards recursively, we conclude that $f_1<0$ (see Figure \ref{fig-fk}). 
\end{proof}

%\begin{table}[ht]
%\centering
%\begin{tabular}{|l|l|l|} 
%\hline
%$k$ & $f_k(\frac{2}{5})$ & $f_k(1)$  
%\\ \hline
%$1$ 
%&
%$- 8(\log\frac{2}{5})^2 + \frac{84}{25}\log\frac{2}{5} 
%- \frac{564}{1375}
%+ (6\log\frac{2}{5} - \frac{693}{100})\log\frac{11}{25}
%\approx -0.00269$
%& 
%0
%\\ \hline
%$2$ 
%&
%$ \frac{5463612}{390625}(\log\frac{2}{5})^2 
%+ \frac{19991286}{1953125}\log\frac{2}{5} 
%- \frac{81080109}{48828125}
%\approx 0.704$
%& 
%0
%\\ \hline
%$3$ 
%&
%$ \frac{6875519483952}{152587890625}\log\frac{2}{5}
%+ \frac{107238939732441}{3814697265625}
%\approx -13.2$
%& 
%0
%\\ \hline
%\end{tabular}
%\caption{Boundary values of $f_1,f_2,f_3$}
%\label{table-bd-value-123}
%\end{table}

\begin{table}[htbp]
\centering
\caption{Boundary values of $f_k$ in the proof of \eqref{g11t}}
\label{table-bd-value-fk-g11t}
\renewcommand{\arraystretch}{1.2}
\begin{tabular}{|l|r|r|} 
\hline
 & $t=\frac{2}{5}$ & $t=1$  
\\ \hline
$f_1$ & $\approx -0.00269$ & $0$
\\ \hline
$f_2$ & $\approx 0.704$ & $0$
\\ \hline
$f_3$ & $\approx -13.2$ & $0$
\\ \hline
$f_4$ 
& $ -\frac{15221744572681889469}{11920928955078125}$
& 3596400
\\ \hline
     $f_5:=f'_4$ 
& $-\frac{25063807070406017691}{2384185791015625}$
& $54275184$
\\ \hline
     $f_6:=f'_5/6$ 
& $-\frac{5221050470773832781}{476837158203125}$
& $130454253$
\\ \hline
     $f_7:=f'_6$
& $-\frac{308644680335820318}{95367431640625}$
& $1807469505$
\\ \hline
     $f_8:= f'_7 / 4$
& $\frac{1580116058314183251}{3814697265625}$
& % $6024859857$
\\ \hline
     $f_9:= f'_8 / 15$
& $\frac{2713516452538527489}{3814697265625}$
& % $5151771285$
\\ \hline
     $f_{10}:= f'_9 / 2$
& $\frac{5231822432670641043}{762939453125}$
& % $31733373933$
\\ \hline
     $f_{11}:= f'_{10} / 14$
& $\frac{1285654319889708897}{152587890625}$
& % $26757635355$
\\ \hline
     $f_{12}:= f'_{11} / 12$
& $\frac{344854420540353128}{30517578125}$
& % $25147949002$
\\ \hline
     $f_{13}:= f'_{12} / 3$
& $\frac{70423837809853329}{1220703125}$
& % $89995986108$
\\ \hline
    $f_{14}:= f'_{13} / 10$
& $\frac{103322515145357454}{1220703125}$
& % $91572756321$
\\ \hline
    $f_{15}:= f'_{14} / 33$
& $\frac{8798669534981667}{244140625}$
& % $26617374429$
\\ \hline
    $f_{16}:= f'_{15} / 8$
& $\frac{2945602219999709}{48828125}$
& % $29887117297$
\\ \hline
\end{tabular}
\end{table}

\begin{figure}[htbp]
\centering
\includegraphics{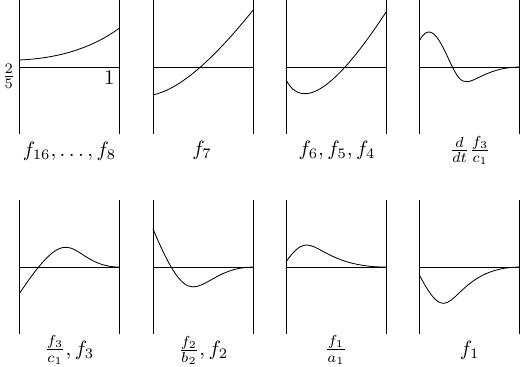}
\caption{Determining the sign of $f_k$ in the proof of \eqref{g11t}}
\label{fig-fk}
\end{figure}

\begin{proof}
[Proof of \eqref{log-W1}]
Let $t\in(\frac{2}{5},1)$. 
Previously, we constructed a rational function $W_1$ satisfying $\frac{V'_1}{V_0}<W_1< \frac{V'_1}{V_1}$. Further, we construct an intermediate function between $\frac{V'_1}{V_0}$ and $W_1$: 
\begin{align*}
W &= \frac{((\frac{5}{9}t-1)(t-1)+\frac{2}{3})(t-1)+2}{(t-1)^3}(\log t)^2 \\
&= \frac{5 t^3 - 19 t^2 + 29 t + 3}{9(t-1)^3}(\log t)^2.
\end{align*}
To prove the inequality \eqref{log-W1}, it suffices to verify 
\begin{align}
\label{W-W1}
&W <W_1<0,\\
\label{log-W}
&(\log t)^4\left( 
\log t-\frac{8}{\log t}+\frac{1-2t^2}{1+2t^2}
\right)+3(t-1)^5\left(1-\frac{t}{2}\right) W^2>0. 
\end{align}

We first prove $W <W_1$. 
After simplification and differentiation, we obtain 
\begin{align*}
f_1(t)&:=(\log t)^2 - \frac{6(t - 1)^2 (t^2 - 4t + 6)}{5 t^3 - 19 t^2 + 29 t + 3}>0,\\
f_2(t) &:= t f'_1(t) = 
2\log t -\frac{6t(t - 1)(5 t^5 - 33 t^4 + 93 t^3 - 83 t^2 - 96 t + 222)}{(5 t^3 - 19 t^2 + 29 t + 3)^2},\\
f_3(t) &:= t f'_2(t) = 
\frac{2 (t - 1)^2}{(5 t^3 - 19 t^2 + 29 t + 3)^3} f_4,\\
f_4(t)&:= - 75 t^8 + 830 t^7 - 4307 t^6 + 12957 t^5 - 24362 t^4  + 25406 t^3 - 12339 t^2 + 2835 t + 27. 
\end{align*}
From the boundary values of the successive derivatives of $f_4$ (Table \ref{table-bd-value-fk-log-W1}), it is straightforward to deduce that $f_4>0$, and hence $f_3>0$. Since $f_2(1)=0$, it follows that $f_2<0$. Further, since $f_1(1)=0$, we conclude that $f_1>0$. This completes the proof of \eqref{W-W1}. 

\begin{table}[ht]
\centering
\caption{Boundary values of $f_4$ in the proof of \eqref{log-W1}}
\label{table-bd-value-fk-log-W1}
\renewcommand{\arraystretch}{1.2}
\begin{tabular}{|r|r|r|r|r|r|r|r|r|} 
\hline
$k$ & $f_4^{(k)}(\frac{2}{5})$ & $f_4^{(k)}(1)$  
& $k$ & $f_4^{(k)}(\frac{2}{5})$ & $f_4^{(k)}(1)$ 
& $k$ & $f_4^{(k)}(\frac{2}{5})$ & $f_4^{(k)}(1)$ 
\\ \hline
$0$ & $\frac{4772277}{15625}$ & $972$ 
& $3$ & $\frac{1759284}{125}$ & $ -22572$
& $6$ & $-1669680$ & $ -429840$
\\ \hline
$1$ & $\frac{1058427}{3125}$ & $ 1080$
& $4$ & $-169440$ & $ -9168$
& $7$ & $2973600$ & $ 1159200$
\\ \hline
$2$ & $\frac{1961442}{625}$ & $ -3996$
& $5$ & $616824$ & $ 41400$
& $8$ & $-3024000$ & $ -3024000$
\\ \hline
\end{tabular}
\end{table}

We now prove the inequality \eqref{log-W}. After simplification, it is equivalent to 
$$\log t-\frac{8}{\log t}+\frac{1-2t^2}{1+2t^2}
+\frac{(1-\frac{t}{2})(5 t^3 - 19 t^2 + 29 t + 3)^2}{27(t-1)}>0.$$
For $t\in(\frac{2}{5},1)$, we have $\log t>\frac{5}{3}(t-1)$ and $\frac{1-2t^2}{1+2t^2}>-\frac{1}{3}$. 
Thus, it suffices to prove
$$\frac{5}{3}(t-1) - \frac{8}{\log t} - \frac{1}{3}
+\frac{(1-\frac{t}{2})(5 t^3 - 19 t^2 + 29 t + 3)^2}{27(t-1)}>0.$$
Next, by introducing the intermediate function $\frac{14t+10}{3}$, we reduce the problem to verifying 
$$%\begin{equation}
%\label{mid-function-14t-10-3}
\frac{8(1-t)}{-\log t}>\frac{14t+10}{3}>\left(2-\frac{5}{3}t\right)(1-t)+\frac{(1-\frac{t}{2})(5 t^3 - 19 t^2 + 29 t + 3)^2}{27}.
$$%\end{equation}
The left inequality is equivalent to $\log t+\frac{12(1-t)}{7t+5}>0$, whose proof is straightforward and omitted here. The right inequality simplifies to 
$$g(t) := 25 t^7 - 240 t^6 + 1031 t^5 - 2374 t^4 + 2871 t^3 - 1370 t^2 + 111 t + 54 > 0.$$
A graphical study of $g$ over $[\frac{2}{5},1]$ shows that its minimum occurs near $t=\frac{1}{2}$. Computing the Taylor expansion of $g$ at $t=\frac{1}{2}$ and using the bound $|2t-1|<1$, we obtain 
\begin{align*}
128 g(t)&= 
789 - 821 (2 t-1) + 14581 (2 t-1)^2 + 2483 (2 t-1)^3 \\
&\ \ \ \  
-4697 (2 t-1)^4 + 1769 (2 t-1)^5 -305 (2 t-1)^6+25 (2 t-1)^7 \\
&> 789 - 821 (2 t-1)  + (14581 - 2483 - 4697 - 1769 - 305 -25)(2 t-1)^2\\
& = 789 - 821 (2 t-1) + 5302 (2 t-1)^2 > 0.
\end{align*}
The positivity of the last quadratic in $2t-1$ follows from its discriminant being negative: 
$$821^2 - 4\times789\times5302 = -16059071<0.$$
This completes the proof of \eqref{log-W}. 

Finally, combining \eqref{W-W1} and \eqref{log-W}, we conclude that \eqref{log-W1} holds. 
\end{proof}

\begin{proof}
[Proof of \eqref{log-Gk}]
Let $t\in(\frac{2}{5},1)$. We aim to prove that $G_0(0,t)<0$ and $G_1(0,t)<0$. 

We first prove $G_0(0,t)<0$. Direct computation gives
\begin{align*}
G_0(0,t) &= 2\log t + \frac{(1+2t)(1 - t)(- 6t^4 + 21 t^3 + 10 t^2 - 5 t + 16)}{6 (1 + 2t^2)^2},\\
\frac{d}{dt} G_0(0,t) &= 
\frac{(1+2t)(1 - t)^2 (24 t^5 - 12 t^4 + 66 t^3 - 74 t^2 + 11 t + 12)}{6t(1+2t^2)^3}>0. 
\end{align*}
Here, $\frac{d}{dt} G_0(0,t)>0$ because
\begin{align*}
&24 t^5 - 12 t^4 + 1= 12t^4(2t-1) + 1
> -12\times 0.5^4\times 0.2 + 1 = 0.85 >0,\\
&66 t^3 - 74 t^2 + 11 t + 11 > 11 (6t^3-7t^2+t+1)>0. 
\end{align*}
The cubic polynomial $6t^3-7t^2+t+1$ attains its minimum on $[\frac{2}{5},1]$ at $t=\frac{7+\sqrt{31}}{18}$, which is positive. 
Since $G_0(0,1)=0$ and $\frac{d}{dt} G_0(0,t)>0$, it follows that $G_0(0,t)<0$. 

Next, we prove $G_1(0,t)<0$. Computation yields 
\begin{align*}
G_1(0,t) &= 2\log t + \frac{t - 1}{54 (1+2 t^2)}(6 t^9 - 72 t^8 + 389 t^7 - 1212 t^6 \\
&\ \ \ \ + 2353 t^5 - 2924 t^4 + 2400 t^3 - 1456 t^2 + 444 t - 252),\\
\frac{d}{dt} G_1(0,t) &= 
\frac{(1 - t)^2}{54 t (1 + 2t^2)^2}f,\\
f&:= 96 t^{10} - 900 t^9 + 3696 t^8 - 8420 t^7 + 11672 t^6 \\
&\ \ \ \ - 11105 t^5 + 8804 t^4 - 5384 t^3 + 2156 t^2 - 480 t + 108. 
\end{align*}
Computing the Taylor expansion of $f$ at $t=\frac{1}{2}$ and using $|2t-1|<1$, we obtain 
\begin{align*}
128 f(t) &= 8523 + 2463 (2 t - 1) + 6476 (2 t - 1)^2 + 9592 (2 t - 1)^3+ 7322 (2 t - 1)^4- 3014 (2 t - 1)^5\\
&\ \ \ \    - 232 (2 t - 1)^6 - 296 (2 t - 1)^7  + 363 (2 t - 1)^8 - 105 (2 t - 1)^9 + 12 (2 t - 1)^{10}\\
&> 8523 - 2463 + (6476 - 9592) (2 t - 1)^2 + (7322 - 3014 - 232 -296 - 105) (2 t - 1)^4\\
&> 6060 - 3116 (2 t - 1)^2 >0. 
\end{align*}
Hence, $\frac{d}{dt} G_1(0,t)>0$. 
Combining this with $G_1(0,1)=0$ yields $G_1(0,t)<0$. 

In conclusion, \eqref{log-Gk} holds for both $k=0$ and $k=1$. 
\end{proof}

\section{The explicit formulas for $r_{p,p^*}(\mathbb{Z}_3)$, $r_{p,2}(\mathbb{Z}_3)$ and $r_{2,p^*}(\mathbb{Z}_3)$ 
}

\subsection{First proof}
Let \(p \in (1, 2)\). As the definitions of \(\Phi\) in \eqref{phi}, we let 
$
\alpha = 1 - \tfrac{2}{p}$ and $\alpha^* = 1 - \tfrac{2}{p^*} = -\alpha.
$
Recall from \eqref{eq-H-alpha-t} that  
$$H(\alpha, t) = \left( -\frac{\alpha}{2} \log \frac{1 + 2t^2}{3} + \frac{1}{2} \log \frac{1 + 2t^{1+\alpha}}{1 + 2t^{1-\alpha}},\ \ \frac{(1 - t^{1-\alpha})(1 - t^{1+\alpha})}{1 + 2t^2} \right).
$$ 
In Proposition \ref{proposition-unique}, we have prove that any distinct $H(\alpha,\cdot)$ have a unique intersection point. 
Particularly, the curves \(H(\alpha, \cdot)\) and \(H(-\alpha, \cdot)\) intersect at a unique point  
\[
H(\alpha,\tfrac{1}{2}) = H(-\alpha,\tfrac{1}{2}) =( 0,\tfrac{(2 - 2^\alpha)(2 - 2^{-\alpha})}{6} ).
\]  
Observe that  
\begin{align}\label{eq:dual-two}
\Phi(\alpha, \tfrac12) = 
( \tfrac{2}{1-\alpha}, 2^{\alpha-1}), \quad 
\Phi(-\alpha, \tfrac12) = ( \tfrac{2}{1+\alpha}, 2^{-\alpha-1}).
\end{align}  
Hence, for the pair \((p, p^*)\), the unique solution to \eqref{xy-eq} is  
$
(x , y) = (2^{\alpha-1},  2^{-\alpha-1}).
$
By Theorem \ref{thm-main-thm-one-xy},  
\[
r_{p,p^*}(\mathbb{Z}_3) = \frac{(1 + 2x)(1 - y)}{(1 + 2y)(1 - x)} = \frac{(1+2^{\alpha})(1-2^{-\alpha-1})}{(1+2^{-\alpha})(1-2^{\alpha-1})} = \frac{2(4^{1/p^*} - 1)}{4 - 4^{1/p^*}}.
\]  

By symmetry, the curves \( H(0, \cdot)\), \( H(\alpha, \cdot)\), and \( H(\alpha^*, \cdot)\) intersect at a unique point.  
That is, there is a unique $t\in(0,1)$ such that 
$H(0,t) = H(\alpha,\tfrac{1}{2}) = H(-\alpha,\tfrac{1}{2}).$
By \(\Phi(0,t) = (2, t)\) and \eqref{eq:dual-two}, for the pairs \((p, 2)\) and \((2, p^*)\), the corresponding solutions of system \eqref{xy-eq} are respectively  
$
 (x_1,y_1) = (2^{\alpha-1},t)$ and $
(x_2,y_2) = (t,2^{-\alpha-1}). 
$
By Theorem \ref{thm-main-thm-one-xy},  
\[
r_{p, 2}(\mathbb{Z}_3) = \frac{(1 + 2^{\alpha})(1 - t)}{(1 + 2t)(1 - 2^{\alpha-1})}
\quad \text{and} \quad
r_{2, p^*}(\mathbb{Z}_3) = \frac{(1 + 2t)(1 - 2^{-\alpha-1})}{(1 + 2^{-\alpha})(1 - t)}.
\]  
Therefore,  
$
r_{p, 2}(\mathbb{Z}_3)\cdot  r_{2, p^*}(\mathbb{Z}_3)= r_{p, p^*}(\mathbb{Z}_3)$.   
Since \(r_{p, 2}(\mathbb{Z}_3) = r_{2, p^*}(\mathbb{Z}_3)\), we obtain  
\[
r_{p, 2}(\mathbb{Z}_3) = r_{2, p^*}(\mathbb{Z}_3) = \sqrt{2(4^{1/p^*} - 1)/(4 - 4^{1/p^*})}.
\]  

\subsection{Second proof}\label{sec-sym} Here we also give the sketch of a second proof of $r_{p,p^*}$ for $p\in (1,2)$ and leave the second proof for $r_{p,2} = r_{2, p^*}$ to the reader. 

Indeed, by the equivalent equation-system \eqref{self-dual-form}, it is clear that the solution $(x,y)\in (0,1)^2$ to the following system is also a solution to \eqref{self-dual-form} and hence \eqref{xy-eq}: 
\[
\ell(p, x)= \ell(p^*, y) \text{\,\,and \,\,} y= x^{p-1}. 
\]
Then one can check easily that $(x,y)=(2^{-2/p}, 2^{-2/p^*})$ is a solution  to \eqref{xy-eq} (which must be unique by Theorem  \ref{thm-main-thm-one-xy}). Thus $r_{p,p^*}(\mathbb{Z}_3)$ is obtained using \eqref{cross-rpq}.  

\section{Proof of Corollary~\ref{coro-algebraic}}\label{sec-daishushu}
\begin{lemma}\label{lemma-daishushu}
Let $P, Q \in \mathbb{Z}[u, v]$ be two polynomials with integer coefficients. 
Suppose that  in an open set  $\Omega \in \mathbb{R}^2$, the system of equations
\begin{equation}\label{PQ-sys}
        P(u, v) =         Q(u, v) = 0
\end{equation}
has a unique real solution  $(u_0, v_0)$. Then both $u_0$ and $v_0$ are algebraic numbers.
\end{lemma}
\begin{proof}
Set
$
V_{\mathbb{R}}(P, Q) = \{ (u, v) \in \mathbb{R}^2: P(u, v) = Q(u, v) = 0 \}. 
$
%The condition that $(u_0, v_0)$ is unique in $\Omega$ implies that $(u_0, v_0)$ is an \textbf{isolated point} of the real algebraic set $V_{\mathbb{R}}(P, Q)$.
Consider the greatest common divisor of the two polynomials:
\[
H(u, v) = \gcd(P(u, v), Q(u, v)) \in \mathbb{Q}[u, v].
\]
Without loss of generality, we may assume $H(u, v) \in \mathbb{Z}[u, v]$. Then 
$
V_{\mathbb{R}}(P, Q) = V_{\mathbb{R}}(H) \cup V_{\mathbb{R}}(\tilde{P}, \tilde{Q}),
$
where $\tilde{P} = P/H$ and $\tilde{Q} = Q/H$ are coprime.  We analyze two cases.

\noindent
\textbf{Case 1}.
Assume that $(u_0, v_0)$ is a solution to the system $\tilde{P} = \tilde{Q} = 0$.

Since $\tilde{P}, \tilde{Q}$ are coprime, their resultant (see, e.g.,  \cite[Chapter 3]{CLO05} for its precise definition) with respect to $v$ is not identically zero:
\[
R(u) = \operatorname{Res}_v(\tilde{P}, \tilde{Q}) \in \mathbb{Q}[u], \quad R(u) \not\equiv 0.
\]
According to the properties of the resultant (see, e.g., \cite[Chapter 3]{CLO05}), 
\[
\tilde{P}(u_0, v_0) = \tilde{Q}(u_0, v_0) = 0 \text{\,\, and \,\,}
R(u_0) = 0.
\]
Since $R\in \mathbb{Q}[u]$ is a non-zero polynomial, its root $u_0$ must be an algebraic number. By symmetry (computing $\operatorname{Res}_u(\tilde{P}, \tilde{Q})$), $v_0$ is also an algebraic number.

\noindent
\textbf{Case 2}. 
Assume that $(u_0, v_0)$ is a root of the common factor $H(u, v) = 0$. 

By assumption,  $(u_0, v_0)$ is an isolated point of the set $\{(u, v): H(u, v) = 0\}$ in $\mathbb{R}^2$. The implicit function theorem implies that $(u_0, v_0)$ must be a singular point of the curve $H(u,v)=0$:
\begin{equation}
        H(u_0, v_0) = 0 ,\quad 
        \frac{\partial H}{\partial u}(u_0, v_0) = 0,\quad      \frac{\partial H}{\partial v}(u_0, v_0) = 0.
\end{equation}
Note that since $H \in \mathbb{Z}[u, v]$, the partial derivatives $H_u = \frac{\partial H}{\partial u}$ and $H_v = \frac{\partial H}{\partial v}$ are also in $\mathbb{Z}[u, v]$. 

Now we consider the new system  $H = H_u = 0$. Let $L(u, v) = \gcd(H, H_u)$.
\begin{itemize}
    \item If $H$ and $H_u$ are coprime, we return to \textbf{Case 1}, and  $(u_0,v_0)$ are both algebraic.
    \item Otherwise,  $H$ has repeated factors. We can divide $H$ by $L$ repeatedly until we obtain a polynomial $\widehat{H}(u, v)$  which shares the same zero set but has no repeated factors. The point $(u_0, v_0)$ remains an isolated zero of $\widehat{H} = 0$, and thus must be a singular point of $\widehat{H}$. Then we are reduced to the coprime case.
\end{itemize}
Therefore, $u_0$ and $v_0$ are algebraic numbers.
\end{proof}

\begin{proof}[Proof of Corollary~\ref{coro-algebraic}]
Assume 
$
p=m/n, q=j/k$ with $m,n,j,k\in\mathbb{N}$. By Theorem~\ref{thm-main-thm-one-xy},  using the formula \eqref{cross-rpq} for $r_{p,q}(\mathbb{Z}_3)$, it suffices to show that all coordinates of  the unique solution $(x,y)$ in $(0,1)^2$ to \eqref{xy-eq} are algebraic.  It can be easily checked that, the pair $(u,v)=(x^{1/n}, y^{1/k})$ satisfies a system of the form \eqref{PQ-sys} admitting  a unique solution in $(0,1)^2$.   Therefore, by $(x,y)  =(u^n, v^k)$, Corollary \ref{coro-algebraic} follows immediately from Lemma~\ref{lemma-daishushu}.
\end{proof}

{\flushleft \bf Declaratioins: Competing Interests}  The authors declare no competing interests.

\bibliographystyle{alpha}
%\bibliography{bib-Z3}

\begin{thebibliography}{99}
\bibitem[ADFS04]{ADFS-04}
Noga Alon, Irit Dinur, Ehud Friedgut, and Benny Sudakov.
\newblock Graph products, {F}ourier analysis and spectral techniques.
\newblock {\em Geom. Funct. Anal.}, 14(5):913--940, 2004.

\bibitem[And99]{And99}
Mats~Erik Andersson.
\newblock Remarks on hypercontractivity for the smallest groups.
\newblock In {\em Complex analysis and differential equations ({U}ppsala,
  1997)}, volume~64 of {\em Acta Univ. Upsaliensis Skr. Uppsala Univ. C Organ.
  Hist.}, pages 51--60. Uppsala Univ., Uppsala, 1999.

\bibitem[And02]{Andersson}
Mats~Erik Andersson.
\newblock Beitrag zur {T}heorie des {P}oissonschen {I}ntegrals \"{u}ber
  endlichen {G}ruppen.
\newblock {\em Monatsh. Math.}, 134(3):177--190, 2002.

\bibitem[AGV85]{Arnold-Sing}
Vladimir Arnold, Sabir Gusein-Zade, and Alexander Varchenko.
\newblock {\em Singularities of differentiable maps. {V}ol. {I}}, volume~82 of
  Monogr. Math.,
\newblock Birkh\"{a}user Boston, Inc., Boston, MA, 1985.
\newblock The classification of critical points, caustics and wave fronts,
  Translated from the Russian by Ian Porteous and Mark Reynolds.

\bibitem[Bea95]{Bea-95}
Alan Beardon.  
\newblock {\em The geometry of discrete groups}. 
Grad. Texts Math., vol. 91. SpringerVerlag, New York, 1995. 

\bibitem[Bec75]{Beckner}
William Beckner.
\newblock Inequalities in {F}ourier analysis.
\newblock {\em Ann. of Math. (2)}, 102(1):159--182, 1975.

\bibitem[BJJ83]{BJJ83}
William Beckner, Svante Janson, and David Jerison.
\newblock Convolution inequalities on the circle.
\newblock In {\em Conference on harmonic analysis in honor of {A}ntoni
  {Z}ygmund, {V}ol. {I}, {II} ({C}hicago, {I}ll., 1981)}, Wadsworth Math. Ser.,
  pages 32--43. Wadsworth, Belmont, CA, 1983.



\bibitem[BKKKL92]{BKKKL}
Jean Bourgain, Jeff Kahn, Gil Kalai, Yitzhak Katznelson, and Nathan Linial.
\newblock The influence of variables in product spaces.
\newblock {\em Isr. J. Math.}, 77(1-2):55--64, 1992.

\bibitem[Bon70]{Bonami}
Aline Bonami.
\newblock \'{E}tude des coefficients de {F}ourier des fonctions de
  {$L^{p}(G)$}.
\newblock {\em Ann. Inst. Fourier (Grenoble)}, 20(fasc. 2):335--402 (1971),
  1970.

\bibitem[Bor82]{Borell-82}
Christer Borell.
\newblock Positivity improving operators and hypercontractivity.
\newblock {\em Math. Z.}, 180:225--234, 1982.

\bibitem[CLO05]{CLO05}
David Cox, John Little, and Donal O'Shea.
\newblock{\em Using algebraic geometry}.
\newblock 2nd ed., Grad. Texts Math., 185, Springer, 2005.

\bibitem[CLSC08]{CLS08}
Guan-Yu Chen, Wai-Wai Liu, and Laurent Saloff-Coste.
\newblock The logarithmic {S}obolev constant of some finite {M}arkov chains.
\newblock {\em Ann. Fac. Sci. Toulouse Math. (6)}, 17(2):239--290, 2008.

\bibitem[DFR08]{DFR-08}
Irit Dinur, Ehud Friedgut, and Oded Regev.
\newblock Independent sets in graph powers are almost contained in juntas.
\newblock {\em Geom. Funct. Anal.}, 18(1):77--97, 2008.

\bibitem[DSC96]{DS96}
Persi Diaconis and Laurent Saloff-Coste.
\newblock Logarithmic {S}obolev inequalities for finite {M}arkov chains.
\newblock {\em Ann. Appl. Probab.}, 6(3):695--750, 1996.

\bibitem[EGSC18]{EGS-18}
Nathaniel Eldredge, Leonard Gross, and Laurent Saloff-Coste.
\newblock Strong hypercontractivity and logarithmic {S}obolev inequalities on
  stratified complex {L}ie groups.
\newblock {\em Trans. Am. Math. Soc.}, 370(9):6651--6683, 2018.

\bibitem[FI21]{FI-21}
Rupert~L. Frank and Paata Ivanisvili.
\newblock Hypercontractivity of the semigroup of the fractional {L}aplacian on
  the {$n$}-sphere.
\newblock {\em J. Funct. Anal.}, 281(8):Paper No. 109145, 10, 2021.

\bibitem[FF14]{FF-14}
Dvir Falik and Ehud Friedgut.
\newblock Between {A}rrow and {G}ibbard-{S}atterthwaite; a representation
  theoretic approach.
\newblock {\em Isr. J. Math.}, 201(1):247--297, 2014.

\bibitem[FK96]{FK-96}
Ehud Friedgut and Gil Kalai.
\newblock Every monotone graph property has a sharp threshold.
\newblock {\em Proc. Am. Math. Soc.}, 124(10):2993--3002, 1996.

\bibitem[Fri98]{Fri-98}
Ehud Friedgut.
\newblock Boolean functions with low average sensitivity depend on few
  coordinates.
\newblock {\em Combinatorica}, 18(1):27--35, 1998.

\bibitem[Fri23]{Fri-23}
Ehud Friedgut.
\newblock K{KL}'s influence on me.
\newblock In {\em I{CM}---{I}nternational {C}ongress of {M}athematicians.
  {V}ol. 6. {S}ections 12--14}, pages 4568--4581. EMS Press, Berlin, 2023.

  
\bibitem[Gro75]{Groos}
Leonard Gross.
\newblock Logarithmic {S}obolev inequalities.
\newblock {\em Am. J. Math.}, 97(4):1061--1083, 1975.


\bibitem[Gro99]{Gro-99}
Leonard Gross.
\newblock Hypercontractivity over complex manifolds.
\newblock {\em Acta Math.}, 182(2):159--206, 1999.

\bibitem[Gro02]{Gro-02}
Leonard Gross.
\newblock Strong hypercontractivity and relative subharmonicity.
\newblock {\em J. Funct. Anal.}, v190, 38--92. 2002.

\bibitem[HHM18]{HHM-18}
Jan H{\c a}z{\l}a, Thomas Holenstein, and Elchanan Mossel.
\newblock Product space models of correlation: between noise stability and
  additive combinatorics.
\newblock {\em Discrete Anal.}, Paper No. 20, 63, 2018.

\bibitem[IN22]{IN-22}
Paata Ivanisvili and Fedor Nazarov.
\newblock On {Weissler}'s conjecture on the {Hamming} cube. {I}.
\newblock {\em Int. Math. Res. Not.}, 2022(9):6991--7020, 2022.

\bibitem[Jan83]{Jan-83}
Svante Janson.
\newblock On hypercontractivity for multipliers on orthogonal polynomials.
\newblock {\em Ark. Mat.}, 21(1):97--110, 1983.

\bibitem[JPPP17]{JPPP}
Marius Junge, Carlos Palazuelos, Javier Parcet, and Mathilde Perrin.
\newblock Hypercontractivity in group von {N}eumann algebras.
\newblock {\em Mem. Am. Math. Soc.}, 249(1183):xii+83, 2017.

\bibitem[JPPR15]{JPPE}
Marius Junge, Carlos Palazuelos, Javier Parcet, and Eric Ricard.
\newblock Hypercontractivity for free products.
\newblock {\em Ann. Sci. \'{E}c. Norm. Sup\'{e}r. (4)}, 48(4):861--889, 2015.

\bibitem[KLo18]{KL-18}
Peter Keevash and Eoin Long.
\newblock A stability result for the cube edge isoperimetric inequality.
\newblock {\em J. Comb. Theory, Ser. A}, 155:360--375, 2018.

\bibitem[KLi18]{KLi-18}
Nathan Keller and Noam Lifshitz.
\newblock Approximation of biased {B}oolean functions of small total influence
  by {DNF}s.
\newblock {\em Bull. Lond. Math. Soc.}, 50(4):667--679, 2018.

\bibitem[KLLM24]{KLLM-24}
Peter Keevash, Noam Lifshitz, Eoin Long, and Dor Minzer.
\newblock Hypercontractivity for global functions and sharp thresholds.
\newblock {\em J. Am. Math. Soc.}, 37(1):245--279, 2024.

\bibitem[KLM26]{KLM-26}
Nathan Keller, Noam Lifshitz, and Omri Marcus.
\newblock Sharp hypercontractivity for global functions.
\newblock {\em J. Eur. Math. Soc.}, 2026.
\newblock Online first.

\bibitem[KS88]{KS-88}
Wieslaw Krakowiak and Jerzy Szulga.
\newblock Hypercontraction principle and random multilinear forms.
\newblock {\em Probab. Theory Relat. Fields}, 77(3):325--342, 1988.

\bibitem[KS91]{KS-91}
Stanis{\l}aw Kwapie{\'n} and Jerzy Szulga.
\newblock Hypercontraction methods in moment inequalities for series of
  independent random variables in normed spaces.
\newblock {\em Ann. Probab.}, 19(1):369--379, 1991.

\bibitem[Kul22]{Kui-22}
Aleksei Kulikov.
\newblock Functionals with extrema at reproducing kernels.
\newblock {\em Geom. Funct. Anal.}, 32(4):938--949, 2022.

\bibitem[LaO00]{LaO00}
Rafa{\l} Lata\l{a} and Krzysztof Oleszkiewicz.
\newblock Between {S}obolev and {P}oincar\'{e}.
\newblock In {\em Geometric aspects of functional analysis}, volume 1745 of
  {\em Lect. Notes Math.}, pages 147--168. Springer, Berlin, 2000.

\bibitem[Mel23]{Mel-23}
Petar Melentijevi\'{c}.
\newblock Hypercontractive inequalities for weighted {B}ergman spaces.
\newblock {\em Bull. Lond. Math. Soc.}, 55(6):2611--2616, 2023.

\bibitem[Mos10]{Mos-10}
Elchanan Mossel.
\newblock Gaussian bounds for noise correlation of functions.
\newblock {\em Geom. Funct. Anal.}, 19(6):1713--1756, 2010.

\bibitem[MOO10]{MDO-10}
Elchanan Mossel, Ryan O'Donnell, and Krzysztof Oleszkiewicz.
\newblock Noise stability of functions with low influences: invariance and
  optimality.
\newblock {\em Ann. of Math. (2)}, 171(1):295--341, 2010.

\bibitem[MOS13]{MOS-13}
Elchanan Mossel, Krzysztof Oleszkiewicz, and Arnab Sen.
\newblock On reverse hypercontractivity.
\newblock {\em Geom. Funct. Anal.}, 23(3):1062--1097, 2013.

\bibitem[Nel73]{Nelson-73}
Edward Nelson.
\newblock The free {M}arkoff field.
\newblock {\em J. Funct. Anal.}, 12:211--227, 1973.

\bibitem[O'D14]{Donnell}
Ryan O'Donnell.
\newblock {\em Analysis of {B}oolean functions}.
\newblock Cambridge University Press, New York, 2014.

\bibitem[Ole03]{Ole03}
Krzysztof Oleszkiewicz.
\newblock On a nonsymmetric version of the {K}hinchine-{K}ahane inequality.
\newblock In {\em Stochastic inequalities and applications}, volume~56 of {\em
  Progr. Probab.}, pages 157--168. Birkh\"{a}user, Basel, 2003.

\bibitem[SVZ24]{SVZ24}
Joseph Slote, Alexander Volberg, and Haonan Zhang.
\newblock Bohnenblust-{H}ille inequality for cyclic groups.
\newblock {\em Adv. Math.}, 452:Paper No. 109824, 35, 2024.

\bibitem[Tal94]{Tala-94}
Michel Talagrand.
\newblock On {Russo}'s approximate zero-one law.
\newblock {\em Ann. Probab.}, 22(3):1576--1587, 1994.

\bibitem[Wei80]{Wei-90}
Fred~B. Weissler.
\newblock Logarithmic {S}obolev inequalities and hypercontractive estimates on
  the circle.
\newblock {\em J. Funct. Anal.}, 37(2):218--234, 1980.

\bibitem[Wol07]{Wolff}
{Pawe\l} Wolff.
\newblock Hypercontractivity of simple random variables.
\newblock {\em Stud. Math.}, 180(3):219--236, 2007.

\bibitem[Yao25]{yao2024optimal}
Gan Yao.
\newblock Optimal hypercontractivity and {Log--Sobolev} inequalities on cyclic
  groups $\mathbb{Z}_{m\cdot 2^k}$.
\newblock {\em arXiv preprint}, December 2025.

\bibitem[Zha21]{Zhao2021}
Yu~Zhao.
\newblock {\em Generalizations and Applications of Hypercontractivity and
  Small-Set Expansion}.
\newblock Ph.d. dissertation, Columbia University, 2021.
\end{thebibliography}

\end{document}